\RequirePackage[T1]{fontenc}
\documentclass[12pt,twoside]{amsart}
\calclayout

\usepackage[whole]{bxcjkjatype}

\usepackage[utf8]{inputenc}
\usepackage{amsmath}
\usepackage{amssymb}
\usepackage{mathtools}
\numberwithin{equation}{section}
\usepackage{slashed}
\usepackage{braket}
\usepackage[svgnames]{xcolor}
\usepackage{colortbl}
\definecolor{lightgray}{gray}{0.9}
\definecolor{xgray}{gray}{0.8}
\colorlet{conjcolor}{blue!5!white}
\usepackage[colorlinks,citecolor=DarkGreen,linkcolor=FireBrick,urlcolor=FireBrick,linktocpage,breaklinks=true]{hyperref}
\usepackage{cite}
\usepackage[pdftex]{graphicx}
\usepackage{tikz-cd}
\usepackage{times}
\usepackage{courier}
\usepackage{bm}
\usepackage{subfig}
\usepackage{fnpct}

\usepackage{mathrsfs} 
\usepackage{tablefootnote}

\usepackage[all]{xy}
\usepackage{enumitem}
\usepackage{ulem}

 \usepackage{stmaryrd}

\usepackage{float}

\usepackage{arydshln}

\if0
\usepackage{mdframed}

\renewenvironment{figure}[1][]{
  \begin{originalfigure}[#1]
    \begin{mdframed}[linecolor=black!0,backgroundcolor=black!1]
}{
    \end{mdframed}
  \end{originalfigure}
}
\fi

\newcommand{\ie}{\begin{equation}\begin{aligned}}
\newcommand{\fe}{\end{aligned}\end{equation}}
\newcommand{\tw}{\mathrm{tw}}

\newcommand{\cF}{\mathcal{F}}
\newcommand{\cD}{\mathcal{D}}
\usepackage[makeroom]{cancel}
\usepackage{multirow}
\newcommand{\coloredbox}[2]{
    \colorbox{#1}{\parbox[t]{\linewidth}{#2}}
}

\usepackage{xparse}%  For \NewDocumentCommand
\usepackage{calc}%    For the \widthof macro
\usepackage{tikz}
\usetikzlibrary{calc}

\newcommand{\tikzmark}[1]{\tikz[overlay,remember picture] \node (#1) {};}

\makeatletter

\NewDocumentCommand{\DrawBoxWide}{s O{}}{%
    \tikz[overlay,remember picture]{
    \IfBooleanTF{#1}{%
        \coordinate (RightPoint) at ($(left |- right)+(\linewidth-\labelsep-\labelwidth,0.0)$);
    }{%
        \coordinate (RightPoint) at (right.east);
    }%
    \draw[red,#2]
      ($(left)+(-\labelwidth,0.9em)$) rectangle
      ($(RightPoint)+(0.2em,-0.3em)$);}
}

\makeatother

\theoremstyle{plain}
\newtheorem{defn}[equation]{Definition}

\newtheorem{thm}[equation]{Theorem}

\newtheorem{prop}[equation]{Proposition}

\newtheorem{fact}[equation]{Fact}

\def\PhysicsFact{Physics Assumption}

\newtheorem{fact?}[equation]{Fact?}
\newtheorem{cor}[equation]{Corollary}

\newtheorem{lem}[equation]{Lemma}

\newtheorem{conjc}[equation]{Conjecture}\newtheorem{claim}[equation]{Claim}
\newtheorem{claim?}[equation]{Claim?}

\theoremstyle{remark}
\newtheorem{rem}[equation]{Remark}
\newtheorem{ex}[equation]{Example}

\newtheorem{notation}[equation]{Notation}

\usepackage[many]{tcolorbox}

\tcbset{
  breakable,
  colback=conjcolor,
  boxrule=-1pt,
  boxsep=1pt,
  left=2pt,right=2pt,top=2pt,bottom=2pt,
  oversize=2pt,
  sharp corners,
  before skip=\topsep,
  after skip=\topsep,
}
\tcolorboxenvironment{proofc}{}
\tcolorboxenvironment{defnc}{}
\tcolorboxenvironment{propc}{}
\tcolorboxenvironment{thmc}{}
\tcolorboxenvironment{corc}{}
\tcolorboxenvironment{conjc}{}

\let\oldendrem\endrem
\def\endrem{\hfill {$\lrcorner$} \oldendrem}

\def\bC{\mathbb{C}}
\def\bH{\mathbb{H}}

\def\cS{\mathcal{S}}

\def\Hom{\mathrm{Hom}}

\def\tr{\mathop{\mathrm{tr}}}

\def\KO{\mathrm{KO}}
\def\KU{\mathrm{KU}}
\def\TJF{\mathrm{TJF}}

\def\MF{\mathrm{MF}}
\def\TMF{\mathrm{TMF}}

\def\tmf{\mathrm{tmf}}
\def\mf{\mathrm{mf}}
\def\pt{\mathrm{pt}}

\def\BO{B\mathrm{O}}
\def\Z{\mathbb{Z}}
\def\MString{\mathrm{MString}}

\def\Q{\mathbb{Q}}
\def\R{\mathbb{R}}
\def\ch{e}

\def\SU{\mathrm{SU}}

\def\cA{\mathcal{A}}

\def\stri{\text{string}}

\def\id{\mathrm{id}}

\def\fr{\mathrm{fr}}

\def\ABS{\mathrm{ABS}}
\def\Ind{\mathrm{Ind}}
\def\Ch{\mathrm{Ch}}

\def\fgt{\mathrm{fgt}}
\def\tr{\mathrm{tr}}
\def\Sp{\mathrm{Sp}}

\def\CP{\mathbb{CP}}

\def\moduli{\mathcal{M}}

\def\res{\mathrm{res}}
\def\Mod{\mathrm{Mod}}
\def\RO{\mathrm{RO}}
\def\Map{\mathrm{Map}}
\def\colim{\mathrm{colim}}
\def\op{\mathrm{op}}

\def\CAlg{\mathrm{CAlg}}
\def\Nm{\mathrm{Nm}}

\def\QCoh{\mathrm{QCoh}}
\def\HP{\mathbb{HP}}
\def\Jac{\mathrm{Jac}}
\def\JF{\mathrm{JF}}
\def\Pic{\mathrm{Pic}}
\def\C{\mathbb{C}}

\def\RO{\mathbf{R}\mathrm{O}}
\def\Spectra{\mathrm{Spectra}}
\def\Ad{\mathrm{Ad}}
\def\Rep{\mathrm{Rep}}

\def\coev{\mathrm{coev}}

\def\triv{\mathrm{triv}}
\def\rank{\mathrm{rank}}
\def\Ell{\mathrm{Ell}}
\def\cEll{\mathcal{E}ll}

\def\EvJF{\mathrm{EJF}}

\def\TEJF{\mathrm{TEJF}}
\def\EJF{\mathrm{EJF}}

\def\jF{\mathrm{jF}}
\def\clas{\text{clas}}

\def\cH{\mathcal{H}}
\def\cS{\mathcal{S}}

\def\sX{\mathsf{X}}

\def\cG{\mathcal{G}}
\def\fraks{\mathfrak{s}}
\def\frako{\mathfrak{o}}
\def\bT{\mathbb{T}}

\def\cL{\mathcal{L}}
\def\cE{\mathcal{E}}
\def\stab{\mathrm{stab}}

\def\ev{\mathrm{ev}}
\def\Res{\mathrm{Res}}

\def\cB{\mathcal{B}}
\def\PT{\mathrm{PT}}
\def\coll{\mathrm{coll}}
\def\lMap{\underline{\Map}}
\def\ind{\mathrm{ind}}

\def\Orb{\mathrm{Orb}}
\def\univ{\mathrm{uinv}}
\def\CHD{\mathrm{CHD}}
\def\frK{\mathfrak{K}}
\def\Eul{\mathrm{Euler}}
\def\cO{\mathcal{O}}
\def\cofib{\mathrm{cofib}}
\def\tw{\mathrm{tw}}
\def\red{\mathrm{red}}
\def\frS{\mathfrak{S}}
\def\cH{\mathcal{H}}
\def\cK{\mathcal{K}}

\def\frC{\mathfrak{C}}
\def\cptLie{\mathrm{cptLie}}
\def\Shv{\mathrm{Shv}}
\def\bfB{\mathbf{B}}
\def\th{\mathrm{th}} 
\def\Th{\mathbf{th}}
\def\frc{\mathfrak{c}}
\def\frb{\mathfrak{b}}
\def\Pic{\mathrm{Pic}}
\def\JacKO{\Jac^{\KO}}
\def\Gm{\mathbb{G}_m}
\def\RString{\mathbf{R}\mathrm{String}}
\def\Gpds{\mathrm{Gpds}}

\let\oldtext\text
\def\text#1{\oldtext{\upshape\mdseries #1}}

\DeclareSymbolFont{yhlargesymbols}{OMX}{yhex}{m}{n} \DeclareMathAccent{\reallywidehat}{\mathord}{yhlargesymbols}{"62}

\DeclareFontFamily{U}{mathx}{}
\DeclareFontShape{U}{mathx}{m}{n}{<-> mathx10}{}
\DeclareSymbolFont{mathx}{U}{mathx}{m}{n}
\DeclareMathAccent{\widehat}{0}{mathx}{"70}
\DeclareMathAccent{\widecheck}{0}{mathx}{"71}

\let\hat\widehat

\let\oldwidehat\widehat
\def\widehat#1{\oldwidehat{#1}{}}

\def\paragraph#1{

\medskip\noindent \textit{#1} --- }

\tikzset{circ/.style = {fill, circle, inner sep = 0, minimum size = 3}}

\newcounter{eta}
\newcounter{nu}
\newenvironment{celldiagram}{
  \setcounter{eta}{0}
  \setcounter{nu}{0}
  \def\n##1{\node [circ] at (0, ##1) {};}
  \def\invert##1{\setcounter{##1}{1 - \the\value{##1}}}
  \def\getside##1{\ifnum\the\value{##1}=0 right\else left\fi}
  \def\two##1{\draw [blue, thick] (0, ##1) to (0, ##1 + 1);}
  \def\eta##1{\draw [red, thick] (0, ##1) to  (0, ##1 + 2);\invert{eta}}
  \def\nu##1{\draw [blue, thick] (0, ##1) to [bend \getside{nu}=50] (0, ##1 + 4);\invert{nu}}
  \def\eightnu##1{\draw [orange, thick] (0, ##1) to [bend \getside{nu}=50] (0, ##1 + 4);\invert{nu}}
}{}

\newcommand{\highlight}[2][conjcolor]{\mathchoice%
	{\colorbox{#1}{$\displaystyle#2$}}%
	{\colorbox{#1}{$\textstyle#2$}}%
	{\colorbox{#1}{$\scriptstyle#2$}}%
	{\colorbox{#1}{$\scriptscriptstyle#2$}}}%

\begin{document}

\title[Topological Elliptic Genera I]{Topological Elliptic Genera I---The mathematical foundation}
\author{Ying-Hsuan Lin}
\author{Mayuko Yamashita}
\date{}
\address{Jefferson Physical Laboratory, Harvard University, Cambridge, MA 02138, USA}
\email{yhlin@alum.mit.edu}
\address{Perimeter Institute for Theoretical Physics, 31 Caroline Street North,
Waterloo, Ontario, Canada, N2L 2Y5
	}
	\email{myamashita@perimeterinstitute.ca}
\def\funding{
	The work of MY is supported by Grant-in-Aid for JSPS KAKENHI Grant Number 20K14307, JST CREST program JPMJCR18T6, and the Simons Collaborations on Global Categorical Symmetries. She would like to thank the Center of Mathematical Sciences and Applications, Harvard University, for support and hospitality, where a large part of the work on this paper was undertaken. YL was supported by the Simons Collaboration Grant on the Non-Perturbative Bootstrap for a portion of this work. This research was supported in part by grant NSF PHY-1748958 to the Kavli Institute for Theoretical Physics (KITP).
	}
\thanks{{\it Acknowledgments}: 
	The authors thank Akira Tominaga for the earlier collaboration on this project, Lennart Meier for the collaboration especially on the equivariant sigma orientation, and Du Pei for suggesting the relation with level-rank dualities. They are grateful to Tilman Bauer, David Gepner, and L.~Meier for sharing their ongoing work. They also thank Michael Hopkins, Theo Jonhson-Freyd, Kaoru Ono, Tatsuki Kuwagaki, and Thomas Schick for helpful discussions. 
	\funding
}

\begin{abstract}
	We construct {\it Topological Elliptic Genera}, homotopy-theoretic refinements of the elliptic genera for $SU$-manifolds and variants including the Witten-Landweber-Ochanine genus. The codomains are genuinely $G$-equivariant Topological Modular Forms developed by Gepner-Meier \cite{GepnerMeier}, twisted by $G$-representations.
	As the first installment of a series of articles on Topological Elliptic Genera, this issue lays the mathematical foundation and discusses immediate applications. Most notably, we deduce an interesting divisibility result for the Euler numbers of $Sp$-manifolds. 
\end{abstract}

\maketitle

\tableofcontents

\section{Introduction}

There is a classical construction of the {\it elliptic genus} for $SU$-manifolds, (e.g., \cite{Witten:1987cg} and \cite{Gritsenko:1999fk}),
which assigns, for each tangential $SU(k)$-manifold $M$ with dimension $\dim_\R M = 2k$, 
\begin{align}\label{eq_intro_Jac_clas}
	\Jac_{\clas}(M) \in \left\{ \mbox{integral Jacobi Forms with weight}=0, \mbox{ index}=k/2 \right\}. 
\end{align}
The formula is given by (see Section \ref{subsec_notations} (\ref{notation_JF}) for the convention of Jacobi Forms)
\begin{align}\label{eq_char_Jac_clas}
	\Jac_{\clas} (M) (y, q) 
		=  y^{k/2} \cdot \int_M \mathrm{Todd}(TM) \wedge \Ch\left( \mathbb{TM}_{q, y}\right) ,
	\end{align}
	where (in the formula below all the tensor/exterior products are over $\C$, )
	\begin{align}
		\mathbb{TM}_{q, y} := \bigotimes_{m \ge 0} \wedge_{-q^{m}y^{-1}}T^*M \otimes \bigotimes_{m \ge 1} \wedge_{-q^{m}y} TM \otimes  \bigotimes_{m \ge 1} \mathrm{Sym}_{q^{m}}T^*M \otimes \bigotimes_{m \ge 1} \mathrm{Sym}_{q^{m}} TM . 
	\end{align}
	For example, in the notation of \cite{Gritsenko:1999fk},
	  \begin{align}
		\Jac_\clas ([K3]) = {2\phi_{0, 1}}, \quad 
		\Jac_\clas ([\mathrm{CY}_3]) = (h^{1,1} - h^{1,2}) \cdot \phi_{0, \frac32},
	\end{align}
	where $\mathrm{CY}_3$ is any Calabi-Yau threefold with Hodge numbers $h^{1,1}$ and $h^{1,2}$. 
	Related constructions include the {\it level-$N$ genera} which produce modular forms with level structures; most notably, the case of $N=2$ is called the {\it Witten-Landweber-Ochanine genus} for spin manifolds \cite{Ochanine}.

The main construction of this paper concerns the {\it topological}, (or {\it spectral}) refinements\footnote{
What we mean by {\it topological refinements} here is analogous to how the (homotopy-theoretic) sigma orientation \cite{ando2010multiplicative} refines the (classical) Witten genera \cite{Witten:1987cg}. The following diagrams illustrate the concept. 
\begin{align}
	\xymatrix{
		MString \ar[rr]^-{\sigma} \ar@{~>}[d]^-{\mbox{\tiny refine} }& \ar@{}[d]|-{\rotatebox[origin=c]{-90}{$\rightsquigarrow$}}^-{\mbox{\tiny refine}}& \TMF  \ar@{~>}[d]^-{\mbox{\tiny refine} }\\
\Omega^{\stri} \ar[rr]^-{\mbox{\tiny Wit}} && \MF
	} \quad 
	\xymatrix{
		MTSU(k) \ar[rr]^-{\Jac_{U(1)_k}} \ar@{~>}[d]^-{\mbox{\tiny refine} } & \ar@{}[d]|-{\rotatebox[origin=c]{-90}{$\rightsquigarrow$}}^-{\mbox{\tiny refine}}& \TJF_k  \ar@{~>}[d]^-{\mbox{\tiny refine} }\\
	\Omega^{SU(k)} \ar[rr]^-{\Jac_\clas} && \JF_k
}
\end{align}
The left square is about the sigma orientation and the right square is about our topological elliptic genera. 
The bottom row consists of {\it classical} objects, namely maps between abelian groups, whereas the top row consists of {\it homotopy-theoretical} objects, namely morphism between spectra. The upper row refines the lower one.

We also remark that the classical notion of elliptic cohomology (i.e., a complex-oriented theory whose associated formal group law comes from an elliptic curve) can also be regarded as ``topological refinements'' of the classical elliptic genera \cite{TMFbook}. From that point of view, what we construct here can be regarded as the {\it universal} topological refinement. 
} of those classical numerical genera; for this reason, we call them the {\it topological elliptic genera}.
This relies heavily on the recent developments \cite{GepnerMeier, LurieElliptic3} in the {\it genuinely} equivariant refinements of the spectrum of {\it Topological Modular Forms}, or $\TMF$\footnote{Our construction is closely related to the work by Ando-French-Ganter \cite{Ando_2008}. As explained in Section \ref{subsec_vs_JacOri}, the constructions in this paper are regarded as {\it genuine} and {\it unstable} versions of their construction. }.

Our exemplary case is the refinement of the classical elliptic genus \eqref{eq_intro_Jac_clas}. We define a morphism of spectra which we call the {\it $U(1)$-topological elliptic genus}, 
\begin{align}\label{eq_intro_Jac_U}
	\Jac_{U(1)_k} &\colon	MTSU(k) \to \TJF_k,
\end{align}
for each nonnegative integer $k$, where the $U(1)$-equivariance can be understood as arising from the complex structure of $SU$-manifolds.  Here, 
\begin{itemize}
	\item $MTSU(k)$ is the bordism spectrum of {\it tangential} $SU(k)$-manifolds. See Section \ref{subsec_thom} for the explanation. 
	\item $\TJF_k$ is a $\TMF$-module spectrum called {\it Topological Jacobi Forms}, realized as genuinely $U(1)$-equivariant twisted $\TMF$. It can be regarded naturally as the topological refinement of Jacobi Forms with index $\frac{k}{2}$ and is investigated in an upcoming paper by Bauer-Meier \cite{BauerMeierTJF}. We collect the necessary facts in Appendix \ref{app_TJF} as a user guide. 
\end{itemize}

The spectrum $\TJF_k$, being a refinement of the module of Jacobi Forms, comes equipped with a map
\begin{align}\label{eq_e_JF_intro}
	e_\JF \colon	\pi_m\TJF_k \to  \left\{ \mbox{integral Jacobi Forms with weight}=m/2-k, \mbox{ index} =k/2\right\} =: \JF_k|_{\deg = m}. 
\end{align}
(here $\JF_k$ is the $\Z$-graded module of integral Jacobi Forms with index $\frac{k}{2}$, whose degree convention is explained in Section \ref{subsec_notations} \eqref{notation_JF}), 
and the topological elliptic genus $\Jac_{U(1)_k}$ refines the classical elliptic genus $\Jac_{\rm clas}$ in the sense that, when applied to the case $m=2k$, we have
\begin{align}\label{eq_clas_vs_top_intro}
	\Jac_\clas(M) = e_\JF \circ \Jac_{U(1)_k}(M) . 
\end{align}

Why do we care about such a topological refinement? Indeed, the refinement gives us nontrivial information that cannot be obtained from the numerical elliptic genus, as follows. 
\begin{enumerate}
	\item The map $e_\JF$ is not injective for general $k$. The kernel consists of torsion elements which are invisible as classical Jacobi Forms.
	For example, we have
	\begin{align}
		\pi_5 \TJF_2 \simeq \Z/2, \quad \pi_7 \TJF_2 = \Z/2. 
	\end{align}
	Accordingly, our topological elliptic genus can detect torsion elements in $\pi_* MTSU(k)$. 
	\item The map $e_\JF$ is not surjective, although it is rationally equivalent. This means that we get divisibility constraints for Jacobi Forms inside the image of $e_{\JF}$. For example, 
	\begin{align}
		\mathrm{Coker}\left( e_\JF \colon \pi_4\TJF_2 \to \JF_2|_{\deg = 4}\right) = \Z\phi_{0,1}/(2\phi_{0,1}),
	\end{align}
    meaning that half of the $K3$ elliptic genus is not in the image (see Remark~\ref{rem_enriques} for further comments). 
	The general non-surjectivity of $e_\JF$, combined with \eqref{eq_clas_vs_top_intro}, implies nontrivial divisibility constraints on the classical elliptic genus and consequently on various characteristic numbers for tangential $SU(k)$-manifolds. We investigate this in Section~\ref{subsec_divisibility}. 
	\item Our topological elliptic genus is {\it unstable}, in the sense that the codomain depends on $k$. 
	The relations among different $k$ are captured by the commutative diagram
	\begin{align}
		\xymatrix{
			MTSU(k)\ar[rr]^-{\Jac_{U(1)_k}}\ar[d]_-{SU(k) \hookrightarrow SU(k+1)} && \TJF_k \ar[d]^-{\stab} \ar@{~>}[r]^-{\pi_m}& \pi_m \TJF_k \ar[d]^-{\stab} \ar[r]^-{e_\JF}& \JF_k|_{\deg=m} \ar@{_{(}->}[d]^-{\phi_{-1, \frac12} \cdot} \\
			MTSU(k+1) \ar[rr]^-{\Jac_{U(1)_{k+1}}}&& \TJF_{k+1} \ar@{~>}[r]^-{\pi_m} & \pi_m \TJF_{k+1} \ar[r]^-{e_\JF} &\JF_{k+1}|_{\deg=m}
		}
	\end{align}
	Interestingly, the third vertical arrow is neither injective nor surjective in general. Rather, it is part of a long exact sequence. 
	This is in contrast to the rightmost vertical arrow, which is injective. 
	Thus, our topological elliptic genus can detect nontrivial $\pi_m MTSU(k)$ elements that vanish in $\pi_m MTSU(\infty) = \pi_m MSU$ and cannot be detected by $\Jac_\clas$. 
	Such an example is explained in Section \ref{subsec_pi5MSp}. 
	\end{enumerate}

The above $U(1)$-topological elliptic genus is just a special case of a more general construction we study in this paper. The most general construction is in Section \ref{subsec_construction}. 
Under the settings listed there, we construct a class of morphisms from certain Thom spectra to $RO(G)$-graded genuinely equivariant $\TMF$, which we generally call {\it topological elliptic genera}. 
Besides $U(1)$, other key cases include\footnote{see the last paragraph of Introduction}
\begin{align}
&	\Jac_{Sp(1)_k} \colon MTSp(k) \to \TMF[kV_{Sp(1)}]^{Sp(1)} =: \TEJF_{2k}, \label{eq_intro_Jac_Sp} \\
&\highlight{\Jac_{O(1)_k} \colon MTSpin(k) \to \TMF[kV_{O(1)}]^{O(1)}. } \label{eq_intro_Jac_O} 
\end{align}

The codomain of \eqref{eq_intro_Jac_Sp}, $\TEJF_{2k}$, is {\it defined} to be the genuinely $Sp(1)$-equivariant twisted $\TMF$, and studied in detail in Appendix \ref{app_TEJF}. We name it the spectrum of {\it Topological Even Jacobi Forms} since, as explained there, it is naturally regarded as refining the following direct summand of $\JF_{2k}$:
\begin{align}
\EvJF_{2k} &:= \{\phi(z, \tau)\in \JF_{2k} \ | \ \phi(z, \tau) = \phi(-z, \tau) \},
\end{align}
\coloredbox{conjcolor}{The morphism \eqref{eq_intro_Jac_O} is a topological refinement of the Witten-Landweber-Ochanine genus \cite{Ochanine}. }

In the remainder of this introduction, we focus on the $Sp(1)$-topological elliptic genus \eqref{eq_intro_Jac_Sp} and illustrate why it tells us interesting things about $Sp$-manifolds beyond the $U(1)$-topological elliptic genus \eqref{eq_intro_Jac_U}. 
Note that, at the classical level, the elliptic genus for $Sp$-manifolds is just the restriction of the assignment \eqref{eq_intro_Jac_clas}, and we obtain no further information. However, after the topological refinement, we detect an interesting difference. The relationship between the $Sp(1)$ and $U(1)$-topological elliptic genera is captured in a commutative diagram
\begin{align}\label{diag_intro_SpvsU}
	\xymatrix{
		MTSp(k) \ar[rr]^-{\Jac_{Sp(1)_k}} \ar[d] && \TEJF_{2k} \ar[d]_-{r} \ar@{~>}[r]^-{\pi_m}& \pi_m \TEJF_{2k} \ar[d]^-{r} \ar[r]^-{e_\EJF}& \EJF_{2k}|_{\deg = m} \ar@{^{(}-_>}[d]^-{= \mbox{ {\tiny or} }0}  \\
		MTSU(2k) \ar[rr]^-{\Jac_{U(1)_{2k}}} && \TJF_{2k} \ar@{~>}[r]^-{\pi_m} & \pi_m \TJF_{2k} \ar[r]^-{e_\JF} & \JF_{2k}|_{\deg = m}
	}
\end{align}

What makes the diagram \eqref{diag_intro_SpvsU} interesting is that the third vertical arrow $r$ is neither injective nor surjective. This is in contrast to the rightmost vertical arrow, which is simply zero or the identity depending on the degrees. 
This means the following: 
\begin{enumerate}[resume]
	\item The genuine $Sp$-topological elliptic genus $\Jac_{Sp(1)_k}$ can detect (necessarily torsion) elements in the tangential $Sp$-bordism groups that vanish in $SU$-bordism groups. 
	Examples of such elements are given in Section \ref{subsec_pi5MSp}. 
	\item For $m \equiv 0$ (mod $4$), although the rightmost vertical arrow is an isomorphism, the map
	\begin{align}
		\mathrm{im}\left(e_\EJF \colon \pi_m \TEJF_{2k} \to  \EJF_{2k}|_{\deg = m}\right) \hookrightarrow 	\mathrm{im}\left(e_\JF \colon \pi_m \TJF_{2k} \to  \JF_{2k}|_{\deg = m}\right) 
	\end{align}
	is generally a proper inclusion. This implies nontrivial divisibility constraints on the classical elliptic genera and the associated characteristic classes for $Sp$-manifolds. We investigate this in \ref{subsec_divisibility}.
\end{enumerate}

This paper lays the basics of topological elliptic genera and discusses immediate applications. A notable application is to the divisibility of Euler numbers, such as the following result for tangential $Sp$-manifolds: 
\begin{thm}[{Theorem \ref{thm_divisibility_constraints_concrete} (1)}]
		For any closed strict\footnote{See Definition \ref{def_strict_str}.} tangential $Sp(k)$-manifold $M_{4k}$ of real dimension $4k$, its Euler number satisfies
		\begin{align}
			\left. \frac{24}{\gcd(24, k)} \,\right|\, \Eul(M_{4k}). 
		\end{align}
	\end{thm}
This comes from an elementary analysis of $\TEJF_{2k}$, together with the classical relation between $\Jac_\clas$ and Euler numbers. As we explain in Section \ref{subsec_divisibility}, this strictly refines the divisibility constraints obtained by classical numerical methods. The base case of $k=1$ for $Sp(1) = SU(2)$-manifolds gives divisibility by 24, which is saturated by the Euler number of a $K3$ surface.

Section \ref{sec_duality} discusses an interesting byproduct of our main construction, the {\it level-rank duality} in equivariant $\TMF$. 
The definition of the $U(1)_k$-topological elliptic genus uses a $\TMF$-module morphism 
\begin{align}\label{eq_intro_coev}
	\TMF \to \TJF_k \otimes_\TMF \TMF[\overline{V}_{SU(k)}]^{SU(k)}, 
\end{align}
where $\TMF[\overline{V}_{SU(k)}]^{SU(k)}$ is the $SU(k)$-equivariant $\TMF$ with fundamental (``level $1$'') twist. It turns out that the morphism \eqref{eq_intro_coev} exhibits $\TJF_k$ as the {\it dual} of $\TMF[\overline{V}_{SU(k)}]^{SU(k)}$ in $\Mod_\TMF$ (in the categorical sense). 
More generally, we find the following dualities (Theorems~\ref{thm_duality_Sp},~\ref{thm_duality_U_SU}), 
\begin{align}
	\TMF[kV_{Sp(n)}]^{Sp(n)} &\stackrel{\mbox{\tiny dual} }{\longleftrightarrow}	\TMF[n\overline{V}_{Sp(k)}]^{Sp(k)} \\
	\TMF[kV_{U(n)}]^{U(n)} &\stackrel{\mbox{\tiny dual  }}{\longleftrightarrow}	\TMF[n\overline{V}_{SU(k)}]^{SU(k)}  \quad \mbox{in }\Mod_\TMF. 
\end{align}
This coincides with the {\it level-rank duality} \cite{frenkel2006representations,nakanishi1992level} in affine Lie algebras and conformal field theory. 
Such an agreement is naturally expected in the context of the Segal-Stolz-Teichner proposal (see Remark \ref{rem_levelrank}). 

The authors plan to explore this topological elliptic genera in a series of papers. This is the first part of the series, where we lay the basics of the theory. In Part II \cite{LinYamashita2} of the series, we plan to discuss the physical interpretations. In further volumes, we plan to explore further examples and applications. 

The paper is organized as follows. After the preliminary Section \ref{sec_preliminary}, in Section \ref{sec_jacobi} we give the general definition and basics of topological elliptic genera. In Section \ref{sec_ex}, we introduce an important class of our construction, the $U$, $Sp$ \colorbox{conjcolor}{and $O$-}topological elliptic genera. We will see that these families of topological elliptic genera organize into a unified picture, and we refer to them as the {\it trio}. 
Section \ref{sec_char_formula} gives the characteristic class formula for the equivariant Modular Forms associated with our topological elliptic genera. Section \ref{sec_duality} discusses the level-rank duality. Finally, in Section \ref{sec_application} we discuss immediate applications of our construction, including the divisibility of Euler numbers mentioned above. The contents of Sections \ref{sec_char_formula}, \ref{sec_duality}, and \ref{sec_application} 
can be read independently of each other, and the reader may find it useful to read in their preferred order.

Appendices \ref{app_TJF} and \ref{app_TEJF} are about the basics of $\TJF$ ($=U(1)$-equivariant twisted $\TMF$) and $\TEJF$($=Sp(1)$-equivariant twisted $\TMF$), respectively. 
The authors believe these spectra are of independent interest and hope that the self-contained appendices contribute to future studies. 
The content of Appendix \ref{app_TJF} is contained in an upcoming work by Bauer-Meier \cite{BauerMeierTJF}, so the authors claim no originality of the content. On the other hand, the content of Appendix \ref{app_TEJF} is a new result of this paper. 

Appendix \ref{sec_app_toymodel} explains a toy model of the main body of this paper, where we replace $\TMF$ with $\KO$, resulting in {\it topological $\Gm$-genera}. Although the contents of that section are not used in the main body, the authors hope they serve as a warm-up to the main part. 

We conclude the introduction with an important remark. This paper relies on the equivariant refinement of the sigma orientation. As explained in Section \ref{subsec_stringstr}, currently, we have partially established the equivariant sigma orientation for a nice class of compact Lie groups, but not for all compact Lie groups. 
In this paper, we derive mathematical results based on the current status (Fact \ref{fact_sigma}). However, we would also like to present the results we can get once we assume the full establishment of the equivariant sigma orientation, \colorbox{conjcolor}{Conjecture \ref{conj_sigma}}; this gives us a complete and unified picture of our topological elliptic genera. Therefore, in this paper, we put \colorbox{conjcolor}{shaded backgrounds} on the statements and proofs that depend on \colorbox{conjcolor}{Conjecture \ref{conj_sigma}}. The rest of the contents are based on the current status and are completely valid.

\subsection{Notations and conventions}\label{subsec_notations}

\begin{enumerate}
\item $\Spectra$ denotes the stable $\infty$-category of spectra, $\cS$ the $\infty$-category of spaces, and $\cS_*$ that of pointed spaces. $S \in \Spectra$ the sphere spectrum, and $\cptLie$ the category of compact Lie groups and continuous homomorphisms. 
$e \in \cptLie$ denotes the trivial group. 

\item We denote by $\eta \in \pi_1 S$ and $\nu \in \pi_3 S $ the ({\it integral}, not $2$-local) generators of $\pi_1 S \simeq \Z/2$ and $\pi_3 S \simeq \Z/24$ respectively, which we choose to be represented by the Lie group manifolds $U(1)$ and $SU(2)$, respectively.

\item The notations on $G$-equivariant homotopy theory are summarized in Section \ref{subsec_preliminary_SpG}. Among others, we note that $E^G$ denotes the {\it genuine} $G$-fixed point spectrum of a genuine $G$-spectrum $E \in \Spectra^G$. 
\item \label{notation_BG} For $G \in \cptLie$, we dentote by $\bfB G$ the topological stack $\bfB G := [*//G]$. On the other hand, $BG \in \cS$ denotes the classifying space. 
Let $\Rep_O(G)$ denote the {\it groupoid} of orthogonal representation of $G$ and isomorphisms, and $\RO(G)$ denote those of virtual orthogonal representations. 
\item Given an element $\tau \in \RO(G)$, we denote by $S^\tau \in \Spectra^G$ the virtual representation sphere spectrum, and denote
    \begin{align}\label{notation_E[tau]}
        E[\tau] := E \otimes S^\tau \in \Spectra^G. 
    \end{align}
    In particular, we write, for any $E \in \Spectra$ and any integer $n \in \Z$,
    \begin{align}
    	E[n] := E[n\underline{\R}]= \Sigma^n E. 
    \end{align}
    \item For an element $\tau \in \RO(G)$, we define
    \begin{align}\label{eq_red_dim_rep}
    \overline{\tau} := \tau - \dim(\tau) \cdot 1 \in \RO(G)
    \end{align}
     where $1 = \underline{\R} \in \RO(G)$ is the class of the one-dimensional trivial representation. Similarly, for a real virtual vector bundle $\theta$ over a topological space $X$, we denote
     \begin{align}\label{eq_red_dim_vecbdle}
     	\overline{\theta} := \theta - \rank(\theta) \cdot \underline{\R}, 
     \end{align}
     where $\underline{\R}$ denotes the trivial real vector bundle over $X$; when $X = BG$, this agrees with the previous meaning of $\underline{\R}$.
    \item \label{notation_chi}
    For a real $G$-representation $V$, we define
    \begin{align}
        \chi(V) \in \Map(S^0, S^V)^{G}
    \end{align}
    to be the unique nontrivial $G$-equivariant map sending $0 \mapsto 0$ and $\infty \mapsto \infty$. We also use the same symbol to mean the $G$-equivariant map
    \begin{align}\label{eq_def_chiV}
        \chi(V) := \id_E \otimes \chi_V \colon E \to E \otimes S^V = E[V]
    \end{align}
    for any $G$-equivariant spectrum $E$. 
    The homotopy class of $\chi(V)$ is called the {\it Euler class} for the representation $V$, and we also denote it by the same symbol $\chi(V) \in \pi_0 E[V]^G$.
    
    \item For $G= \cG(n)$ with $\cG$ being one of $U, SU, O, Sp, Spin, SO$, we denote by $V_{G} \in \Rep_O(G)$ its fundamental (a.k.a., defining, or vector) representation. 
    
    \item For a space $X$, we denote by $X \to P^nX $ the $n$-th Postnikov truncation, and by $X\langle n \rangle \to X$ the $n$-connected cover. 
    In particular, the Whitehead towers of $BU$ and $BO$ are in low degrees related as follows. 
    \begin{align}\label{diag_Whitehead}
    	\xymatrix{
    	BU\langle 6 \rangle \ar[r]  \ar[d] & BO\langle 8 \rangle = BString  \ar[d] \\
    	BU \langle 4 \rangle = BSU \ar[r] \ar[d] & BO\langle 4 \rangle =BSpin  \ar[d] \\
    	BU \ar[r] & BO\langle 2 \rangle = BSO   \ar[d] \\
    	 & BO. 
    	}
    \end{align}
    
    \item \label{notation_tw} For an element $\tau \in \RO(G)$, we denote by $\tw(\tau)$ the map
    \begin{align}
    	\tw(\tau) := \dim \tau + \left(BG \xrightarrow{\overline{\tau}} BO \to P^4BO \right) \in \Z \times \Map(BG, P^4BO). 
    \end{align} 
    We also abuse the notation to denote by $\tw(\tau)$ its homotopy class in $\Z \times [BG, P^4BO]$. The notation comes from the fact that $\tw(\tau)$ is understood as twists of $G$-equivariant $\TMF$ associated to $\tau \in \RO(G)$, as explained in Section \ref{subsubsec_twist}. 
    \item \label{prelim_compatible} In a symmetric monoidal category $(\mathcal{C}, \otimes)$, suppose we have objects $a, b, c, d, x$ and morphisms $f \colon x \to a \otimes b$, $g \colon x \to c \otimes d$, $h \colon a \to c$ and $k \colon d \to b$. We say that the diagram
    \begin{align}
        \xymatrix{
        x \ar[r]^-{f} \ar[rd]_-{g}& a \ar[d]^-{h} \ar@{}[r]|{\otimes}& b \\
        & c \ar@{}[r]|{\otimes}& d \ar[u]^-{k}
        }
    \end{align}
    is {\it compatible} if the square
    \begin{align}\label{diag_compatible_square}
    \xymatrix{
    x \ar[r]^-{f} \ar[d]^-{g} & a \otimes b \ar[d]^-{h \otimes \id_b} \\
    c \otimes d \ar[r]^-{\id_c \otimes k} & c \otimes b
    }
    \end{align}
    commutes. 
    \item \label{notation_Thom} Given a space $X$ with a real vector bundle $\theta$, we denote the associated Thom spectrum by
    \begin{align}
        X^\theta := \Sigma^{\infty} \mathrm{Thom}(\theta \to X) \in  \mathrm{Spectra}. 
    \end{align}
    More generally, this notation allows $\theta$ to be a virtual vector bundle, e.g., \cite{ABGThomSpectra}. 
    
    \item \label{notation_MT} In this article, it is important to distinguish between tangential and normal bordism Thom spectra. Given a space $\mathcal{B}$ with a map $f \colon \mathcal{B} \to BO$, 
    \begin{align}
        M(\mathcal{B}, f) &:= \mathcal{B}^{f}, \\
        MT(\mathcal{B}, f) &:= \mathcal{B}^{-f}, 
    \end{align}
    where we identify $f$ with a virtual real vector bundle of rank $0$ over $\mathcal{B}$. The spectra $M(\mathcal{B}, f)$ and $MT(\mathcal{B}, f)$ classify the bordism (co)homology theories of manifolds with {\it normal} and {\it tangential} $(\mathcal{B}, f)$-structures. The details are explained in Section \ref{subsec_thom}. 
    When $f$ is canonically understood, we often omit it from the notation and write, e.g., $MSU(k)$ and $MTSU(m)$. 
  
    \item For an $E_\infty$ ring spectrum $R$, we denote by $u \colon S \to R$ the unit map. 
    \item \label{notation_dual} Let $R$ be an $E_\infty$ ring spectrum. For a dualizable object $x \in \Mod_R$, we denote by $D_R(x)$ its dual in $\Mod_R$. In this article, we mostly use this notation for $R=\TMF$, so we adopt the shorthand $D:= D_{\TMF}$. 
    \item \label{notation_deg} For a $\Z$-graded abelian group $A$ and an integer $m$, we denote by $A|_{\deg=m}$ the degree-$m$ component of $A$. 
    \item  We use the following convention on modular forms. 
	We denote by $$\MF := \Z[c_4, c_6, \Delta, \Delta^{-1}]/(c_4^3 - c_6^2 - 1728\Delta)$$ the ring of weakly-holomorphic integral modular forms (i.e., holomorphic away from the cusps and having integral Fourier coefficients in the variable $q =\exp(2\pi i \tau)$). 
	    In the text, we capitalize ``Modular Forms'' to mean weakly holomorphic modular forms. 
        We put the  $\Z$-graded ring structure so that $ \MF|_{\deg = m}$ consists of those of weight $\frac{m}{2}$. This way we have a canonical map
        \begin{align}
            e_{\MF} \colon \pi_{m} \TMF  \to  \MF|_{\deg = m}. 
        \end{align}
        {\it Holomorphic} modular forms (holomorphic also at the cusps) figure in Section \ref{subsubsec_divisibility_naive}. We denote by $$\mf := \Z[c_4, c_6, \Delta]/(c_4^3-c_6^2 - 1728\Delta)$$ the corresponding graded ring.
    \item \label{notation_JF} We use the convention on Jacobi forms following, e.g., \cite{dabholkar2014quantum,Gristenko2020382}. 
    We denote by $\mathfrak{H} := \{\tau \in \C \mid \mathrm{Im}(\tau ) > 0 \}$ the upper half space of the complex plane. 
    For each $k \in \Z_{\ge 0}$ and $w \in \Z$, consider holomorphic functions of $(z, \tau) \in \C \times \mathfrak{H}$ satisfying the transformation properties (c.f., Definition \ref{def_Loo}),
    \begin{align}
    \phi\left(\frac{a\tau + b}{c \tau + d}, \frac{z}{c\tau + d} \right) &= (c\tau + d)^w e^{\frac{\pi i k cz^2}{c\tau + d}} \phi(\tau, z),  \\
    \phi(\tau, z + \lambda \tau + \mu) &= e^{-\pi i  k(\lambda ^2 \tau + 2 \lambda z) } \phi(\tau, z)  
    \end{align}
    for all $\begin{pmatrix}
    	a & b \\ c & d 
    \end{pmatrix} \in SL(2, \Z) $ and $(\lambda, \mu ) \in \Z^2$, and having Fourier expansions
        \ie 
            \phi(q, y) = \sum_{r \in \Z + \frac{k}{2}} \sum_{n \ge N} c(n, r) q^n y^r 
        \fe
        where $(q, y) = (\exp(2\pi i \tau), \exp(2\pi i z))$ for some integer $N$. 
        \begin{itemize}
            \item Such functions are called {\it weakly holomorphic} Jacobi forms of index $\frac{k}{2}$ and weight $w$. We mostly deal with this type of Jacobi forms in this paper. 
            \item If $c(n, r) \neq 0$ only when $n \ge 0$, then such functions are called {\it weak} Jacobi forms. This type of Jacobi forms only appear in Section \ref{subsubsec_divisibility_naive}.
            \item In addition, if $c(n, r) \neq 0$ only when $r^2 \ge 4kn$, then such functions are called {\it holomorphic} Jacobi forms. But in this paper we do not talk about this type of Jacobi forms.
            \item If all $c(n,r) \in \Z$, we add the adjective {\it integral} in all the above cases.
        \end{itemize}
        In the text, we capitalize the first letters in ``Jacobi Forms'' to mean {\it weakly holomorphic Jacobi forms}, and denote by $\JF_k$ the set of all integral Jacobi Forms with index $\frac{k}{2}$. We put the $\Z$-grading on $\JF_k$ so that $\JF_k|_{\deg = m}$ consists of Jacobi Forms with weight $w = -k+\frac{m}{2}$. This makes $\JF_k$ a $\Z$-graded module over the $\Z$-graded ring $\MF$. As will be recalled in Section \ref{subsec_JF_app}, we have a canonical map
        \begin{align}
            e_{\JF} \colon \pi_{m}\TJF_k = \pi_m \Gamma(\mathcal{E}; \mathcal{O}_{\mathcal{E}}(ke)) \to    \JF_k|_{\deg = m}. 
        \end{align}
        
        {\it Weak} Jacobi forms figure in Section \ref{subsubsec_divisibility_naive}. We denote by $\jF_k$ the $\mf$-submodule consisting weak Jacobi forms, i.e.,
        \begin{align}
			\jF_k := \JF_k \cap \Z((y))[[q]]. 
        \end{align}
    \item \label{notation_a} For notational ease, we write
    \begin{align}\label{eq_notation_a}
    	a = \phi_{-1, \frac12} = 	\frac{\theta_{11}(z, q)}{\eta^3(q)}  
    	=(e^{\pi i z} - e^{-\pi i z}) \prod_{m \ge 1} \frac{(1-q^m e^{2\pi i z})(1-q^me^{-2\pi i z})}{(1-q^m)^2}. 
    \end{align}
    This is an element in $\JF_1|_{\deg = 0}$ and a generator of the $\Z$-graded ring $\oplus_k \JF_k$ of Jacobi Forms \eqref{eq_JF_generators}; the notation $\phi_{-1, \frac12}$ is employed in \cite{Gritsenko:1999fk}. 
\end{enumerate}

\section{Preliminaries}\label{sec_preliminary}

\subsection{Generalities on genuinely equivariant spectra}
\label{subsec_preliminary_SpG}

Equivariant stable homotopy theory is an expansive subject, and there are various realizations of the equivariant stable homotopy category, e.g., those based on orthogonal spectra and those based on orbispaces. We refer to \cite[Appendix C]{GepnerMeier} for a nice account of those formulations and relations. 
However, in this paper, we only need the basic structure of the equivariant stable homotopy category, and this section is aimed at giving a minimal account of what we need in this paper and setting up the notation. 
Practically, we employ the definition $\mathrm{Spectra}^G := \mathrm{Sp}_{\mathcal{U}}^G$ in \cite[Definition C.1]{GepnerMeier}, and call it the $\infty$-category of genuinely $G$-equivariant spectra. This is based on orbispaces, but it was shown in \cite[Appendix C.2]{GepnerMeier} that they are equivalent to the more classical definition based on orthogonal spectra. 

Let $\mathcal{S}^G_*$ be the $\infty$-category of pointed $G$-spaces and $G$-equivariant maps, where equivalence is given by maps $f \colon X \to Y$ that induce a weak equivalence on the fixed points $f \colon X^{H} \simeq Y^{H}$ for all subgroups $H \subset G$. $\cS^G$ is a symmetric monoidal category with the smash product $\wedge$. 
In particular, we have $S^V \in \cS^G_*$ for all orthogonal representations $V \in \Rep_O(G)$. 
Informally speaking, the stable $\infty$-category $\Spectra^G$ is obtained by formally inverting the operation $S^V \wedge - $ on $\cS^G$. The symmetric monoidal structure $\wedge$ in $\cS^G_*$ extends to a symmetric monoidal structure on $\Spectra^G$, which we denote by $\otimes$. 
In particular, we have a symmetric monoidal functor
\begin{align}
	\Sigma^\infty \colon \cS^G_* \to \Spectra^G,
\end{align}
which preserves colimits\footnote{However, it does not preserve limits. So a {\it cofiber} sequence $X \to Y \to Z$ in $\cS^G_*$ produces a fiber=cofiber sequence
\begin{align}
	E \otimes X \to & E \otimes Y \to E \otimes  Z, \\
	\lMap_G( Z, E) \to & \lMap_G( Y, E) \to \lMap_G( X, E)
\end{align}
in $\Spectra^G$ for each $E \in \Spectra^G$, but we could not have started from a fiber sequence in $\cS^G_*$. 
}. We abuse the notation to denote $E \otimes X := E \otimes \Sigma^\infty X$ for $E \in \Spectra^G$ and $X \in \cS^G_*$. 
The category $\Spectra^G$ has internal homs, which we denote by $$\underline{\Map}_G(X, Y ) \in \Spectra^G$$ for $X, Y \in \Spectra^G. $

There are several notions of fixed point spectra for an equivariant spectrum $E \in \Spectra^G$, and in this paper, we use the {\it genuine} fixed point spectra, denoted by $E^G \in \Spectra$. This assignment gives a functor of stable $\infty$-categories
\begin{align}
	(-)^G \colon \Spectra^G \to \Spectra, \quad E \to E^G. 
\end{align} 

From this, we get the classical notion of $\RO(G)$-graded equivariant (co)homology groups as follows. 
For a virtual orthogonal representation $\tau \in \RO(G)$, we denote by $S^\tau \in \Spectra^G$ the virtual representation sphere spectrum.\footnote{In particular, we abuse the notation to denote $\Sigma^\infty S^V \in \Spectra^G$ by $S^V$ when $V$ is a orthogonal representation. }
Given another genuinely $G$-equivariant spectrum $X \in \Spectra^G$, its $G$-equivariant $E$-cohomology groups and homology groups with degree $\tau \in \RO(G)$ are defined as
\begin{align}
	E^\tau_G(X) &:= \pi_{0}\underline{\Map}_G(X, E \otimes S^\tau)^G = \pi_0 \underline{\Map}_G(X, E[\tau])^G , \\
	E_\tau^G(X)& := \pi_0 (X \otimes E\otimes S^{-\tau})^G = \pi_0(X \otimes E[-\tau])^G , 
\end{align}
respectively, where we have employed the notation $E[\tau] := E \otimes S^\tau \in \Spectra^G$ as in \eqref{notation_E[tau]}. 

Another important ingredient is the {\it norm map} between homotopy orbit and genuine fixed points. Let us denote by $\Ad G \in \Rep_O(G)$ the adjoint representation of $G$. For each $E \in \Spectra^G$, the norm map is a morphism in $\Spectra$ defined as
\begin{align}\label{eq_norm_equivariant}
	\Nm \colon E_{hG} \simeq (EG_+ \otimes E [-\Ad G])^G \xrightarrow{EG_+ \to S^0} E[-\Ad G]^G, 
\end{align}
where $E_{hG}$ is the homotopy orbit spectrum and the first equivalence is the Adams isomorphism. 

Finally, let us introduce notions related to the change of groups. Given a homomorphism of compact Lie groups $f \colon H \to G$, we have a restriction functor
\begin{align}
	\res_f \colon \Spectra^G \to \Spectra^H. 
\end{align}

Moreover, if $f$ is an inclusion of a subgroup $f \colon H \hookrightarrow G$, we often denote the restriction by $\res_G^H$. In this case, $\res_G^H$ admits both the left and right adjoints. We denote the left adjoint by $\ind_H^G$ and also use the suggestive notation $\ind_H^G E = E \wedge_H G_+$. By the Wirthm\"uller isomorphism, we get the transfer map (only along an inclusion of closed subgroups!) 
\begin{align}\label{eq_tr_incl}
	\tr_H^G \colon \left(\res_G^H (E) [-\Ad H]\right)^H \to E[-\Ad G]^G 
\end{align}

\subsection{Equivariant $\TMF$ and their twists}\label{subsec_preliminary_TMF_G}

In this subsection, we summarize the theory of genuinely equivariant elliptic cohomology developed by Gepner and Meier \cite{GepnerMeier}. 
They refine elliptic cohomology theory, in particular $\TMF$, to a globally equivariant spectrum, in the sense that $\TMF$ is refined to objects in $\Spectra^{G}$ for all compact Lie groups $G$ all at once, functorially in $G$. 
First, we briefly summarize their construction in Section \ref{subsubsec_GepnerMeier}, and then we relate it with the more elementary complex analytic story in Section \ref{subsubsec_TMF_C}. 

\begin{rem}
	 The Gepner-Meier work is based on spectral algebraic geometry, so Section \ref{subsubsec_GepnerMeier} below necessarily involves that language. However, we do not assume the reader to have {\it any} knowledge of spectral algebraic geometry at all; all we need in this paper is the consequence of the Gepner-Meier construction, that we obtain a genuinely equivariant refinement of $\TMF$ with nice dualizability properties, as summarized below. 
\end{rem}

\subsubsection{The construction of Gepner-Meier \cite{GepnerMeier}}\label{subsubsec_GepnerMeier}
For details of the following content, we refer to the original paper \cite{GepnerMeier}. 
As developed in the works of Lurie \cite{LurieElliptic1,LurieElliptic2,LurieElliptic3}, spectral algebraic geometry gives a conceptual framework of elliptic cohomology. 
Given a preoriented spectral elliptic curve $ \mathcal{E} \to \moduli$ over a spectral Deligne-Mumford stack $\moduli$ (the term ``spectral algebraic'' is henceforth often omitted), the associated {\it elliptic spectrum} is simply defined as 
\begin{align}\label{eq_SAG_ellcoh}
	R_{\cE} := \Gamma(\moduli; \mathcal{O} ) \in \CAlg, 
\end{align}
the global section of the structure sheaf of the moduli $\moduli$. 
In particular, if we apply it to the universal elliptic curve $\cE_{\univ} \to \moduli_\univ$, we get the spectrum of Topological Modular Forms, $\TMF:= R_{\cE_{\univ}} = \Gamma(\moduli_\univ; \mathcal{O})$. 

Gepner and Meier's work refines the elliptic spectrum \eqref{eq_SAG_ellcoh} into a globally equivariant $E_\infty$-spectrum, as follows. Their main construction is the {\it equivariant elliptic cohomology functor }
\begin{align}\label{eq_Ell_GM}
    \Ell \colon \mathcal{S}_{\rm Orb} \to \mathrm{Shv}(\moduli), 
\end{align}
for each $\cE \to \moduli$, where $\cS_{\rm Orb}$ is the category of orbispaces regarded as a setting of globally equivariant homotopy theory. 
The category $\cS_{\rm Orb}$ includes the object $\bfB G = [*//G]$ (see Section \ref{subsec_notations} \eqref{notation_BG}) for each compact Lie group $G$, and the functor \eqref{eq_Ell_GM} is defined so that $\Ell(\bfB G)$ is regarded as a spectral algebraic counterpart of the complex analytic moduli stack $\moduli_\C^G$ (see \eqref{eq_MG_C} below) of flat $G$-bundles on dual elliptic curves; namely, we have a canonical identification $\Ell(\bfB A) \simeq \Hom(\hat{A}, \cE) $ for each compact abelian Lie group $A$ with its Pontryagin dual $\hat{A}$, so in particular
\begin{align}\label{eq_EllBU(1)=E}
	\Ell(\bfB U(1)) \simeq \mathcal{E} , \quad  \Ell(\bfB C_n) \simeq \cE[n]
\end{align}
where $\cE[n] \subset \cE$ is the $n$-torsion of elliptic curves, and the functor \eqref{eq_Ell_GM} is given by the left Kan extension from the above cases. 

For each compact Lie group $G$, we have the Yoneda inclusion functor $\cS^G_* \to \cS_\Orb{}_{/\bfB G}$. Precomposing this with the functor $\Ell$, we get a colimit-preserving functor
\begin{align}\label{eq_EllG}
    \widetilde{\cEll}_G \colon \mathcal{S}^G_* \to \QCoh(\Ell(\bfB G))^\op.  
\end{align}
We further compose with the functor $\Gamma$ taking the global sections to get a colimit-preserving functor
\begin{align}\label{eq_GammaEllG}
\Gamma	\widetilde{\cEll}_G \colon \mathcal{S}^G_* \to \Spectra^\op, \quad X \mapsto \Gamma(\Ell(\bfB G); \widetilde{\cEll}_G(X)) \simeq \Gamma(\moduli;  \Ell(X //G)).  
\end{align}
Furthermore, they show that the functor \eqref{eq_GammaEllG} is represented by a genuine $G$-spectrum, also denoted by $R_\cE \in \Spectra^G$ in a way that is functorial in $G$. This means that we have canonical identifications
\begin{align}\label{eq_def_TMFcoh}
	\lMap_G(X, R_{\cE})^G \simeq  \Gamma(\Ell(\bfB G); \widetilde{\cEll}_G(X)) \simeq \Gamma(\moduli; \Ell(X//G)), 
\end{align}
for each $X \in \cS^G_*$ so that the equivariant cohomology group is identified as
\ie
    R_{\cE, G}^*(X) \simeq \pi_{-*}\Gamma(\Ell(\bfB G); \widetilde{\cEll}_G(X)) \simeq \pi_{-*} \Gamma(\moduli;  \Ell(X//G)).
\fe
In particular, we have
\begin{align}
	(R_{\cE})^G \simeq \Gamma(\Ell(\bfB G); \mathcal{O}_{\Ell(\bfB G)}) \simeq \Gamma(\moduli; \Ell(\bfB G)). 
\end{align}
This gives the desired globally equivariant refinement of $R_\cE$ in \eqref{eq_Ell_GM}.

For each orthogonal representation $V \in \Rep_O(G)$ of $G$, we set
\begin{align}
    \cL(V) := \widetilde{\cEll}_G(S^V) \in \Pic(\Ell(\bfB G)) := \QCoh(\Ell(\bfB G))^\times, 
\end{align}
which is shown to be the {\it invertible} elements in $\QCoh(\Ell(\bfB G))$. This allows us to more generally denote, for each virtual representation $V=W_1-W_2 \in \RO(G)$ with $W_1, W_2\in\Rep_O(G)$, 
\begin{align}
    \cL(V) := \cL(W_1) \otimes \cL(W_2)^{-1} . 
\end{align}
We get 
\begin{align}\label{eq_def_ROGTMF}
   \TMF[V]^G :=  (\TMF \otimes S^V)^G = \TMF(S^{-V})^G = \Gamma(\Ell(\bfB G); \cL(-V)). 
\end{align}
If $G, H \in \cptLie$ with $V_G \in \RO(G)$ and $V_H \in \RO(H)$, we have an isomorphism of $\TMF$-modules \cite{GepnerMeierNEW}, 
\begin{align}\label{eq_TMF_tensor}
	\TMF\left[\res_{G}^{G \times H} V_G \oplus \res_{H}^{G\times H} V_H \right]^{G \times H} \simeq \TMF[V_G]^{G} \otimes_\TMF \TMF[V_H]^H. 
\end{align}

\begin{rem}\label{rem_tensor_flat}
	The isomorphism \eqref{eq_TMF_tensor} does not necessarily mean that we have the corresponding tensor product decomposition of homotopy groups. Rather, the computation of the homotopy group of \eqref{eq_TMF_tensor} involves the Kunneth spectral sequence. 
	But {\it rationally}, it is often true that $\pi_* \TMF_\Q[V_G]^G$ is flat over $\pi_* \TMF_\Q = \MF_\Q$. 
	This is true for $\TMF_\Q[V]^{U(1)^n}$ for arbitrary $n$ and $V \in \RO(U(1)^n)$, as well as for $\TMF_\Q[kV_{U(n)}]^{U(n)}$ and $\TMF_\Q[kV_{SU(n)}]^{SU(n)}$ for arbitrary $n$ and $k$, for example\footnote{
		The outline of the proof is as follows. 
	The $G= U(1)$-case follows because $\pi_* \TJF_k \otimes \Q$ is free over $\MF_\Q$ for any $k$. This implies by \eqref{eq_tensor_flat} that the $\TMF_\Q$-module of the form $\TMF_\Q[W_1 \oplus W_2 \oplus \cdots \oplus W_n]^{U(1)^n}$ also satisfies this property where $W_i \in \RO(U(1))$. For general $V \in \RO(U(1))$, we use Fact \ref{fact_sigma} and the classification of bilinear forms on $\Z^n$ by Witt groups to reduce the previous diagonal case. 
	The case for $G=SU(n)$ and $G=U(n)$ follows from the $U(1)^n$-case by taking the Weyl group invariant part. 
	}. 
	In such cases, we have
	\begin{align}\label{eq_tensor_flat}
		\pi_* \TMF_\Q \left[\res_{G}^{G \times H} V_G \oplus \res_{H}^{G\times H} V_H \right]^{G \times H} \simeq
		\pi_* \TMF_\Q[V_G]^{G} \otimes_{\MF_\Q} \pi_* \TMF_\Q[V_H]^H \\
		\mbox{if }	\pi_* \TMF_\Q[V_G]^{G} \mbox{ or } \pi_* \TMF_\Q[V_H]^H \mbox{ is flat over }\MF_\Q \notag
	\end{align}
\end{rem}

\begin{ex}[$G=U(1)$: Topological Jacobi Forms]\label{ex_TJF_preliminary}
	The case of $G=U(1)$ is fundamental, and plays an important role in this paper. It is called {\it Topological Jacobi Forms} and studied in detail in an upcoming paper by Bauer-Meier \cite{BauerMeierTJF}, and we have summarized the necessary results in Appendix~\ref{app_TJF}. In this paper, we employ the definition (Definition \ref{def_TJF}) that, for each integer $k$, 
\begin{align}
	\TJF_{k} := \TMF[kV_{U(1)}]^{U(1)} \simeq   \Gamma(\cE; \cO_{\cE}(ke)), 
\end{align}
where we have used $\Ell(\bfB U(1)) \simeq \cE$ \eqref{eq_EllBU(1)=E} and the fact that $\cL(-kV_{U(1)}) \simeq \cO_\cE(ke)= \cO_\cE(e)^{\otimes k} \in \QCoh(\cE)^\times$, where $\cO_{\cE}(e)$ is the (SAG-version of the) sheaf of meromorphic functions on $\cE$ having pole of order at most $1$ at the zero section $e \colon \moduli \to \cE$. 
As explained below and in more detail in Appendix \ref{app_TJF}, $\TJF_k$ is regarded as a spectral refinement of the module of integral Jacobi Forms of index $k/2$. 
\end{ex}

\begin{ex}[$G=Sp(1)$: Topological Even Jacobi Forms]\label{ex_TEJF_preliminary}
	The case of $G=Sp(1)$ is also of particular importance for us. The twisted $Sp(1)$-equivariant $\TMF$ is surprisingly nicely understood, and we give a detailed account in Appendix \ref{app_TEJF}. 
	We employ the notation (Definition \ref{def_TEJF_app})
	\begin{align}
		\TEJF_{2k} := \TMF[kV_{Sp(1)}]^{Sp(1)}
	\end{align}
	for each $k \in \Z$ and call it {\it Topological Even Jacobi Forms}, by the reason explained in Example \ref{ex_EJF} below and in more detail in Appendix \ref{app_TEJF}. 
\end{ex}

An important feature of the genuinely equivariant $\TMF$ is the following dualizability statement: 
\begin{fact}[{Dualizability of $\TMF^G$ \cite{GepnerMeierNEW}}]\label{fact_dualizability_TMF}
	For any compact Lie group $G$, $\TMF^G$ is dualizable in $\Mod_\TMF$, with its dual (see Section \ref{subsec_notations} \eqref{notation_dual}) canonically identified as 
	$D(\TMF^G) \simeq \TMF[-\Ad (G)]^G$.  
\end{fact}
We remark that this is a special feature of equivariant $\TMF$; indeed, for example in the case of genuinely equivariant $\KU$-theory (with the usual equivariance), this dualizability does not hold. 
This allows us to define, for {\it any} homomorphism $f \colon G \to H$ of compact Lie groups, the {\it transfer map along $f$}, 
\begin{align}\label{eq_tr_f}
\tr_f \colon	\TMF[-\Ad(G)]^G \to \TMF[-\Ad(H)]^H 
\end{align}
to be the dual of the restriction map 
$\res_f \colon \TMF^H \to \TMF^G$. 
This extends the transfer map along inclusions $G \hookrightarrow H$ in \eqref{eq_tr_incl}, which exists for any genuinely $H$-equivariant spectra. The existence of this general transfer is a special feature of the genuinely equivariant $\TMF$\footnote{As we will explain in detail in Part II of this series of the papers, these transfer maps should correspond to {\it gauging} in quantum field theories.}. 

We also note that, for every $V \in \RO(G)$, $\TMF[V]^G$ is also dualizable in $\Mod_\TMF$ whose dual is identified as
\begin{align}\label{eq_twisted_dual}
	D(\TMF[V]^G) \simeq \TMF[-V-\Ad(G)]^G, 
\end{align}
by the coevaluation being 
\begin{align}\label{eq_coev_duality_TMF}
	\TMF[V]^G \otimes \TMF[-V -\Ad(G)]^G \xrightarrow{\rm multi} \TMF[-\Ad(G)]^G \xrightarrow{\tr_G^e} \TMF. 
\end{align}

Finally, let us remark on the Atiyah-Segal completion in this context. We have an adjunction (upper=left adjoint)
\begin{align}\label{diag_adj_Orb}
	\xymatrix{
		|\bullet| \colon	\cS_\Orb \ar@/^7pt/[rrr]^-{\sX \mapsto |\sX|}_-{\perp} &&& \mathcal{S}  \ar@/^7pt/[lll]^{ \Map(\bullet, Y) \leftarrow Y} \colon y
	}
\end{align}
For example, we have $|\bfB G| \simeq BG$ for $G \in \cptLie$. 
We recover the usual $\TMF$-cohomology from the equivariant elliptic cohomology functor $\Ell$ in \eqref{eq_Ell_GM} by the fact that the following diagram commutes: 
\begin{align}\label{diag_Ell_TMF}
	\xymatrix@C=5em{
	\cS \ar[r]^-{y} \ar[rrd]_-{Y \mapsto \lMap(\Sigma^\infty Y, \TMF)\quad \quad }& \cS_\Orb \ar[r]^-{\Ell} &\Shv(\moduli) \ar[d]^-{\simeq}_-{\Gamma}\\
	 &&  \Mod_\TMF
}
\end{align}
Let us denote the unit of the adjunction \eqref{diag_adj_Orb} by
\begin{align}\label{eq_u_adj}
	u_\sX \colon \sX \to y(|\sX|) 
\end{align}
for each $\sX \in \cS_\Orb$. Then we get, for each pointed $G$-space $X \in \cS_*^G$, the map
\begin{align}\label{eq_AS}
\zeta &\colon \lMap_{G}(\Sigma^\infty X, \TMF)^G \simeq	\Gamma(\moduli, \Ell(X//G)) \\
& \xrightarrow{u_{X//G}} \Gamma(\moduli, \Ell(y(|X//G|))) \stackrel{\eqref{diag_Ell_TMF}}{\simeq} \lMap(\Sigma^\infty X \wedge_G EG_+, \TMF) \simeq  \lMap_{G}(\Sigma^\infty X, \TMF)^{hG}
\end{align}
This coincides with the canonical map from the genuine to homotopy fixed points and is regarded as a generalization of the Atiyah-Segal completion map. In particular, we get the following map of the homotopy groups.
\begin{align}
	\zeta \colon \TMF_G^*(X) \to \TMF^*(X \wedge_G EG_+). 
\end{align}

\subsubsection{Specialization to elliptic curves over $\C$}\label{subsubsec_TMF_C}
Let $ \moduli_\C$ denote the classical Deligne-Mumford stack of elliptic curves over $\C$, and $p \colon \cE_\C \to \moduli_\C$ denote the universal elliptic curve over it. We use the usual identification (where we use the notation $\mathfrak{H} := \{\tau \mid \mathrm{Im}(\tau) > 0\}$)
\begin{align}\label{eq_ell_coord}
    \moduli_\C \simeq \mathfrak{H}//SL_2(\Z),  \quad 
    \mathcal{E}_\C \simeq (\C\times \mathfrak{H}) // (\Z^2 \rtimes SL_2(\Z) ). 
\end{align}
For $G$ connected and $\pi_1 G$ torsion-free, we have an identification \cite{GepnerMeierNEW}, 
\begin{align}\label{eq_MG_C}
   \moduli_\C^G \simeq \Ell(\bfB G)^{{\heartsuit}}_\C, 
\end{align}
where $\moduli_\C^G$ is the moduli stack of flat $G$-bundles over the dual elliptic curve $\mathcal{E}^\vee_\C$, and $\Ell(\bfB G)^{{\heartsuit}}_\C$ is the underlying Deligne-Mumford stack of $\Ell(\bfB G)$ after taking $\C$-points. So a virtual representation $V \in \RO(G)$ produces a line bundle $\cL_\C(-V) := \cL(-V)_\C^\heartsuit \in \Pic(\moduli_\C^G)$. 
By the functoriality of the Gepner-Meier construction, we have a canonical map 
\begin{align}\label{eq_evC}
	\red_\C \colon \pi_\bullet \TMF[V]^{G} \to \Gamma(\moduli_\C^G; \cL_\C(-V) \otimes p^*\omega^{\bullet/2}). 
\end{align}

In the case where $G$ is connected and $\pi_1 G$ is torsion-free\footnote{This condition is sufficient for the identification \eqref{eq_moduliG=moduliT/W} to hold. }, the right hand side of \eqref{eq_evC} can be nicely understood in terms of multivalued Jacobi Forms as follows. 
For each compact connected {\it abelian} Lie group $T$, we have a canonical identification
\begin{align}
    \moduli_\C^{T} \simeq \mathcal{E}_\C \times_\Z \Hom(S^1, T) , 
\end{align}
and an identification $T \simeq U(1)^r$ gives the corresponding identification $\moduli_\C^T \simeq (\mathcal{E}_\C )^{\times r}$ where the product is taken over $\moduli_\C$. 
Furthermore, for each connected compact Lie group $G$ with $\pi_1 G$ being torsion free, choosing a maximal torus $T \subset G$ with the Weyl group $W$, we have a canonical identification
\begin{align}\label{eq_moduliG=moduliT/W}
    \moduli_\C^G \simeq \moduli_\C^{T}/W  \quad \mbox{ for }G \mbox{ connected, }\pi_1 G \mbox{ torsion-free}
\end{align}
This allows us to identify sections of sheaves over $\moduli^G_\C$ in terms of those over $\moduli_\C^{T}$. In particular, given $V \in \RO(G)$ we have a canonical identification
\begin{align}\label{eq_G_MF}
    \Gamma(\moduli_\C^G; \cL_\C(-V)) \simeq \Gamma(\moduli_\C^T; \cL_\C(-\res_G^T V))^W. 
\end{align}
In this setting, 
The line bundle $\mathcal{L}_\C(V)$ is related to the line bundles constructed by Looijenga and its generalization \cite{gukov2025newapproach31dimensionaltqfts}: 
\begin{defn}[{Looijenga's line bundle $\mathcal{A}(\xi)$ \cite{gukov2025newapproach31dimensionaltqfts}}]\label{def_Loo}
	\begin{enumerate}
		\item For each nonnegative integer $r$, we have a canonical bijection
		\begin{align}\label{eq_twist_bil}
			[BU(1)^r, P^4BO] \simeq \left\lbrace b(-, -) \colon \Z^r \times \Z^r \to \Z : \mbox{symmetric bilinear form} \right\rbrace 
		\end{align}
		\item For each element $\xi \in [BU(1)^r, P^4 BO]$ we define the Looijenga's line bundle $\mathcal{A}(\xi)$ over $\moduli_\C^{U(1)^r} \simeq \cE^{\times r}_\C = (\C^{\times r} \times \mathfrak{H})//((\Z^2)^r \rtimes SL_2(\Z))$ by	\begin{align}\label{eq_cA_def}
			\mathcal{A}(\xi) := \C \times (\C^{\times r} \times \mathfrak{H}) //  ((\Z^r)^2 \rtimes SL_2(\Z)), 
		\end{align}
		where $\mathfrak{H}$ is the upper half plane and the action is given by (we use the coordinates $z=(z_1, z_2, \cdots, z_r) \in \C^r$, $\tau \in \mathfrak{H}$ and $u \in \C$)
		\begin{align}
			A \cdot (u, z, \tau) &= \left( e^{\pi i (c(c\tau + d)^{-1}\xi (z, z)) }u , (c\tau + d)^{-1}z, \frac{a\tau + b}{c\tau + d} \right) , \\
			(m_1, m_2) \cdot (u, z, \tau) &= \left( e^{-2\pi i (\xi(z, m_1) + \frac12 \xi(m_1, m_1))} u, z + m_1 + m_2, \tau \right) , 
		\end{align}
		for each $A = \begin{pmatrix}
			a & b \\
			c&d
		\end{pmatrix} \in SL_2(\Z)$ and $(m_1, m_2) \in (\Z^r)^2$. 
		Here we have denoted by $\xi(-, -) \colon \C^r \times \C^r \to \C$ the $\C$-linear extension of the symmetric bilinear form on $\Z^{r}$ corresponding to $\xi$ by the bijection \eqref{eq_twist_bil}. 
		\item 
		More generally, let $G$ be a connected compact Lie group with $\pi_1 G$ torsion free. Choose a maximal torus $\iota \colon U(1)^r \hookrightarrow G$ with the Weyl group $W$, and identify $\moduli_\C^G \simeq \moduli_\C^T/W$. 
		Given an element $\xi \in [BG, P^4 BO]$, we define the line bundle $\mathcal{A}(\xi)$ over $\moduli_\C^G $ by the following: 
		Consider the line bundle $\mathcal{A}(\iota^* \xi)$ over $\moduli_{\C}^{U(1)^{r}}$ constructed in (2), and observe that the $W$-action on $\moduli_{\C}^{U(1)^{r}}$ naturally lifts to $\mathcal{A}(\iota^* \xi)$. Thus it descends to a line bundle $\mathcal{A}(\xi)/W$ on $\moduli_\C^{G}$, which we denote by $\cA(\xi)$.\footnote{We note that the definition does not depend on the choice of the maximal torus, in the sense that, any two different maximal torus are conjugate to each other, and a choice of a conjugating element associates isomorphism of the line bundle constructed.}
		\item We also extend the notation to denote, given $n + \xi := (n, \xi) \in \Z \times [BG, P^4\BO]$, 
		\begin{align}
			\cA(n + \xi) := \cA(\xi) \otimes p^*\omega^{-\frac{n}{2}}, 
		\end{align}
		where $p \colon \moduli_\C^G \to \moduli_\C$ is the projection and $\omega$ is the cotangent sheaf on $\moduli_\C$. 
		For a virtual representation, $V \in \RO(G)$ we denote $\cA(V) := \cA(\tw(V))$, where $\tw(V)$ is defined in Section \ref{subsec_notations} \eqref{notation_tw}
	\end{enumerate}
\end{defn}
This means that a holomorphic section $\phi \in \Gamma(\cE_\C^{\times r}, \cA(\xi))$ can be written as a multivariable function $\phi(z_1, \cdots, z_r, \tau)$ with $(z_1, \cdots, z_r, \tau) \in \C^r\times \mathfrak{H}$ and the transformation rule induced by \eqref{eq_cA_def} + the weight factor appearing in the definition of Modular Forms.
We also use the coordinate $(y_1, y_2, \cdots, y_r, q)$ with $y_a := e^{2 \pi i z_a}$ and $q := e^{2\pi i \tau}$ interchangeably. 
A holomorphic section $ \phi \in \Gamma(\moduli_\C^G; \cA(\xi)) \simeq \Gamma(\cE_\C^{\times r}; \cA(\iota^*\xi))^W$ is expressed as those $\phi(\bm{z}, \tau)$ that are additionally invariant under the action of $W$. 

\begin{defn}[{multi-variable Jacobi Forms and $G$-equivariant Modular Forms}]\label{def_JF_MFG}
	\begin{enumerate}
		\item Let $r$ be a positive integer. Given a class $\xi  \in \Z \times [BU(1)^r, P^4BO]$, we define a $\Z$-graded $\MF$-module $\MF[{\xi}]^{U(1)^r}$ by setting
		\begin{align}
 \left( 	\MF[{{\xi}}]^{U(1)^r}\right)_{\deg=m}:=  \Gamma(\cE_\C; \cA(-m +\xi)) \cap \Z((y_1, y_2, \cdots, y_r, q)), 
		\end{align}
		for each $m\in\Z$. Here we have used the coordinates $y_a := e^{2 \pi i z_a}$ and $q := e^{2\pi i \tau}$ as above. 
		In the case $r=1$, we also denote
		\begin{align}\label{eq_def_JFk}
			\JF_k :=  \MF[\tw(kV_{U(1)})]^{U(1)} = \MF[2k + k\xi_{U(1)}]^{U(1)}, 
		\end{align}
		where $\xi_{U(1)} \in [BU(1), P^4 BO] \simeq \Z$ is the generator represented by the (normalized) fundamental representation $\overline{V}_{U(1)}$. 
		Following the usual convention, we call an element in $\JF_k|_{\deg=m} =\left( \MF[k \xi_{U(1)}]^{U(1)}\right) _{\deg = m-2k}$ an {\it integral Jacobi Form of index $\frac{k}{2}$ and weight $\frac{m}{2}-k$}. 
		
		\item Let $G$ be a compact connected Lie group with $\pi_1 G$ torsion-free. Choose a maximal torus $\iota \colon U(1)^r \hookrightarrow G$ with the Weyl group $W$. 
		Given an element $\xi \in [BG, P^4 BO]$, we define a $\Z$-graded $\MF$-module $\MF[{\xi}]^{G}$ by setting
		\begin{align}
\MF[\xi]^G := \left( \MF[\iota^* \xi]^{U(1)^r} \right)^W \subset \Gamma(\moduli_\C^G; \cA(\xi)),   
		\end{align}
		where $\left( -\right)^W$ means the $W$-invariant part.  
		For $V \in \RO(G)$, we also denote $\MF[V]^G := \MF[\tw(V)]^G$. 
		We call an element in $\MF[\xi]^G$ an integral $G$-equivariant $\xi$-twisted Modular Form. 
	\end{enumerate}
\end{defn}

The relation between $\cL_\C(-V)$ and $\cA(V)$ is the following. 
\begin{lem}[{\cite{ando2010circleequivariant} and \cite{gukov2025newapproach31dimensionaltqfts}}]\label{lem_AvsL}
For each compact connected $G$ with $\pi_1 G$ torsion-free and $V \in \RO(G)$, we have an isomorphism
	\begin{align}\label{eq_theta_V}
	\Phi_V \cdot \colon	\mathcal{L}_\C(-V) \simeq \cA(V) \mbox{ in } \Pic(\moduli_\C^G), 
	\end{align}
	equivalently an invertible holomorphic section $\Phi_V \in \Gamma(\moduli_\C^G; \cL_\C(V) \otimes \cA(V))^\times$, characterized by the following properties: 
	\begin{itemize}
		\item Functorial in $(G, V)$. 
		\item compatible with the monoidal structure in $\RO(G)$.  
		\item In the case $G=U(1)$ and $V = V_{U(1)}$, the section $\Phi_{V_{U(1)}} \in \Gamma(\cE_\C; \cO_{\cE_\C}(-e) \otimes \cA(V_{U(1)}))^\times$ is given by the Jacobi theta function as
		\begin{align}
			\Phi_{V_{U(1)}} = a = \phi_{-1, \frac12} = 	\frac{\theta_{11}(z, q)}{\eta^3(q)}  
			=(e^{\pi i z} - e^{-\pi i z}) \prod_{m \ge 1} \frac{(1-q^m e^{2\pi i z})(1-q^me^{-2\pi i z})}{(1-q^m)^2}. 
		\end{align}
		Here the notation $a$ follows our shorthand notation introduced in \eqref{eq_notation_a}. 
	\end{itemize}
\end{lem}

The map \eqref{eq_evC} factors through integral $G$-equivariant Modular Forms as
\begin{align}\label{eq_character}
		\red_\C \colon \pi_\bullet \TMF[V]^{G} \xrightarrow{e} \left( \MF[V]^G\right)  |_{\deg = \bullet} \subset  \Gamma(\moduli_\C^G; \cA(V) \otimes p^*\omega^{\bullet/2})  \stackrel{\Phi_V}{\simeq} 
		 \Gamma(\moduli_\C^G; \cL_\C(-V) \otimes p^*\omega^{\bullet/2}). 
\end{align}
We call the first map $e$ as the {\it $G$-equivariant character map}. 

\begin{rem}[{The relation between Euler class $\chi(V) \in \TMF[V]^G$ and $\Phi_V$}]\label{rem_PhiV=chiV}
	In the case where $V \in \Rep_O(G)$, i.e., $V$ is not virtual but a genuine representation, we have a natural map $\cL_\C(V) \to \cO_{\moduli_\C^G}$ in $\QCoh(\moduli_\C^G)$ by applying the $G$-equivariant elliptic cohomology functor \eqref{eq_EllG} to the map $\chi(V) \colon S^0 \hookrightarrow S^V$. 
	We abuse the notation to also denote by $\Phi_V \in \MF[V]^G \subset \Gamma(\moduli_\C^G; \cA(V))$ the section corresponding to the composition
	\begin{align}
		\cO_{\moduli_\C^G} \xrightarrow{} \cL_\C(-V) \stackrel{\eqref{eq_theta_V}}{\simeq} \cA(V)
	\end{align}
	For example, we regard $\Phi_{V_{U(1)}} = a \in \pi_0 \JF_1$. In general, $\Phi_V$ is essentially the generalization of the Theta functions studied in \cite{ando2010circleequivariant}.  
	This means that the $G$-equivariant Euler class $\chi(V) \in \pi_0 \TMF[V]^G$ in \eqref{eq_def_chiV} satisfies
	\begin{align}\label{eq_Phi_V=chiV}
		e(\chi(V)) = \Phi_V. 
	\end{align}
	The Euler class $\chi(V)$ is of particular importance in our paper. Physically, it is supposed to correspond to ``$G$-symmetric $V$-valued Majorana fermions''. 
\end{rem}

\begin{ex}[$G=U(n)$]\label{ex_U(n)_JF}
	In the case of $G=U(1)$, the character map \eqref{eq_character} becomes
	\begin{align}
		e_\JF \colon \pi_\bullet \TJF_k \to   \JF_k|_{\deg=\bullet}, 
	\end{align}
	which allows us to regard $\TJF_k$ as spectral refinement of $\JF_k$ as promised in Example \ref{ex_TJF_preliminary}. 
	More generally, for $G=U(n)$, we use the standard diagonal maximal torus $U(1)^n \xhookrightarrow{\rm diag} U(n)$ with the Weyl group $W=\Sigma_n$, the symmetric group permuting the factors. So a $U(n)$-equivariant Modular Forms are expressed as $n$-variable Jacobi Forms $\phi(z_1, \cdots, z_n, \tau)$ which are symmetric in $z_i$. 
	For any nonnegative integer $k$, we have (see \eqref{eq_tensor_flat})
	\begin{align}\label{eq_MFU(n)}
		\MF[kV_{U(n)}]^{U(n)} \otimes \Q  = \left( \bigotimes^{\MF_\Q}_{1 \le i \le n}\JF_k \otimes \Q \right)^{\Sigma_n} 
	\end{align}
	where the tensor product is formed over $\MF$. 
\end{ex}

\begin{ex}[$G=Sp(1)$: Even Jacobi Forms]\label{ex_EJF}
In the case of $G = Sp(1) = SU(2)$, we choose a maximal torus $T = U(1) \subset Sp(1)$. Then the Weyl group $W = \Z/2$ acts on $\moduli_\C^T \simeq \mathcal{E}_\C^\vee$ by the inverse involution of abelian varieties; in terms of the coordinate $(z, \tau) \in \C \times \mathfrak{H}$, the involution becomes $(z, \tau) \mapsto (-z, \tau)$. Thus the $SU(2)$-equivariant Modular Forms are identified as the Jacobi Forms that are {\it even} in $z$; so we employ the following notation: 
\begin{align}
\EJF_{2k} &:= \MF[kV_{Sp(1)}]^{Sp(1)} = \{\phi(z, \tau)\in  \JF_{2k} \ | \ \phi(z, \tau) = \phi(-z, \tau) \}. 
\end{align}
See Appendix \ref{app_TEJF} for more detailed descriptions. 
The $Sp(1)$-equivariant character map \eqref{eq_character} becomes
\begin{align}\label{eq_char_EJF}
	e_\EJF \colon \pi_\bullet \TEJF_{2k} \stackrel{\mbox{\tiny Def. }\ref{def_TEJF_app} }{:=} \pi_\bullet \TMF[kV_{Sp(1)}]^{Sp(1)} \to \EJF_{2k}|_{\deg=\bullet}, 
\end{align}
verifying our notation $\TEJF_{2k}$. 
\end{ex}

\begin{ex}[$G=Sp(n)$]\label{ex_Sp(n)}
	More generally, in the case of $G=Sp(n)$, we choose the maximal torus to be $U(1)^n \xhookrightarrow{\rm diag} U(n) \hookrightarrow Sp(n)$, the image of the standard maximal torus of $U(n)$ under the canonical inclusion $U(n) \hookrightarrow Sp(n)$. 
	Then the Weyl group is $W = (\Z/2)^n \rtimes \Sigma_n$, where each $\Z/2$ flips the sign of the coordinate $z_i \mapsto -z_i$ and $\Sigma_n$ permutes the factors. 
    Hence, $Sp(n)$-equivariant Modular Forms are regarded as $U(n)$-equivariant Modular Forms that are even in each variable $z_i$. 
\end{ex}

\begin{ex}[$G=SU(n)$]\label{ex_SU(n)_JF}
	For $G=SU(n)$, we follow the conventional approach that, rather than using the maximal torus of $SU(n)$, we first regard $SU(n) \subset U(n)$ and use the maximal torus $U(1)^{n} \hookrightarrow U(n)$ to identify
	\begin{align}
	\moduli_\C^{SU(n)} = \moduli_\C^{U(n)} \cap \{z_1 + z_2 + \cdots + z_n = 0\}. 
	\end{align}
	This means that we have (see \eqref{eq_MFU(n)})
	\begin{align}
	\MF[kV_{SU(n)}]^{SU(n)} \otimes \Q =\left( \frac{\bigotimes^{\MF_\Q}_{1 \le i \le n} \JF_k \otimes \Q}{(x_1+x_2+\cdots +x_n)} \right)^{\Sigma_n}
	\end{align}
\end{ex}

\subsection{Genuinely equivariant refinement of the sigma orientation}\label{subsec_genuine_sigma}

In \cite{ando2010multiplicative}, an $E_\infty$ ring map
\begin{align}\label{eq_AHR}
    \sigma \colon \MString \to \TMF
\end{align}
was constructed and called the {\it sigma orientation} of $\TMF$. 
In this article, we use an equivariant refinement of the sigma orientation which we now explain. 
In order to state it, first let us set the notation. Let $f \colon \mathcal{B} \to BO$ be a continuous map. Given a compact Lie group $G$ with a virtual representation $V \in \RO(G)$, a {\it $(\mathcal{B}, f)$-structure} $\fraks$ on $V$ is a lift of the classifying map $\overline{V} \colon BG \to BO$ to $\mathcal{B}$ along $f$\footnote{It is equivalent to the stable $(\cB,f)$-structure on the associated virtual vector bundle $EG \times_G V \to BG$ in the sense of Definition \ref{def_Bfstr}. }. We are particularly interested in {\it string structures}, which is classified by the map $\varrho \colon BString = BO \langle 8 \rangle \to BO$. 

First, recall the Thom isomorphism in $\TMF$ induced by the usual sigma orientation. Consider the following map, 
\begin{align}\label{eq_sigma_clas}
 \th \colon \cS_{/BO} \xrightarrow{} \Mod_\TMF, \quad (\theta \colon X \to BO) \mapsto \lMap(X^\theta , \TMF). 
\end{align}
The sigma orientation \eqref{eq_AHR} induces a natural isomorphism, also denoted by $\sigma$, in the following diagram, 
\begin{align}\label{diag_sigma_htpy}
	\xymatrix@C=10em{
	\cS_{/BString} \ar[r]^-{\fgt = (BString \to \pt)_*} \ar[d]^-{\varrho_*} & \cS \ar[d]^-{X \mapsto \lMap(\Sigma^\infty X_+ , \TMF)}  \ar@{}[ld]|{\rotatebox[origin=c]{210}{$\Longrightarrow$}}_-{\sigma}^-{\simeq} \\
\cS_{/BO} \ar[r]^-{\th}_-{\eqref{eq_sigma_clas}} & \Mod_\TMF. 
	}
\end{align}
This homotopy is equivalent to the data of the functorial assignment of the Thom isomorphism for string-oriented vector bundles. 

Now we introduce our formulation of the sigma orientation for the genuinely equivariant setting. 
For each compact Lie group $G$, recall that we have defined $\RO(G)$ to be the {\it groupoid} consisting of virtual orthogonal $G$-representations and isomorphisms. Let $\RO^{d=0}(G)$ denote the full subgroupoid consisting of those with virtual dimension $0$. Now define the groupoid $\RString^{d=0}(G)$ to be the pullback,
\begin{align}
	\xymatrix{
		\RString^{d=0}(G) \ar[r] \ar[d]^{\varrho} & \Map(BG, BString) \ar[d] \\
		\RO^{d=0}(G) \ar[r] & \Map(BG, BO). 
	}
\end{align}
This gives us functors
\begin{align}
\RString^{d=0}, \RO^{d=0}  \colon \cptLie^\op \to \Gpds, 
\end{align}
where $\Gpds$ is the category of groupoids. 
For any subcategory $\frC \subset \cptLie$, we can perform the Grothendieck construction, 
\begin{align}
	\int_\frC \RO^{d=0}, \quad  \int_\frC \RString^{d=0} \in \mathrm{Cat} 
\end{align}
The former is the groupoid whose objects are pairs $(G, V)$ with $G \in \frC$ and $V \in \RO^{d=0}(G)$, and morphism $(G, V) \to (H, W)$ consists of a pair $(f, \psi)$ where $f \colon G \to H$ is a group homomorphism in $\frC$ and $\psi \colon V \simeq \res_f W $ in $\RO^{d=0}(G)$. 
The latter is the groupoid whose objects are triples $(G, V ,\fraks)$ where $\fraks$ is a string structure on $V \in \RO^{d=0}(G)$, and morphisms are $(f, \psi)$ as above where $\psi$ is required to be compatible with the string structures. 

\begin{defn}[{A sigma orientation on a subcategory $\frC \subset \cptLie$}]\label{def_sigma_cat}
	Let $\frC \subset \cptLie$ be a subcategory. 
	A {\it sigma orientation} on $\frC$ is a natural isomorphism $\widetilde{\sigma}$ in the following diagram of categories, 
	\begin{align}\label{eq_tildesigma_def}
		\xymatrix@C=10em{
			\int_\frC \RString^{d=0} \ar[r]^-{(G, V, \fraks ) \mapsto G} \ar[d]^-{\varrho} & \frC \ar[d]^-{G \mapsto \TMF^G = \Gamma(\Ell(BG), \cO)}  \ar@{}[ld]|{\rotatebox[origin=c]{210}{$\Longrightarrow$}}_-{\widetilde\sigma \quad }^-{\simeq} \\
				\int_\frC \RO^{d=0} \ar[r]^-{\Th}_-{\eqref{eq_Th}} & \mathrm{Ho}(\Mod_\TMF). 
		}
		\end{align}
			where $\Th$ is defined by
			\begin{align}\label{eq_Th}
				\Th \colon\int_\frC \RO \xrightarrow{} \Mod_\TMF, \quad (G, {V}) \mapsto \Gamma(\Ell(\bfB G), \cL(V)) \simeq \Gamma(\moduli, \Ell(S^V //G)). 
			\end{align}
We require it to be compatible with the natural isomorphism $\sigma$ in \eqref{diag_sigma_htpy} via the Atiyah-Segal completion map $\zeta$ in \eqref{eq_AS}. More precisely, this condition is stated as follows. Consider the following diagram, 
		\begin{align}\label{diag_big_sigma}
			\xymatrix{
				\cS_{/BString} \ar[rrr]^-{\fgt } \ar[ddd]^-{\varrho_*} & &\ar@{}[dl]|\circlearrowright& \cS  \ar@/^3ex/[ddl]^-{X \mapsto \lMap(\Sigma^\infty X_+ , \TMF)}   \\ \ar@{}[rd]|\circlearrowright
				&	\int_\frC \RString^{d=0} \ar[r]^-{} \ar[d]^-{\varrho} \ar[lu]|-{(G, V, \fraks) \mapsto (BG \xrightarrow{V, \fraks}BString)} & \frC \ar@{}[r]|{\rotatebox[origin=c]{210}{$\Longrightarrow$}}_-{\zeta} \ar[d]^-{ }  \ar@{}[ld]|{\rotatebox[origin=c]{210}{$\Longrightarrow$}}_-{\widetilde\sigma}^-{ \simeq} \ar[ru]^-{G \mapsto BG}& \\
				&\int_\frC	\RO^{d=0} \ar@{}[d]|{\rotatebox[origin=c]{210}{$\Longrightarrow$}}^-{\zeta} \ar[r]^-{\Th} \ar[ld]|-{(G, V) \mapsto (BG \xrightarrow{V} BO)}& \mathrm{Ho}(\Mod_\TMF) & \\
				\cS_{/BO} \ar@/_3ex/[rru]_-{\th}&&& . 
			}
		\end{align}
		Here the middle square is \eqref{eq_tildesigma_def},
		 and the top and left square canonically commutes. The remaining two triangles are not commutative but equipped with the natural transformation by \eqref{eq_AS} as indicated. 
		We require that, the natural transformation between the two outer compositions $\int_\frC \RString^{d=0}\to \mathrm{Ho}(\Mod_\TMF )$, obtained by composing the natural transformations in \eqref{diag_big_sigma}, conicides with the natural isomorphism obtained by composing the leftup arrow in \eqref{diag_big_sigma} with the natural isomorphism $\sigma$ in \eqref{diag_sigma_htpy}. 
	\end{defn}

\begin{rem}\label{rem_sigma_concrete}
The data of sigma orientation in the Definition \ref{def_sigma_cat} can be concretely understood as follows.	For each element $G \in \frS$ and each virtual representation $V \in \RO(G)$ equipped with a string structure $\fraks$ on $V$, the equivalence of $G$-equivariant $\TMF$-module spectra, 
	\begin{align}
		\sigma(V, \fraks) \colon \TMF[\overline{V}] \simeq \TMF. 
	\end{align}
	is assigned (up to homotopy), and this assignment satisfies the following. 
	\begin{enumerate}
		\item functoriality in $G \in \frS$.
		\item compatiblility with the monoidal structure in $\RO(G)$. 
		\item compatibility with the usual Thom isomorphism induced by the sigma orientation after the Atiyah-Segal completion. \label{compatibility_AS_sigma}
	\end{enumerate}

\end{rem}

The following statement is proved in an upcoming paper \cite{MeierYamashita} by L. Meier and the second author of this article. 
\begin{fact}[{\cite{MeierYamashita}}]\label{fact_sigma}
	There exists a full subcategory $\frS \subset \cptLie$ with a preferred string orientation (in the sense of Definition \ref{def_sigma_cat}), which satisfies
	\begin{enumerate}
	\item $\frS$ contains $U(1)^n$, $SU(n)$, $Sp(n)$ and $U(n)$ for all $n$. 
	\item $\frS$ is closed under taking finite products. 
		\end{enumerate}
	
\end{fact}

The authors expect the following conjecture to be true.

\begin{conjc}[{Conjecture on equivariant sigma orientation}]\label{conj_sigma}
	There exists a sigma orientation on the whole category $\cptLie$ (in the sense of Definition \ref{def_sigma_cat}). Moreover, there is a preferred choice of sigma orientation, which restricts to the sigma orientation on $\frS$ supplied by Fact \ref{fact_sigma}. 
\end{conjc}

The formalism of this paper works once for all we fix a sigma-oriented subcategory $\frC \subset \cptLie$. We are basically based on the subcategory $\frS \subset \cptLie$ with the sigma orientation given in Fact \ref{fact_sigma}, and derive mathematical results at the current status. However, we also would like to present the results we can get once we assume the whole establishment of the equivariant sigma orientation, Conjecture \ref{conj_sigma}; if we assume that, we can get rid of technical restrictions and get a complete and unified picture of our topological elliptic genera. Therefore, in this paper we put \colorbox{conjcolor}{shaded backgrounds} on the statements and proofs which depend on Conjecture \ref{conj_sigma}. 

\begin{rem}
	The authors believe that the difficulties which are currently preventing us from fully establishing the equivariant sigma orientation is only technical, and Conjecture \ref{conj_sigma} should be eventually proved. We plan to update this article as we progress on the equivariant sigma orientation, and hoping that we completely remove \colorbox{conjcolor}{the shade} soon. 
\end{rem}

\subsubsection{A remark on $\RO(G)$-grading versus $G$-equivariant twists}\label{subsubsec_twist}

In general, for genuinely $G$-equivariant commutative ring spectrum, the $\RO(G)$-grading is naturally regarded as special cases of $G$-equivariant twists generally classified by $\Pic(\Mod_E)$. Namely, we have the map
\begin{align}
	\RO(G) \to \Pic(\Mod_E) , \quad \tau \mapsto E \otimes S^\tau. 
\end{align}
In the case of $E=\TMF$, non-equivariantly we have a map \cite{ando2010twists}
\begin{align}
P^4 BO \to BGL_1(\TMF), 
\end{align}
which allows us to twist $\TMF$-comomology by a map to $P^4 BO$. 
It is widely expected, from mathematical point of views \cite{lurie2009survey} as well as physical point of views \cite[Appendix A]{tachikawa2023anderson} \cite{LinYamashita2}, that the twists by $P^4 BO$ canonically refines to the twists of genuinely equivariant $\TMF$. More precisely we expect that there is a map $\mathfrak{t}_G \colon \Map(BG, P^4 BO) \to \Pic(\Ell(\bfB G))$, functorial in $G$, which makes the following diagram commute
\begin{align}\label{diag_twist}
	\xymatrix@C=7em{
\RO(G) \ar[r]^-{\tau \mapsto \cL(V)} \ar[d]^-{\tw} & \Pic(\Ell(\bfB G)) \ar[d]^-{u_{\bfB G}} \\
\Z \times \Map(BG, P^4 BO) \ar@{.>}[ru]^-{\mathfrak{t}_G} \ar[r]_-{\mbox{\tiny \cite{ando2010twists}}} & \Pic(\Ell(BG)) \simeq \Map(BG, \Pic(\TMF))
	}
\end{align}
Note that this claim is stronger than Conjecture \ref{conj_sigma}; indeed, Conjecture \ref{conj_sigma} follows by the commutativity of the diagram \eqref{diag_twist}, but the existence of the map $\mathfrak{t}_G$ implies that we can twist genuinely equivariant $\TMF$ by maps $BG \to P^4 BO$ which does not come from $\RO(G)$. 

Then, a natural question is how much of the expected twists come from $\RO(G)$-grading. Fortunately, for $G = \Z/p, U(1)^n, U(n), SU(n), Sp(n), O(n), SO(n), Spin(n)$, the map
\begin{align}\label{eq_RO_tw}
	\tw \colon \RO(G) \to \Z \times[BG, P^4 BO]
\end{align}
is surjective, so all the expected twists are realized by $\RO(G)$-gradings up to equivalence. 
On the other hand, for example in the case $G = E_8$, the map \eqref{eq_RO_tw} is known to be non-surjective.

\subsection{On tangential and normal Thom spectra}\label{subsec_thom}

In this article, it is important to distinguish {\it tangential} and {\it normal} bordism Thom spectra, and also to distinguish {\it stable} and {\it strict $=$ unstable} structures, as we now explain. For a detailed account, we refer to \cite[Section 6.6]{FreedLecture}. 
We follow the notation in Section \ref{subsec_notations} \eqref{notation_Thom} to denote by $X^\theta$ the Thom spectrum associated to a virtual vector bundle $\theta$ over a space $X$. 
As written in (\ref{notation_MT}) there, for a map $f \colon \cB \to BO$ which is regarded as a virtual vector bundle with rank $0$, we denote
\begin{align}
	M(\cB, f) := \cB^f, \quad MT(\cB, f) := M(\cB, -f) = \cB^{-f}, 
\end{align}
and call them the {\it normal Thom spectrum} and the {\it tangential Thom spectrum}, respectively. These notations are justified below. 
When $\cB$ is of the form $\cB = BH$ with a compact Lie group $H$, we also use the conventional notation 
\begin{align}
	M(G, f) := M(BG, f), \quad MT(G, f) := MT(BG, f) . 
\end{align}

We employ the following general definition of {\it stable} tangential and normal structures. 

\begin{defn}[{stable ($\cB, f$)-structures and bordism groups}]\label{def_Bfstr}
	Suppose we are given a space $\cB$ with a map $f \colon \cB \to BO$.  
	\begin{itemize}
		\item For a space $X$ with a virtual vector bundle $\theta$, a stable {\it $(\cB, f)$-structure} $\fraks$ on $\theta$ is a map of spectra\footnote{Note that giving a map $X^\theta \to \cB^f$ is equivalent to giving a map $X^{-\theta} \to \cB^{-f}$ since $BO$ is an infinity loop space. \label{footnote_BO}}, 
	\begin{align}
			\fraks \colon X^{\overline{\theta}} \to \cB^{f} :=M(\cB, f). 
		\end{align}
		\item For a manifold $M$ with tangent bundle $TM$, 
		\begin{itemize}
			\item a {\it stable tangential $(\cB, f)$-structure} is a $(\cB, f)$-structure on $TM$. 
			\item a {\it stable normal $(\cB, f)$-structure} is a stable $(\cB, -f)$-structure on $TM$, equivalently a $(\cB, f)$-structure on $(-TM)$ (see footnote \ref{footnote_BO}). 
		\end{itemize}
		\item We denote by $\Omega^{(\cB, f)}_m$ the bordism group of closed $m$-dimensional manifolds with {\it stable tangential} $(\cB, f)$-structures, and by $\Omega^{(\cB, f)^\perp}_m$ the bordism group of those manifolds with {\it stable normal} $(\cB, f)$-structures. 
	\end{itemize}
\end{defn}

We also utilize the notion of {\it strict $=$ unstable} structures. 
We employ the following definition. 
\begin{defn}[{strict $(\cB(d), f)$-structures}]\label{def_strict_str}
	Let $n$ be a nonnegative integer, and suppose we are given a space $\cB(d)$ with a Serre fibration $f \colon \cB(d) \to BO(d)$. 
	\begin{itemize}
		\item For a space $X$ with a vector bundle $\theta$ of real rank $n$, a {\it strict} $(\cB(d), f)$-structure $\fraks$ on $\theta$ is a map $\fraks \colon X \to \cB(d)$ which makes the following diagram commute. 
		\begin{align}
			\xymatrix{
				& \cB(d) \ar[d]^-{f} \\
				X \ar@{.>}[ru]^-{\fraks} \ar[r]_-{\overline{\theta}} & BO(d) . 
			}
		\end{align}
		\item For a manifold $M$ with tangent bundle $TM$, A {\it strict tangential $(\cB(d), f)$-structure} is a $(\cB(d), f)$-structure on $TM$. \footnote{The existence of a strict tangential $(\cB(d)), f)$-structure in particular implies that $\dim_\R M = d$. }
	\end{itemize}
\end{defn}
Of course, a strict $(\cB(d), f)$-structure canonically induces a stable $(\cB(d), f)$-structure, where we abuse the notation to denote by $f$ the composition $\cB(d) \xrightarrow{f} BO(d) \xrightarrow{n \to \infty} BO$. 

An important class of structures in this paper are those of the form $(\cB(d), f) = (BH, {V})$, where $H$ is a compact Lie group and $V \in \Rep_O(H)$ is a real representation of dimension $d$, with induced map $\overline{V} \colon BH \to BO(d)$. 
In this case, unpacking the above definition, we can concretely understand the {\it stable} tangential and normal structures as follows. Let $M$ be an $m$-dimensional manifold. 
\begin{itemize}
	\item A {\it stable} tangential $(BH, V)$-structure on $M$ is represented by a pair $(P, \psi)$, where $P$ is a principal $H$-bundle over $M$, and $\psi$ is an isomorphism of vector bundles over $M$, 
	\begin{align}\label{eq_prelim_tangential}
		\psi \colon TM \oplus \underline{\R}^N \simeq (P \times_H V) \oplus \underline{\R}^{m+N-d}
	\end{align}
	where $N$ is a large enough integer. 
	\item A {\it stable }normal $(BH, V)$-structure on $M$ is represented by a pair $(P, \psi)$, where $P$ is a principal $H$-bundle over $M$, and $\psi$ is an isomorphism of vector bundles over $M$, 
	\begin{align}
	\psi \colon	TM \oplus (P \times_H V) \oplus \underline{\R}^N \simeq   \underline{\R}^{m+d+N}. 
	\end{align}
	where $N$ is a large enough integer. 
	\item A {\it strict} tangential $(BH, V)$-structure on $M$ exists only when $m=d$, and is represented by a pair $(P, \psi)$, where $P$ is a principal $H$-bundle over $M$, and $\psi$ is an isomorphism of vector bundles over $M$, 
	\begin{align}
		\psi \colon TM  \simeq (P \times_H V) . 
	\end{align}
\end{itemize}

Notice that, by the above definition, it makes sense to talk about a stable tangential $SU(k)$-structure on an $m$-dimensional manifold with $2k < m$.

Important cases of $(H, V \in \Rep_O(H))$ come in series, $\{(\cH(k), V_{\cH(k)})\}_{k}$, where examples include $\cH=U, SU, O, Sp, Spin$ with their {\it fundamental} representations. For these cases, we simply call a tangential/normal $(B\cH(k), \overline{V}_{\cH(k)})$-structure a tangential/normal $\cH(k)$-structure, respectively, and denote
\begin{align}
	M\cH(k) := M(\cH(k), \overline{V}_{\cH(k)}) = B\cH(k)^{\overline{V}_{\cH(k)}}, ~~ MT\cH(k) := MT(\cH(k), \overline{V}_{\cH(k)}) = B\cH(k)^{-\overline{V}_{\cH(k)}}.
\end{align}
These representations stably restrict to each other by the inclusions $\cH(k) \subset \cH(k+1)$, so that we have the stabilization sequences,
\begin{align}\label{seq_stab}
\cdots \xrightarrow{\stab} M\cH(k-1) \xrightarrow{\stab} &M\cH(k) \xrightarrow{\stab} M\cH(k+1)  \xrightarrow{\stab} \cdots \\
	\cdots \xrightarrow{\stab} MT\cH(k-1) \xrightarrow{\stab} &MT\cH(k) \xrightarrow{\stab} MT\cH(k+1)  \xrightarrow{\stab} \cdots
\end{align}
and a tangential/normal $\cH(k)$-structure canonically induces a tangential/normal $H(k')$-structure for $k' > k$, respectively. 
More generally we also use the stabilization sequence
\begin{align}\label{eq_stab_MTH_general}
    \cdots \xrightarrow{\stab} MT(\cH(k-1), n\overline{V}_{\cH(k-1)}) \xrightarrow{\stab} &MT(\cH(k), n\overline{V}_{\cH(k)}) \xrightarrow{\stab} MT(\cH(k+1), n\overline{V}_{\cH(k+1)})  \xrightarrow{\stab} \cdots
\end{align}
for each integer $n \in \Z$. 
Taking the colimit of the above stablilization sequences, we define
\begin{align}
	M\cH := \varinjlim_k M\cH(k), \quad MT\cH := \varinjlim_k MT\cH(k). 
\end{align}
and call the corresponding structures {\it $H$-structures}. 

\begin{rem}[stable versus unstable]\label{rem_stable_vs_unstable}
	The word ``stable'' needs to be taken with care, since there are two distinct senses of stability here. The notion of stability in Definition \ref{def_Bfstr} has nothing to do with the stabilizing sequence \eqref{eq_stab_MTH_general}. In other words, although the structure classified by $MT\cH(k)$ or $M\cH(k)$ could be regarded as {\it unstable} in the sense that we are not taking colimit of the stabilization sequence \eqref{seq_stab}, it is stable in the sense of Definition \ref{def_strict_str}.
This distinction is very important for us, since our main construction, the topological elliptic genus, is of the form, e.g., \eqref{eq_intro_Jac_U}
	\begin{align}
		\Jac_{U(1)_k} &\colon	MTSU(k) \to \TJF_k,
	\end{align}
	defined for each $k$, and detects sensitively the information that is lost after stabilization $k \to \infty$. 
\end{rem} 

\begin{rem}\label{rem_tangential_normal}
	It is important to distinguish tangential and normal structures. Typically, we have
	\begin{align}\label{eq_tang_nor_stab}
		M\cH \simeq MT\cH \quad \mbox{(for many cases)}
	\end{align}
	after stabilizing $k \to \infty$. This is the case for the examples listed above. However, we do have counterexamples, such as $MPin^+ \simeq MTPin^- $. 
	Moreover, it is important for us that, even if we have \eqref{eq_tang_nor_stab} after stabilization, we have
	\begin{align}
		M\cH(k) \not\simeq MT\cH(k) \quad \mbox{(for almost all cases!)}
	\end{align}
	for finite $k$. In fact, there is no natural map between $M\cH(k)$ and $MTH(k')$ for any pair $(k, k')$. 
\end{rem}

\begin{ex}\label{ex_spin_Sk}
Consider the manifold $S^k$ for an integer $k \ge 2$. On $S^k$, we can consider
\begin{itemize}
\item The stable tangential framing (i.e., the stable $(\cB, f) = (\pt, 0)$-structure) $\fraks_{\rm BB}^{\fr} := (P=\underline{e}, \psi_{\rm BB})$, commonly called the ``blackboard framing'', where $\psi_{\rm BB} \colon TS^k \oplus \underline\R \simeq  \underline{\R}^{k+1} $ is given by the standard embedding $S^k \hookrightarrow \R^{k+1}$. We have
\begin{align}
	[S^k , \fraks^\fr_{\rm BB}] = 0 \in \Omega^{fr}_k \simeq \pi_k S. 
\end{align} 
This stable tangential framing induces a stable tangential $(\cB, f)$-structure for any $(\cB, f)$ by the unit map $S \to MT(\cB, f)$. In particular, we get the {\it stable} $Spin(k)$-structure on $S^k$, which we denote by $\fraks_{\rm BB}^{Spin(k)}$. 
\item The strict tangential $Spin(k)$-structure which we denote by $\fraks^{Spin(k)}_{\rm str} = (P_{\rm str}, \psi_{\rm str})$. Here, we put the orientation of $S^k$ to coincide with the one induced by the blackboard one to get the strict tangential $SO(k)$-structure, and lift it uniquely to a strict tangential $Spin(k)$-structure using the fact that $\pi_1 S^k = \{*\}$. 
\end{itemize}
It is important to note that we have
\begin{align}
 [S^k, \fraks^{Spin(k)}_{\rm str}] \neq	[S^k, \fraks^{Spin(k)}_{\rm BB}] =0 \in \Omega_{k}^{Spin(k)}
\end{align}
Indeed, as a principal $Spin(k)$-bundle, $P_{\rm str}$ is not isomorphic to the trivial one. 
On the other hand, after the stabilization, we have
\begin{align}
\stab\left(  [S^k, \fraks^{Spin(k)}_{\rm str}]\right)  =  \stab \left(  [S^k, \fraks^{Spin(k)}_{\rm BB}] \right)   = 0\in \Omega_k^{Spin} \simeq \pi_k MTSpin \simeq \pi_k MSpin. 
\end{align}
We will come back to this example in Remark \ref{rem_strict_Euler}. 
\end{ex}

Now let us recall the Pontryagin-Thom isomorphism in this context. 
Given a closed manifold $M$, the {\it Pontryagin-Thom collapse map} is the map of spectra
\begin{align}\label{eq_def_coll}
	\coll \colon S \to M^{-TM}, 
\end{align}
which is defined by embedding $M$ into $\R^N$ for large enough $N$ and collapsing the complement of a tubular neighborhood. 
If furthermore $M$ is equipped with a stable {\it tangential} $(\cB,f)$-structure $\fraks$, we compose
\begin{align}\label{eq_PT_tangential}
	S \xrightarrow{\coll} M^{-TM} \xrightarrow{\fraks} \cB^{-f}[-m] := MT(\cB, f)[-m], 
\end{align}
to get an element in $\pi_m MT(\cB, f)$. 
On the other hand, if $M$ is equipped with a {\it normal} $(\cB, f)$-structure $\fraks^\perp$, we compose (see footnote \ref{footnote_BO})
\begin{align}\label{eq_PT_normal}
	S \xrightarrow{\coll} M^{-TM} \xrightarrow{\fraks^\perp} \cB^{f}[-m] := MT(\cB, f)[-m], 
\end{align}
to get an element in $\pi_m M(\cB, f)$. 

\begin{fact}[{Pontryagin-Thom isomorphism}]\label{fact_PT}
	The above procedure, called the Pontryagin-Thom construction, gives isomorphisms
	\begin{align}
\PT \colon	\Omega_m^{(\cB, f)} &\simeq 	\pi_m MT(\cB, f),  \quad [M, \fraks] \mapsto \eqref{eq_PT_tangential} \\
\PT \colon	 \Omega_m^{(\cB, f)^\perp} &\simeq	\pi_m M(\cB, f), \quad [M, \fraks^\perp] \mapsto \eqref{eq_PT_normal} . 
	\end{align}
\end{fact}
This justifies the terminology introduced in Section \ref{subsec_notations} \eqref{notation_MT}.

\section{The definitions of topological elliptic genera}
\label{sec_jacobi}
In this section we introduce our main construction, the {\it topological elliptic genera}. As explained in Introduction, we produce a class of maps of the form (here $G, H$ are compact Lie groups, $\tau_G \in \RO(G)$, $\tau_H \in \RO(H)$, and $\cD$ is the appropriate data explained in Section \ref{subsec_construction})
\begin{align}
	\Jac_{\cD} \colon MT(H, \tau_H) \to \TMF[\tau_G]^G
\end{align}
which refine the classical elliptic genera for $SU$-manifolds as well as the Witten-Landweber-Ochanine genus for Spin manifolds, and generalizes them further. 
This section is organized as follows. In Section \ref{subsec_JacU1}, as a warm-up to illustrate our ideas, we explain the construction in the most basic case $\Jac_{U(1)_k}$, which refines the classical elliptic genera $\Jac_\clas$ \eqref{eq_intro_Jac_clas}. 
Then in Section \ref{subsec_construction} we introduce the general construction. 

\subsection{The $U(1)$-topological elliptic genus $\Jac_{U(1)_k} \colon MTSU(k) \to \TJF_k$}\label{subsec_JacU1}

Here we introduce the construction of {\it $U(1)$-topological elliptic genus}, which is a map of spectra
\begin{align}
    \Jac_{U(1)_k} \colon MTSU(k) \to \TMF[kV_{U(1)}]^{U(1)}\simeq \TJF_k. 
\end{align}
Here $V_{U(1)}$ denotes the fundamental representation of $U(1)$. 
The left hand side is the tangential $SU(k)$-bordism spectrum in Section \ref{subsec_thom}, and the right hand side is the spectrum of Topological Jacobi Forms with index $\frac{k}{2}$, explained in detail in Appendix \ref{app_TJF}. 

Let us set the notation:
We denote by $V_{SU(k)}$ and $V_{U(k)}$ the fundamental complex representations of the indicated group. They are of real rank $2k$, but it is important that we can, and do, canonically regard them as complex representations of rank $k$. 
Let us consider the following representation of $U(1) \times SU(k)$ of real dimension $2n$, 
\begin{align}
    V_\phi := V_{U(1)} \otimes_\C V_{SU(k)} \in \Rep_O(U(1) \times SU(k)). 
\end{align}

The following proposition is crucial for our main construction. 
\begin{prop}\label{prop_key_JacU1}
    The virtual representation
    \begin{align}
        \overline{V}_{U(1)} \otimes_\C \overline{V}_{SU(k)} = (V_{U(1)} - \underline{\C}) \otimes_\C (V_{SU(k)} - k\underline{\C}) \in \RO(U(1) \times SU(k))
    \end{align}
    has a $BU\langle 6 \rangle$-structure $\mathfrak{s}$ (see \eqref{diag_Whitehead}, in particular it induces a string structure), and it is unique up to homotopy. 
\end{prop}
\begin{proof}
    There exists $BU\langle 6 \rangle$-structure because $c_i(\overline{V}_{U(1)} \otimes_\C \overline{V}_{SU(k)}) = 0$ for $i=1, 2$. 
    Moreover, since $H^{i}(BU(1) \times BSU(k); \Z) = 0$ for $i = 3, 5$, the choice of such a lift is unique up to homotopy. 
\end{proof}

\begin{rem}\label{rem_ku_connective}
    As explained in \cite{BauerYamashitaTJF}, another more conceptual definition of the $BU \langle 6 \rangle$-structure in Proposition~\ref{prop_key_JacU1} is to deduce from the multiplication on $\mathrm{ku}\langle\bullet\rangle$, 
\[
\mathrm{ku}\langle 2\rangle \wedge \mathrm{ku} \langle 4 \rangle \to \mathrm{ku} \langle 6 \rangle,
\]
    where $\mathrm{ku} \langle n \rangle$ denote the $n$-connective cover of $\mathrm{ku}$. 
\end{rem}

Let us denote
\begin{align}\label{def_Theta_U(1)}
    \Theta := \overline{V}_{U(1)} \otimes_\C \overline{V}_{SU(k)} . 
\end{align}
By Proposition \ref{prop_key_JacU1} and the equivariant sigma orientation (Fact \ref{fact_sigma}), we get an equivalence of $U(1) \times SU(k)$-equivariant $\TMF$-module spectra,
     \begin{align}
         \sigma(\Theta, \fraks)\colon \TMF[\Theta] \simeq \TMF . 
     \end{align}
Combining with the following equivalence in $\RO(U(1) \times SU(k))$,
$$\Theta \simeq  V_\phi - k \cdot \res_{U(1)}^{U(1) \times SU(k)}(V_{U(1)} ) - \res_{SU(k)}^{U(1) \times SU(k)}(\overline{V}_{SU(k)}), $$ we get the following equivalence of $\TMF$-modules, also denoted by the same symbol, 
\begin{align}
   \sigma(\Theta, \fraks) \colon \TMF[V_\phi]^{U(1) \times SU(k)} &\simeq \TMF[kV_{U(1)}]^{U(1)} \otimes_\TMF \TMF[\overline{V}_{SU(k)}]^{SU(k)}, \\
   &= \TJF_k \otimes_\TMF \TMF[\overline{V}_{SU(k)}]^{SU(k)}. 
\end{align}

The following is our main construction. 
\begin{defn}[{The coevaluation map $\mathcal{F}_{U(1)_k}$}]\label{def_F_U1}
    We define a morphism in $\Mod_\TMF$,
    \begin{align}\label{eq_F_U1_coev}
        \mathcal{F}_{U(1)_k} \colon \TMF \to \TJF_k \otimes_\TMF \TMF[\overline{V}_{SU(k)}]^{SU(k)},
    \end{align}
    to be the following composition. 
    \begin{align}
         \mathcal{F}_{U(1)_k}  
   \colon \TMF
    \xrightarrow{\chi(V_\phi) \cdot} \TMF[V_\phi]^{U(1) \times SU(k)}  \stackrel{\sigma(\Theta, \fraks)}{\simeq} \TJF_k \otimes_\TMF \TMF[\overline{V}_{SU(k)}]^{SU(k)} . 
    \end{align}
    Using the dualizability result \eqref{eq_twisted_dual} the $\TMF$-linear dual to $\TMF[\overline{V}_{SU(k)}]^{SU(k)}$ is canonically identified with $\TMF[-\overline{V}_{SU(k)}-\Ad(SU(k))]^{SU(k)}$. Thus the morphism \eqref{eq_F_U1_coev} is equivalently regarded as the following morphism,
    \begin{align}
        \mathcal{F}'_{U(1)_k} \colon \TMF[-\overline{V}_{SU(k)}-\Ad(SU(k))]^{SU(k)} \to \TJF_k
    \end{align}
\end{defn}

\begin{defn}[{The topological elliptic genus $\Jac_{U(1)_k}$}]\label{def_Jac_U(1)}
We define $\Jac_{U(1)_k}$ to be the composition
\begin{align}\label{eq_another_Jac}
\Jac_{U(1)_k} &\colon MTSU(k) = BSU(k)^{-\overline{V}_{SU(k)}} \simeq (S^{-\overline{V}_{SU(k)}})_{hSU(k)}\\
& \xrightarrow{u}  \TMF [-\overline{V}_{SU(k)}]_{hSU(k)}\\
&\xrightarrow{\Nm} \TMF[-\overline{V}_{SU(k)} - \Ad_{SU(k)}]^{SU(k)} \\
&\xrightarrow{\mathcal{F}'_{U(1)_k} } \TJF_k, 
\end{align}
where $u \colon S \to \TMF$ is the unit map. 
\end{defn}

An alternative definition is available as follows. 
\begin{prop}[{Alternative definition of $\Jac_{{U(1)}_k}$}]\label{prop_alternative_Jac}
	Consider the following map in $\Spectra^{U(1)}$: 
	\begin{align}\label{eq_alternative_1_U(1)}
		MTSU(k) = BSU(k)^{-\overline{V}_{SU(k)}} \xhookrightarrow{\chi(V_\phi) \cdot } BSU(k)^{V_\phi-\overline{V}_{SU(k)}}. 
	\end{align}
	Here, $MTSU(k)$ is regarded as a spectrum with trivial $U(1)$-equivariance, and $V_\phi=V_{U(1)} \otimes_\C V_{SU(k)}$ is regarded as a $U(1)$-equivariant vector bundle over $BSU(k)$. The map is given by the inclusion of the zero section of $V_\phi$. 
	After tensoring with $\TMF \in \Spectra^{U(1)}$, we get, again in $\Spectra^{U(1)}$, 
	\begin{align}\label{eq_alternative_2_U(1)}
	\eqref{eq_alternative_1_U(1)} \xrightarrow{u \otimes \id}	 \TMF \otimes BSU(k)^{V_\phi-\overline{V}_{SU(k)}} \stackrel{\sigma(\Theta, \fraks)}{\simeq}  \TMF \otimes BSU(k)_+ \otimes S^{kV_{U(1)}}, 
	\end{align} 
	by the $U(1)$-equivariant sigma orientation,
	since the virtual vector bundle $\Theta = \overline{V}_{U(1)} \otimes_\C \overline{V}_{SU(k)}$, regarded as a $U(1)$-equivariant virtual vector bundle over $BSU(k)$, is equipped with a $U(1)$-equivariant $BU\langle 6 \rangle$-structure $\fraks$ by Proposition \ref{prop_key_JacU1}. 
	Take the genuine $U(1)$-fixed point of the composition of \eqref{eq_alternative_1_U(1)} and \eqref{eq_alternative_2_U(1)}, and further consider the following: 
	\begin{align}\label{eq_alternative_3_U(1)}
		\xymatrix@C=5em{
		MTSU(k) \ar[rd]_-{\Jac_{U(1)_k}} \ar[r]^-{\eqref{eq_alternative_2_U(1)}  \circ \eqref{eq_alternative_1_U(1)} }  & \left( \TMF \otimes BSU(k)_+ \otimes S^{kV_{U(1)}}\right)^{U(1)}
		\ar[d]^-{(BSU(k) \to \pt)_*}  \\
		&\TMF[kV_{U(1)}]^{U(1)} =\TJF_k. 
		}
	\end{align}
	We claim that the diagram \eqref{eq_alternative_3_U(1)} commutes; i.e, we can take the composition in that diagram as an alternative definition of $\Jac_{U(1)_k}$. 
\end{prop}

\begin{proof}
	This directly follows from the definition of $\Jac_{U(1)_k}$. 
\end{proof}

\begin{rem}\label{rem_alternative_pro_con}
	Notice that the alternative definition of $\Jac_{U(1)_k}$ in Proposition \ref{prop_alternative_Jac} only use genuine equivariance with respect to $U(1)$ and not to $SU(k)$. Moreover, we do not use the dualizability of the genuinely equivariant $\TMF$. Nevertheless, we employ Definition \ref{def_Jac_U(1)} as the main definition because the coevaluation map $\cF_{U(1)_k}$ (Definition \ref{def_F_U1}) is essential in the level-rank duality we will explore in Section \ref{sec_duality}. 
\end{rem}

\begin{rem}[Geometric description of $\Jac_{U(1)_k}$]\label{rem_geom_JacU(1)}
	Having the alternative definition of $\Jac_{U(1)_k}$ in Proposition \ref{prop_alternative_Jac} at hand, we can easily get the following geometric description of the composition
	\begin{align}
	\Jac_{U(1)_k} \circ \PT \colon	\Omega_m^{SU(k)} \stackrel{\PT}{\simeq} \pi_m MTSU(k) \xrightarrow{\Jac_{U(1)_k}} \pi_m\TJF_k
	\end{align}
	as follows. Recall (Section \ref{subsec_thom}) that a class in $\Omega_m^{SU(k)}$ is represented by a data $(M, P, \psi)$ of closed $m$-dimensional manifold $M$ and a stable tangential $SU(k)$-structure $(P, \psi)$ on $M$. 
	Given such an $(M, P, \psi)$, denote by $V_P := P \times_{SU(k)}V_{SU(k)}$ be the associated bundle to the principal $SU(k)$-bundle $P$, and
	consider the following map of $U(1)$-equivariant Thom spectra, 
	\begin{align}\label{eq_geometric_Jac}
		S^{2k-m} \xrightarrow{\coll} \Sigma^{2k-m}M^{-TM} \stackrel{\psi}{\simeq} M^{-V_P} \xhookrightarrow{\chi(V_P \otimes_\C V_{U(1)})} M^{ V_P \otimes_\C V_{U(1)}- V_P }, 
	\end{align}
	where we are equipping $V_P$ with a trivial $U(1)$-action, and $V_P \otimes_\C V_{U(1)}$ is isomorphic to $V_P$ as a vector bundle but equipped with the nontrivial $U(1)$-action. 
	The first map in \eqref{eq_geometric_Jac} is the Pontryagin-Thom collapse map in \eqref{eq_def_coll}, and
	he last map is the inclusion of the zero section of $V_P \otimes_\C V_{U(1)}$. 
	Note that the following $U(1)$-equivariant virtual vector bundle over $M$, 
	\begin{align}
		\Theta_{P} :=\overline{V}_P \otimes_\C \overline{V}_{U(1)} = V_P \otimes_\C V_{U(1)} - V_P- k\overline{V}_{U(1)} 
	\end{align}
	is equipped with a $U(1)$-equivariant $BU\langle 6 \rangle$-structure $\fraks$ by using the $SU(k)$-structure on $V$ and Proposition \ref{prop_key_JacU1}. 
	Thus we have the Thom isomorphism in $U(1)$-equivariant $\TMF$-homology, 
	\begin{align}\label{eq_geometric_Jac_Theta}
	\TMF_*^{U(1)}(M^{V_P \otimes_\C V_{U(1)}- V_P }) \stackrel{\sigma(\Theta_P, \fraks)}{\simeq}  \TMF^{U(1)}_{*+2k}(M_+ \wedge S^{kV_{U(1)}}). 
	\end{align}
	We get the composition
	\begin{align}\label{eq_geometric_Jac_2}
		 \pi_0 \TMF^{U(1)} \xrightarrow{\eqref{eq_geometric_Jac_Theta} \circ \eqref{eq_geometric_Jac}} \TMF^{U(1)}_{(m-2k)+2k}(M_+ \wedge S^{kV_{U(1)}}) \\
			\xrightarrow{(M \to \pt)_*} \TMF^{U(1)}_{m}(S^{kV_{U(1)}}) = \pi_m \TJF_k
	\end{align}
It directly follows from Proposition \ref{prop_alternative_Jac} that we have
\begin{claim}
	The unit $1 \in \pi_0 \TMF^{U(1)} $ maps to $\Jac_{U(1)_k}[M, P, \psi] \in \pi_m \TJF_k$ by the composition \eqref{eq_geometric_Jac_2}. 
\end{claim}
\end{rem}

The topological elliptic genus $\Jac_{U(1)_k}$ has the following functoriality in increasing $k$. 

\begin{prop}
The following diagram commutes. 
    \begin{align}
        \xymatrix{
MTSU(k-1) \ar[d]^-{(SU(k-1) \hookrightarrow SU(k))_*}_-{\stab} \ar[rrr]^-{\Jac_{U(1)_{k-1}}} &&& \TJF_{k-1}  \ar[d]^-{\chi(V_{U(1)}) \cdot}\\
        MTSU(k)  \ar[rrr]^-{\Jac_{U(1)_k}}&&& \TJF_{k}
        }
    \end{align}

\end{prop}

\begin{proof}
    This is a special case of Proposition \ref{prop_Jac_different_nk}. 
\end{proof}

\subsection{The general construction}\label{subsec_construction}

The idea in the construction of the topological elliptic genus in the last subsection works quite generally. 
Here we explain the construction in the most general setting. 
Assume we are given a set of data as follows, symbolically denoted by $\mathcal{D}$. 
\begin{itemize}
    \item Fix compact Lie groups $G$ and $H$ contained in the subcategory $\frS \subset \mathrm{cptLie}$ in Fact \ref{fact_sigma} \colorbox{conjcolor}{(or simply $G, H \in \cptLie$, if we assume Conjecture \ref{conj_sigma}; see the last paragraph of Section \ref{subsec_stringstr})}, together with $\tau_G \in \RO(G)$ and $\tau_H \in \RO(H)$. 
\item Fix an integer $d$ and a group homomorphism $\phi \colon G \times H \to O(d)$. We denote the corresponding $d$-dimensional orthogonal representation by $V_\phi \in \Rep_{O}(G \times H)$. 
    \item We assume that $\dim \tau_H = 0$ and $d = \dim \tau_G$\footnote{This assumption is technical. In general, we can just add trivial representations to $\tau_G$ or $\tau_H$ to reduce to this case. }. 
    \item We fix a string structure $\mathfrak{s}$ on the virtual representation
    \begin{align}\label{eq_general_stringrep}
        \Theta_{\cD} := V_\phi - \res_G^{G \times H}(\tau_G)- \res_H^{G \times H}(\tau_H) \in \RO(G \times H),
    \end{align}
    i.e., we assume that the composition
    \begin{align}
        BG \times BH \xrightarrow{\Theta_\cD} BO \to P^4 BO
    \end{align}
    is nullhomotopic and $\mathfrak{s}$ is a choice of its nullhomotopy. 
\end{itemize}
By Fact \ref{fact_sigma} on the equivariant sigma orientation, the string structure $\mathfrak{s}$ induces an equivalence of $G \times H$-equivariant $\TMF$-module spectra, 
    \begin{align}\label{eq_general_assumption}
      \sigma(\Theta_\cD, \fraks) \colon \TMF[\Theta_\cD]  \simeq  \TMF
    \end{align}
    This induces the following equivalence in $\Mod_\TMF$ also denoted by the same symbol, 
    \begin{align}\label{eq_sigma_s}
        \sigma(\Theta_\cD, \fraks)\colon \TMF[V_\phi]^{G \times H} \simeq \TMF[\tau_G ]^G \otimes_\TMF \TMF[\tau_H]^H . 
    \end{align}

\begin{ex}\label{ex_JacU(1)}
    To recover the construction in the last subsection, we set $G:=U(1)$, $H:=SU(k)$ with $\tau_G := kV_{U(1)}$, $\tau_H :=\overline{V}_{SU(k)} =V_{SU(k)} - k\underline{\C} $ and $V_\phi:=V_{U(1)} \otimes_\C V_{SU(k)}$, and the equivalence \eqref{eq_general_assumption} is given by Proposition \ref{prop_key_JacU1}.  
\end{ex}

In this general setting, we construct a map of spectra
\begin{align}
    \Jac_{\mathcal{D}}\colon MT(H, \tau_H) \to \TMF[\tau_G]^G
\end{align}
as follows. 

\begin{defn}[{$\mathcal{F}_{\mathcal{D}}$}]\label{def_FD}
We define a morphism in $\Mod_\TMF$,
    \begin{align}\label{eq_F_general_coev}
        \mathcal{F}_{\mathcal{D}} \colon \TMF \to \TMF[\tau_H]^H \otimes_\TMF \TMF[\tau_G]^{G},
    \end{align}
    to be the following composition. 
 \begin{align}
         \mathcal{F}_{\mathcal{D}}  
   \colon \TMF
   \xrightarrow{\chi(V_\phi) \cdot} \TMF[V_\phi]^{G \times H} 
  \stackrel{\sigma(\Theta_\cD, \fraks)}{\simeq}\TMF[\tau_G ]^G \otimes_\TMF \TMF[\tau_H]^H . 
    \end{align}
    The last step uses \eqref{eq_sigma_s}. 
    Using the dualizability result \eqref{eq_twisted_dual}, the $\TMF$-linear dual to $\TMF[\tau_H]^H$ is canonically identified with $\TMF[-\tau_H-\Ad(H)]^H$. Thus the morphism \eqref{eq_F_U1_coev} is equivalently regarded as the following morphism,
    \begin{align}\label{eq_F'_general}
        \mathcal{F}'_{\mathcal{D}} \colon \TMF[-\tau_H-\Ad(H)]^H \to \TMF[\tau_G]^G
    \end{align}    
\end{defn}

\begin{rem}
    For some examples of $\mathcal{D}$ we give in Section \ref{sec_ex} (including Example \ref{ex_JacU(1)} above), we prove in Section \ref{sec_duality} that the morphism $\mathcal{F}'_{\mathcal{D}}$ provides a $\TMF$-module duality isomorphism
\begin{align}
D(\TMF[\tau_G]^G) \simeq \TMF[\tau_H]^H, 
\end{align}
which corresponds to the level-rank duality in physics. But the definition of topological elliptic genus below does NOT use the fact that it is an isomorphism, but only uses the morphism \eqref{eq_F'_general} (which is not in general an isomorphism). 
\end{rem}

\begin{defn}[{The topological elliptic genus $\Jac_{\mathcal{D}}$}]\label{def_general_Jac}
    In the above settings, we define $\Jac_{\mathcal{D}}$ to be the composition
    \begin{align}\label{eq_general_Jac}
\Jac_{\mathcal{D}} &\colon MT(H, \tau_H) = BH^{-\tau_H} \simeq (S^{-\tau_H})_{hH}\\
& \xrightarrow{u} \TMF[-\tau_H]_{hH}\\
&\xrightarrow[\eqref{eq_norm_equivariant}]{\Nm} \TMF[-\tau_H - \Ad(H)]^{H} \label{line1} \\
&\xrightarrow{\mathcal{F}'_{\mathcal{D}}} \TMF[\tau_G]^{G}, 
\end{align}
where $u \colon S \to \TMF$ is the unit map. 
\end{defn}

\begin{rem}[{Alternative definitions and geometric descriptions}]\label{rem_alternative_general}
	Recall that in the case of $\Jac_{U(1)_k}$ we have explained in Proposition \ref{prop_alternative_Jac} and Remark \ref{rem_geom_JacU(1)} that an alternative definition and the corresponding geometric description for $\Jac_{U(1)_k}$ are available. 
	In this general case here, we also have an analogous re-phrasing of the definition which only uses genuine $G$-equivariance and not using genuine $H$-equivariance nor the dualizability of equivariant $\TMF$. We also get the corresponding geometric description. 
	We leave the details to the reader. 
	
\end{rem}

\begin{rem}\label{rem_triv_Jac}
    Let us remark what happens if we take trivial choices of representations. We will see that the associated topological elliptic genus are something trivial. 
    Let $G$ and $H$ be compact Lie groups, $d=0$, $\tau_G=0$ and $\tau_H=0$. Then we have a trivial choice of the string orientation $\mathfrak{s}$ in \eqref{eq_general_stringrep}. Let us denote those data as $\mathcal{D}_\triv$. 
    
    Then, the map $\mathcal{F}_{\mathcal{D}_\triv}$ in Definition \ref{def_FD} factors as
    \begin{align}
        \mathcal{F}_{\mathcal{D}_\triv} \colon \TMF = \TMF \otimes_\TMF \TMF \xrightarrow{\res_e^{G} \otimes \res_e^H} \TMF^G \otimes_\TMF \TMF^H. 
    \end{align}
    So the map in \eqref{eq_F'_general} factors as
    \begin{align}
        \mathcal{F}'_{\mathcal{D}_\triv} \colon \TMF[-\Ad(H)]^H \xrightarrow{\tr_H^e} \TMF \xrightarrow{\res_e^G} \TMF^G.  
    \end{align}
    Now notice that the following diagram commutes, 
    \begin{align}
        \xymatrix{
        \Sigma^\infty BH_+ \ar[r]^-{u \otimes \id} \ar[d]^-{(H \to e)_*} & \TMF \otimes BH_+ \ar[r]^-{\Nm} & \TMF[-\Ad(H)]^H \ar[d]^-{\tr_H^e} \\
        S \ar[rr]^-{u} && \TMF, 
        }
    \end{align}
    Then the resulting topological elliptic genus in Definition \ref{def_general_Jac} just becomes the composition
    \begin{align}\label{eq_trivial_Jac}
        \Jac_{\mathcal{D}_\triv} \colon MT(H, 0) = \Sigma^\infty BH_+ \xrightarrow{(H \to e)_*} S \xrightarrow{u} \TMF \xrightarrow{\res_e^G} \TMF^G. 
    \end{align}
\end{rem}

In the next subsection, we will see further examples of this construction. 

\subsubsection{Functoriality}\label{subsubsec_functoriality}

Here we discuss an easy functoriality of the construction above. Suppose we have two sets of data $\mathcal{D} = (G, H, \tau_G, \tau_H, V_\phi, \mathfrak{s})$ and $\mathcal{D}' = (G', H', \tau_{G'}, \tau_{H'}, V_{\phi'}, \mathfrak{s}')$ as above. 
Assume that $d = \dim_\R V_\phi = \dim_\R V_{\phi'}$. 
We define a {\it morphism} 
\begin{align}
    \alpha \colon \mathcal{D} \to \mathcal{D}'
\end{align}
to consist of the following data: 
\begin{itemize}
    \item Group homomorphisms (note the directions!)
    \begin{align}
        \alpha_G &\colon G' \to G \\
        \alpha_H &\colon H \to H' .  
    \end{align}
    \item Equivalences of (virtual) representations, 
    \begin{align}\label{eq_functoriality_data}
        \alpha_{\tau_G} \colon \res_{\alpha_G}(\tau_G) &\simeq \tau_{G'} \mbox{ in }\RO(G'), \\
        \alpha_{\tau_H} \colon \res_{\alpha_H}(\tau_{H'}) &\simeq \tau_H  \mbox{ in }\RO(H), \\
        \alpha_{\phi} \colon \res_{\alpha_G \times \id_H}(V_\phi) &\simeq \res_{\id_{G'} \times \alpha_H}(V_\phi) \mbox{ in }\Rep_{O(d)}(G' \times H). 
    \end{align}
    \item An equivalence 
    \begin{align}
        \alpha_{\mathfrak{s}} \colon \res_{\alpha_G \times \id_H} (\mathfrak{s}) \simeq  \res_{\id_{G'} \times \alpha_H} (\mathfrak{s}')
    \end{align}
    of string structures on 
    \begin{align}
   \res_{\alpha_G \times \id_H} (\Theta_{\cD}) \simeq \res_{\id_{G'} \times \alpha_H}( \Theta_{\cD'}) \in \RO(G' \times H). 
    \end{align}
\end{itemize}

\begin{prop}\label{prop_functoriality}
    If we have a morphism $\alpha \colon \mathcal{D} \to \mathcal{D}'$ as above, the following statements hold. 
    \begin{enumerate}
        \item The maps $\cF_\cD$ and $\cF_\cD'$ are compatible in the sense that the following diagram commutes,
\begin{align}
    \xymatrix{
  \TMF \ar[rr]^-{\cF_\cD} \ar[d]^-{\cF_{\cD'}} && \TMF[\tau_H]^H \otimes_\TMF \TMF[\tau_G]^{G} \ar[d]^-{\id \otimes \res_{\alpha_G}} \\
    \TMF[\tau_{H'}]^{H'} \otimes_\TMF \TMF[\tau_{G'}]^{G'} \ar[rr]^-{\res_{\alpha_H} \otimes \id} & &
    \TMF[\tau_H]^H \otimes_\TMF \TMF[\tau_{G'}]^{G'}
    }
\end{align}
   \item The topological elliptic genera $\Jac_{\cD}$ and $\Jac_{\cD'}$ are compatible in the sense that the following diagram commute. 
    \begin{align}
    \xymatrix{
          MT(H, \tau_H) \ar[d]^-{\alpha_{\tau_H} \circ \alpha_H} \ar[r]^-{\Jac_{\mathcal{D}}} & \TMF[\tau_G]^G \ar[d]^-{\alpha_{\tau_G} \circ \res_{\alpha_G}}\\
         MT(H', \tau_{H'})\ar[r]^-{\Jac_{\mathcal{D'}}} & \TMF[\tau_{G'}]^{G'}. 
        }
    \end{align}

    \end{enumerate}
\end{prop}
\begin{proof}
(1) follows from the functoriality of the Euler classes and the isomorphism of string structures. 
(2) is a direct consequence of (1). 
\end{proof}

\section{Examples: The trio of $U$-$Sp$ \colorbox{conjcolor}{and $O$}-topological elliptic genera}\label{sec_ex}
In this section, we introduce a {\it trio} of examples---$(U, SU)$, $(Sp, Sp)$, \colorbox{conjcolor}{$(O, Spin)$}---where the general construction of Section \ref{subsec_construction} applies. Those classes come in families. 
\begin{defn}[{The topological elliptic genera $\Jac_{U(n)_k}$,  $\Jac_{Sp(n)_k}$ \colorbox{conjcolor}{and $\Jac_{O(n)_k}$}}]\label{def_Jac_trio}
We define the morphisms
\begin{align}
	& ~\Jac_{U(n)_k} \colon MT(SU(k), n\overline{V}_{SU(k)}) \to \TMF[kV_{U(n)}]^{U(n)}, \label{eq_def_JacU(n)}\\
	& \Jac_{Sp(n)_k} \colon MT(Sp(k), n\overline{V}_{Sp(k)} )\to \TMF[kV_{Sp(n)}]^{Sp(n)}, \label{eq_def_JacSp(n)} \\
    & \highlight{\Jac_{O(n)_k} \colon MT(Spin(k), n\overline{V}_{Spin(k)})\to \TMF[kV_{O(n)}]^{O(n)}  .} 
	\label{eq_def_JacO(n)}
\end{align}
for each $k, n \in \Z_{\ge 1}$, by applying the general construction to the following data. Here, for each group $K$ appearing below, the notation $V_{K} \in \RO(K)$ denotes the fundamental (a.k.a. defining, or vector) representation. 

\begin{itemize}
	\item For  \eqref{eq_def_JacU(n)}, the data $\cD =U(n)_k$ consists of
	\begin{align}
		G:=U(n), \  H:= SU(k), \ \tau_G := kV_{U(n)}, \ \tau_H := n\overline{V}_{SU(k)}, \ V_\phi := V_{U(n)} \otimes_\C V_{SU(k)} 
	\end{align} 
	so that $\Theta_{U(n)_k} =\overline{V}_{U(n)} \otimes_{\C} \overline{V}_{SU(k)}\in\RO(U(n) \times SU(k))$\footnote{We compute
	\begin{align}
		\Theta_{U(n)_k} &= V_\phi - \res_G^{G \times H}(\tau_G) - \res_H^{G \times H}(\tau_H) \\
		& =  V_{U(n)} \otimes_\C V_{SU(k)} - V_{U(n)}\otimes_\C k\underline{\C} - n\underline{\C} \otimes_{\C} (V_{SU(k)}-k\underline{\C})  \\
		& =\overline{V}_{U(n)} \otimes_{\C} \overline{V}_{SU(k)}.
	\end{align}
	Similar computations replacing $\C$ with $\R$ and $\bH$ produce the corresponding formulas for $\Theta_{O(n)_k}$ and $\Theta_{Sp(n)_k}$. \label{footnote_Theta}
	}, with its string structure obtained by Proposition \ref{prop_key_U(n)} below. 
	\item For \eqref{eq_def_JacSp(n)}, the data $\cD = Sp(n)_k$ consists of 
	\begin{align}\label{def_dataSp}
		G:=Sp(n), \ H:= Sp(k), \ \tau_G := kV_{Sp(n)}, \ \tau_H := n\overline{V}_{Sp(k)}, \ V_\phi := V_{Sp(n)} \otimes_\bH V^*_{Sp(k)}
	\end{align}
	Here $V_{Sp(k)}^*$ denotes the quarternionic dual representation so that $(g, h) \in Sp(n) \times_\bH Sp(k)$ acts on $V_\phi$ by $w \otimes_\bH v^* \mapsto g w \otimes v^* h^*$.\footnote{Note that $V_\phi$ is no longer a quaternionic representation, but just a real representation. This corresponds to the standard homomorphism $\phi \colon Sp(n) \times Sp(k) \to SO(4nk)$. } Since $V_{Sp(k)}^* \simeq V_{Sp(k)}$ in the {\it orthogonal} representation ring $\RO(Sp(k))$, the same computation as footnote \ref{footnote_Theta} is valid,
	so that $\Theta_{Sp(n)_k} =  \overline{V}_{Sp(n)}  \otimes_\bH \overline{V^*}_{Sp(k)}\in\RO(Sp(n) \times Sp(k))$, with its string structure obtained by Proposition \ref{prop_key_Sp(n)} below. 
	
\item \hspace{-6pt} \tikzmark{left} \hspace{-3pt} For \eqref{eq_def_JacO(n)}, the data $\cD =O(n)_k$ consists of
	\begin{align}
		G:=O(n), \ H:= Spin(k), \  \tau_G := kV_{O(n)}, \ \tau_H := n\overline{V}_{Spin(k)}, \ V_\phi := V_{O(n)} \otimes_\R V_{Spin(k)} 
	\end{align}
	so that $\Theta_{O(n)_k} =  \overline{V}_{O(n)}  \otimes_\R \overline{V}_{Spin(k)} \in \RO(O(n) \times Spin(k))$, 
	with its string structure obtained by Proposition \ref{prop_key_spin} below. \hfill \tikzmark{right}

\end{itemize}
\DrawBoxWide[thick, white, fill=blue, fill opacity=0.05]

\end{defn}

\begin{notation}
In the text, we generally refer to $\Jac_{U(n)_k}$, $\Jac_{Sp(n)_k}$ \colorbox{conjcolor}{and $\Jac_{O(n)_k}$} as the $U$-,$Sp$, \colorbox{conjcolor}{and $O$-}topological elliptic genera, respectively. When we want to specify $n$, we also use the term ``$U(n)$-topological elliptic genera'', and so on. 
\end{notation}

The particularly important case is $n=1$. We get $U(1)$, $Sp(1)$ \colorbox{conjcolor}{and $O(1)$}-topological elliptic genera from the familiar tangential bordism spectra, 
\begin{align}
	& \Jac_{U(1)_k} \colon MTSU(k) \to \TMF[kV_{U(1)}]^{U(1)}\simeq \TJF_k  \label{eq_def_JacU(1)}\\
&	\Jac_{Sp(1)_k} \colon MTSp(k)\to  \TMF[kV_{Sp(1)}]^{Sp(1)} := \TEJF_{2k} \label{eq_def_JacSp(1)} \\
&\highlight{	\Jac_{O(1)_k} \colon MTSpin(k)\to \TMF[kV_{O(1)}]^{O(1)} } , \label{eq_def_JacO(1)}
\end{align}

Also, it is important that we have obtained the {\it coevaluation maps}
	\begin{align}\label{eq_coev_trio}
		\mathcal{F}_{\cG(n)_k} \colon \TMF\to \TMF[kV_{\cG(n)}]^{\cG(n)} \otimes_\TMF \TMF[n\overline{V}_{\cH(k)}]^{\cH(k)}
	\end{align}
which have been used in Definition \ref{def_general_Jac} of the topological elliptic genera. This is the subject of Section \ref{sec_duality}: In the cases of $(\cG, \cH) = (U, SU) $ and $(Sp, Sp)$, we show that the above coevaluation map exhibits the duality between $\TMF[kV_{\cG(n)}]^{\cG(n)}$ and $\TMF[n\overline{V}_{\cH(k)}]^{\cH(k)}$ in $\Mod_\TMF$, which reflects the {\it level-rank duality} in physics.

The rest of this section is organized as follows. In Section \ref{subsec_stringstr}, we complete Definition \ref{def_Jac_trio} by showing the existence of a canonical choice of string structures on $\Theta_{\cG(n)_k}$ above. 
Then, in Section \ref{subsec_structure_trio} we explain the relations among $\Jac_{\cG(n)_k}$ for different $(\cG, \cH)$ and for different $(n, k)$, to illustrate that the trio of topological elliptic genera are organized in one coherent picture. 

\subsection{The string structures on $\Theta_{\cG(n)_k}$}\label{subsec_stringstr}

\begin{prop}\label{prop_key_U(n)}
	The virtual representation
	\begin{align}\label{eq_key_U(n)}
		\Theta_{U(n), SU(k)} = \overline{V}_{U(n)} \otimes_{\C} \overline{V}_{SU(k)}\in\RO(U(n) \times SU(k))
	\end{align}
	has a $BU\langle 6\rangle$-structure $\mathfrak{s}_{U, SU}$, and it is unique up to homotopy. This induces a string structure by \eqref{diag_Whitehead}. 
\end{prop}
\begin{proof}
	The proof is exactly parallel to that of Proposition~\ref{prop_key_JacU1} (Also see Remark~\ref{rem_ku_connective}).  
\end{proof}

\begin{prop}\label{prop_key_Sp(n)}
	The virtual representation 
	\begin{align}\label{eq_key_Sp}
		\Theta_{Sp(n), Sp(k)} = \overline{V}_{Sp(n)}  \otimes_\bH \overline{V^*}_{Sp(k)}\in\RO(Sp(n) \times Sp(k))
	\end{align}
	has a string structure $\mathfrak{s}_{Sp, Sp}$, and it is unique up to homotopy. 
\end{prop}

\begin{proof}
	Since $H^i(BSp(n) \times BSp(k); \Z) = 0$ for $i=1, 2$, we get a spin structure automatically. We have $H^4(BSp(n) \times BSp(k); \Z) \simeq \Z \oplus \Z$, so the string obstruction class $p_1/2$ for the representation in question is measured by $c_2$ after complexification. Now we have the following canonical identification for any pair of symplectic vector bundles $V$ and $W$ over a space $X$, 
	\begin{align}
		(V \otimes_{\bH} W) \otimes_\R \C \simeq V \otimes_{\C} W, 
	\end{align}
	where on the right hand side we used the underlying complex structures of $V$ and $W$. This means that
	\begin{align}
		c_2((V \otimes_{\bH} W) \otimes_\R \C) = c_2(V \otimes_{\C} W). 
	\end{align}
	Since forgetting symplectic structure to complex structure gives the map $Sp(n) \to SU(2n)$, we get $c_2( \overline{V}_{Sp(n)}  \otimes_\C \overline{V^*}_{Sp(k)}) = 0$ by Proposition \ref{prop_key_U(n)}. Thus we have a string structure as desired. 
	The uniqueness follows from $H^3(BSp(n) \times BSp(k); \Z) = 0$. 
\end{proof}

In order to state the proposition regarding the string orientation of $\Theta_{O(n)_k}$, we need a little preparation. 
Consider the following group homomorphisms, 
\begin{align}
	\alpha_G &\colon O(n) \hookrightarrow U(n), \\
	\beta_H &\colon SU\left(\lfloor k/2\rfloor \right) \hookrightarrow Spin(2\lfloor k/2\rfloor) \hookrightarrow Spin(k), 
\end{align}
where $\alpha_G$ is induced by $\R \hookrightarrow \C$, and $\beta_H$ is induced by forgetting the complex structure of $\C^{\lfloor k/2\rfloor}$ to regard it as the real vector space $\R^{2\lfloor k/2\rfloor}$, and the second arrow is nontrivial only for $k$ odd. 
Then we can easily verify that
\begin{lem}\label{lem_res_O_U}
	We have the following canonical isomorphism in $\RO\left(O(n) \times SU\left(\lfloor k/2\rfloor \right)\right)$,  
	\begin{align}\label{eq_res_O_U}
		\res_{\id \times \beta_H} \left( \overline{V}_{O(n)}  \otimes_\R \overline{V}_{Spin(k)} \right)
		\simeq \res_{\alpha_G \times \id}\left( \overline{V}_{U(n)}  \otimes_\C \overline{V}_{SU\left(\lfloor k/2\rfloor \right)} \right). 
	\end{align}
\end{lem}
The virtual representation appearing on the right hand side of \eqref{eq_res_O_U} is equipped with a string structure $\fraks_{U, SU}$ by Proposition \ref{prop_key_U(n)}. 
Now we can state the proposition for the string structure on $\Theta_{O(n), Spin(k)}$. 

\begin{prop}\label{prop_key_spin}
	The virtual representation
	\begin{align}
		\Theta_{O(n), Spin(k)} = \overline{V}_{O(n)}  \otimes_\R \overline{V}_{Spin(k)} \in \RO(O(n) \times Spin(k)) 
	\end{align}
	admits a string structure, and there is, up to homotopy, a unique choice $\fraks_{O, Spin}$ which admits the following equivalence of string structures when restricted to $O(n) \times Spin(k)$,
	\begin{align}
		\res_{\id \times \beta_H}(\fraks_{O, Spin}) \simeq\res_{\alpha_G \times \id}(\fraks_{U, SU}). 
	\end{align}
	Here we are using Lemma \ref{lem_res_O_U}, and the string structure $\fraks_{U, SU}$ on $\Theta_{U(n), SU(\lfloor k/2 \rfloor)}$ is the one in Proposition \ref{prop_key_U(n)}. 
\end{prop}
\begin{proof}
	The existence of string structures follows by checking the vanishing of $\frac{p_1}{2}$. 
	The second claim follows by the fact that the map
	\begin{align}
		BO(n) \times BSU(k') \xrightarrow{\id \times \beta_H} BO(n) \times BSpin(2k')
	\end{align}
	for any $k' \ge 1$ is $5$-connected, so that giving a string structure on $\Theta_{O(n), Spin(k)}$ is equivalent to giving a string structure on $ \res_{\id \times \beta_H}(\Theta_{O(n), Spin(k)})$. 
\end{proof}

\subsection{Structures of the trio}\label{subsec_structure_trio}

Now we explain the relations among the trio of topological elliptic genera we have constructed, unifying the above constructions into a coherent picture. There are the {\it external} structure relating different $(\cG, \cH)$, and the {\it internal} structure relating different $(n, k)$. 

\subsubsection{External structure: change of $(\cG, \cH)$}
Recall we have set up the notion of {\it morphisms} between the defining data of the general topological elliptic genera in Section \ref{subsubsec_functoriality}. 
We have a natural choice of morphisms
\begin{align}\label{eq_mor_trio}
&\alpha_{Sp}^U 	 \colon Sp(n)_k = (Sp(n), Sp(k), \cdots ) \to U(n)_{2k} = (U(n), SU(2k), \cdots) \\
& \highlight{	\alpha_{U}^O  \colon U(n)_{k} = (U(n), SU(k), \cdots) \to  O(n)_{2k}= (O(n), Spin(2k), \cdots)}
\end{align}
(where we abbreviated rest of the data by ``$\cdots$''), given by the group homomorphisms
\begin{align}
	 \alpha_G \colon U(n) \hookrightarrow Sp(n), \quad
	\alpha_H \colon Sp(k) \hookrightarrow SU(2k), \quad \mbox{ for }\alpha_{Sp}^U \label{eq_alphaSpU}\\
	\highlight{ \alpha_G \colon O(n) \hookrightarrow U(n), \quad 
	\alpha_H \colon SU(k) \hookrightarrow Spin(2k), \quad \mbox{ for }\alpha_{U}^O } \label{eq_alphaU0}
\end{align}
It is easy to complete the remaining ingredients listed in Section \ref{subsubsec_functoriality}, to get morphisms \eqref{eq_mor_trio}. \colorbox{conjcolor}{Note that the string structure in the data $(O, Spin)$ is chosen so that we get a morphism $\alpha_{U}^O$ above. }
From the above morphisms, we get the following maps, which we call the {\it external structure maps}, in the domains and codomains of the topological elliptic genera, 
\begin{align}
&MT(Sp(k), n\overline{V}_{Sp(k)}) \xrightarrow{(Sp(k) \hookrightarrow SU(2k))_*} MT(SU(2k), n\overline{V}_{SU(2k)}), \\
&	MT(SU(k), n\overline{V}_{U(k)}) \xrightarrow{(SU(k) \hookrightarrow Spin(2k))_*}MT(Spin(2k), n\overline{V}_{Spin(2k)}),
\end{align}

\begin{align}
&	\TMF[kV_{Sp(n)}]^{Sp(n)} \xrightarrow{\res^{U(n)}_{Sp(n)}} \TMF[2kV_{U(n)}]^{U(n)}, \\
&	\TMF[kV_{U(n)}]^{U(n)} \xrightarrow{\res^{O(n)}_{U(n)}} \TMF[2kV_{O(n)}]^{O(n)}
\end{align}

By Proposition \ref{prop_functoriality}, we see that our topological elliptic genera are compatible with the above structure maps, as follows. 

\begin{prop}[{Compatibility of $\Jac_{\cG(n)_k}$ for different $(\cG, \cH)$}]\label{prop_Jac_compatibility_external}
	The $U$, $Sp$ \colorbox{conjcolor}{and $O$}-topological elliptic genera are compatible in the sense that the following diagrams commute. 
		\begin{align}
		\xymatrix{
			MT(Sp(k), n\overline{V}_{Sp(k)} )\ar[d]^-{(Sp(k) \hookrightarrow SU(2k))_*} \ar[rr]^-{\Jac_{Sp(n)_k}} & &\TMF[kV_{Sp(n)}]^{Sp(n)} \ar[d]^-{\res^{U(n)}_{Sp(n)}}\\
			MT(SU(2k), n\overline{V}_{SU(2k)} )\ar[rr]^-{\Jac_{U(n)_{2k}}} & &\TMF[2kV_{U(n)}]^{U(n)}. 
		}
	\end{align}
	\begin{align}\label{diag_SU_O}
		\highlight{\begin{gathered}
		\xymatrix{
			MT(SU(k), n\overline{V}_{SU(k)}) \ar[d]^-{(SU(k) \hookrightarrow Spin(2k))_*} \ar[rr]^-{\Jac_{U(n)_k}} & &\TMF[kV_{U(n)}]^{U(n)} \ar[d]^-{\res^{O(n)}_{U(n)}}\\
			MT(Spin(2k), n\overline{V}_{Spin(2k)} )\ar[rr]^-{\Jac_{O(n)_k}} & &\TMF[2kV_{O(n)}]^{O(n)}. 
		}
		\end{gathered}}
	\end{align}

\end{prop}

\subsubsection{Internal structure: Change of $(n, k)$}\label{subsubsec_internal}
Now we introduce the {\it internal} structures in the trio, which relates different pairs of parameters $(n, k)$. In this case we fix $(\cG, \cH)$ to be any one of $(SU, U)$, $(Sp, Sp)$ \colorbox{conjcolor}{and $(Spin, O)$}. Set $N=2, 4, \colorbox{conjcolor}{1}$ in each case, respectively. 
In contrast to the previous structure maps, the internal structure maps relating different $(n, k)$ do NOT come from morphisms of the defining data in Section \ref{subsubsec_functoriality}. 
The internal structure maps here relates the parameters as shown in the following (non-commutative) diagram, 
\begin{align}
	&
	\xymatrix{
		& \ar@{~>}[d]^-{\res}& \ar@{~>}[d]^-{\res}& \ar@{~>}[d]^-{\res}& \\
\ar@{~>}[r]^-{\stab} & 	(k-1, n) \ar@{~>}[r]^-{\stab}_-{\substack{\phantom{X}\\\text{fib seq}}} \ar@{~>}[d]^-{\res}   & (k, n) \ar@{~>}[r]^-{\stab}_-{\substack{\phantom{X}\\\phantom{X}~~\text{fib seq}}} \ar@{~>}[d]^-{\res} & (k+1, n) \ar@{~>}[r]^-{\stab} \ar@{~>}[d]^-{\res} &\cdots \\
\ar@{~>}[r]^-{\stab} & 	(k-1, n-1)\ar@{~>}[r]^-{\stab}_-{\substack{\phantom{X}\\\text{fib seq}}}  \ar@{~>}[d]^-{\res} & (k, n-1) \ar@{~>}[r]^-{\stab} _-{\substack{\phantom{X}\\\phantom{X}~~\text{fib seq}}}  \ar@{~>}[d]^-{\res}& (k+1, n-1) \ar@{~>}[r]^-{\stab}  \ar@{~>}[d]^-{\res}& \cdots \\
&&&&
	}
	&
	\hspace{-4.74in}
	\raisebox{-0.61in}{
	\scalebox{1}[0.7]{
    \xymatrix{
		\phantom{(k-1,n)} \ar@/^27pt/@{~>}[rd] & \phantom{(k,n)} \ar@/^25pt/@{~>}[rd] & \phantom{(k+1,n)}
		\\
		\phantom{(k-1,n-1)} & \phantom{(k,n-1)} & \phantom{(k+1,n-1)} &
	}
	}
	}
	&
	\hspace{-4.11in}
	\raisebox{-1.21in}{
	\scalebox{1}[0.7]{
    \xymatrix{
		\phantom{(k-1,n-1)} \ar@/^27pt/@{~>}[rd] & \phantom{(k,n-1)} \ar@/^25pt/@{~>}[rd] & \phantom{(k+1,n-1)}
		\\
		\phantom{(k-1,n)} & \phantom{(k,n)} & \phantom{(k+1,n)} &
	}
	}
	}
\end{align}
and each $(k-1, n) \xrightarrow{\stab} (k, n) \xrightarrow{\res} (k, n-1)$ forms a fiber sequence of corresponding equivariant twisted $\TMF$ and of tangential Thom spectra, as we will see below. 

\begin{rem}\label{rem_internal_status}
	We do NOT use equivariant sigma orientation (Section \ref{subsec_stringstr}) for definition of the internal structures on equivariant $\TMF$ and the bordism spectra, so the contents from below until Remark \eqref{rem_Landweber_Novikov} does NOT rely on Fact \ref{fact_sigma} nor Conjecture \ref{conj_sigma}. So in particular we can apply Propositions \ref{prop_stab_res_TMF}, \ref{prop_stab_res_MT} and \ref{prop_geom_res} to $\cK = O, Spin$, WITHOUT assuming \colorbox{conjcolor}{Conjecture \ref{conj_sigma}}. 
\end{rem}

\paragraph{The internal structure in equivariant $\TMF$}

First, let us introduce the structure maps in the equivariant $\TMF$s appearing in the trio. Let $\cK$ be any one of $U, SU, Sp, O, Spin$, where we set $N = 2, 2, 4, 1, 1$, respectively. 
For each pair of integers $i \ge 1 $ and $j \in \Z$ (in the case of $\cK=SU, Spin$, we impose $i \ge 2$), consider the maps
\begin{align}
 \chi(V_{\cK(i)}) \cdot & \colon \TMF[(j-1)V_{\cK(i)}]^{\cK(i)} \to \TMF[jV_{\cK(i)}]^{\cK(i)}, \label{eq_TMF_stab}\\
 \res_{\cK(i)}^{\cK(i-1)} &\colon \TMF[jV_{\cK(i)}]^{\cK(i)} \to \TMF[jV_{\cK(i-1)} + Nj]^{\cK(i-1)} \label{eq_TMF_res}
\end{align}
which we call the {\it internal structure maps} in the trio of equivariant $\TMF$. We often call the maps \eqref{eq_TMF_stab} and \eqref{eq_TMF_res} {\it stabilization} and {\it restriction}, respectively. 

\begin{prop}[{The stabilization-restriction fiber sequence of equivariant $\TMF$ }]\label{prop_stab_res_TMF}\footnote{The authors thank Lennart Meier for noting this lemma. }
	Let $\cK$ be any one of $U, SU, Sp, O, Spin$\footnote{See Remark \ref{rem_internal_status}.}, where we set $N = 2, 2, 4, 1, 1$, respectively. Let $i \ge 1 $ (in the case $\cK=SU, Spin$ we impose $i \ge 2$) and $j \in \Z$. 
	The maps \eqref{eq_TMF_stab} and \eqref{eq_TMF_res} form a fiber sequence of $\TMF$-module spectra, 
	\begin{align}\label{seq_SU_induction}
		\TMF[(j-1)V_{\cK(i)}]^{\cK(i)} \xrightarrow[\stab]{ \chi(V_{\cK(i)}) \cdot }\TMF[jV_{\cK(i)}]^{\cK(i)} \xrightarrow[\res]{ \res_{\cK(i)}^{\cK(i-1)} }\TMF[jV_{\cK(i-1)} + Nj]^{\cK(i-1)} .
	\end{align}
\end{prop}

\begin{rem}\label{rem_internal_general}
	This proposition applies to any $\RO(\cK(i))$-graded spectrum, and not just $\TMF$. 
\end{rem}

\begin{proof}[Proof of Proposition \ref{prop_stab_res_TMF}]
	For each integer $i$ in that range,
	the homogeneous space $\cK(i)/ \cK(i-1)$ is identified, as a $\cK(i)$-space, with the unit sphere $S(V_{\cK(i)})$ of the fundamental representation. 
	Thus we have a cofiber sequence of pointed $\cK(i)$-spaces, 
	\begin{align}
		\cK(i)/\cK(i-1)_+ \to S^0 \xrightarrow{\chi(V_{\cK(i)})} S^{V_{\cK(i)}}. 
	\end{align}
	For any integer $j$, wedging with $S^{jV_{\cK(i)}}$ gives
	\begin{align}\label{eq_res_stab_proof}
		\Ind_{\cK(i-1)}^{\cK(i)}\left( S^{-jV_{\cK(i-1)} - Nj}\right) \simeq \cK(i)/\cK(i-1)_+ \otimes S^{-jV_{\cK(i)}} \to S^{-jV_{\cK(i)}} \xrightarrow{\chi(V_{\cK(i)}) \wedge \id} S^{(-j+1)V_{\cK(i)}}. 
	\end{align}
	Here, the first isomorphism used the following general fact: for any inclusion $H \subset G$ between compact Lie groups and any $G$-spectrum $X$, we have an isomorphism of $G$-spectra, 
	\begin{align}\label{eq_IndRes}
		\Ind_{H}^G \circ \Res_G^H (X) \simeq (G/H)_+ \otimes X. 
	\end{align}
	Applying $\lMap_{\cK(i)}(-, \TMF)^{\cK(i)}$ to this, we get the fiber sequence \eqref{seq_SU_induction}. 
\end{proof}

Here, let us make an interesting observation that the stabilization-restriction fiber sequence in Proposition \ref{prop_stab_res_TMF} is {\it self-dual} in the following sense: 
\begin{prop}[The self-duality of stabilization-restriction fiber sequences]\label{prop_selfdual_stabres}
	In the setting of Proposition \ref{prop_stab_res_TMF}, the following diagram commutes. 
	\begin{align}\label{diag_duality_stabres}
		\xymatrix{
				\TMF[jV_{\cK(i-1)} + Nj-1]^{\cK(i-1)} \ar[d] \ar[r]^-\simeq & D\left(\TMF[-jV_{\cK(i-1)} -Ni-1 - \Ad(\cK(i-1))]^{\cK(i)}\right) \ar[d]^-{D(\res)} \\
		\TMF[(j-1)V_{\cK(i)}]^{\cK(i)}\ar[d]^-{\stab}_-{\chi(V_{\cK(i)})\cdot } \ar[r]^-\simeq   &D\left(\TMF[-(j-1)V_{\cK(i)} - \Ad(\cK(i))]^{\cK(i)}\right) \ar[d]^-{D(\stab)}_-{D(\chi(V_{\cK(i)})\cdot)} \\
		\TMF[jV_{\cK(i)}]^{\cK(i)}  \ar[d]^-{\res} \ar[r]^-\simeq  &  D\left(\TMF[-jV_{\cK(i)} - \Ad(\cK(i))]^{\cK(i)}\right) \ar[d] \\
	\TMF[jV_{\cK(i-1)} + Nj]^{\cK(i-1)} \ar[r]^-\simeq  & D\left(\TMF[-jV_{\cK(i-1)} -Ni - \Ad(\cK(i-1))]^{\cK(i-1)}\right) 
		}
	\end{align}
Here both columns are fiber sequences of $\TMF$-modules in Proposition \ref{prop_stab_res_TMF}. $D$ denotes the dual in $\Mod_\TMF$, and we are using the dualizability result in \eqref{eq_twisted_dual}. 
In particular, the connecting map in the stabilization-restriction fiber sequence (the topleft vertical arrow in \eqref{diag_duality_stabres}) is identified with the dual to the restriction map, i.e., the transfer map 
\begin{align}\label{eq_tr_cofib}
	\tr_{\cK(i-1)}^{\cK(i)} \colon \TMF[jV_{\cK(i-1)} + Nj-1]^{\cK(i-1)} \to 	\TMF[(j-1)V_{\cK(i)}]^{\cK(i)}. 
\end{align}
\end{prop}

\begin{proof}
	Since we have identified the fiber of stabilization map as the restriction map in Proposition \ref{prop_stab_res_TMF}, it is enough that the middle square in \eqref{diag_duality_stabres} commutes. But this follows from the fact that the multiplication by an element in $\TMF[V_{\cK(i)}]^{\cK(i)}$ is a self-dual operation, since the coevaluation map of the duality data is provided by \eqref{eq_coev_duality_TMF}. 
\end{proof}

We get the diagram consisting of the structure maps, 
\begin{align}
	\scalebox{0.85}{
	\xymatrix@C=0.5em{
	\ar[r]^-{\stab} &	\TMF[(j-1)V_{\cK(i)}]^{\cK(i)}  \ar[r]^-{\stab} \ar[d]^-{\res}   & \TMF[jV_{\cK(i)}]^{\cK(i)}  \ar[r]^-{\stab} \ar[d]^-{\res} & \TMF[(j+1)V_{\cK(i)}]^{\cK(i)}  \ar[d]^-{\res} \ar[r]^-{\stab} &  \\
	&	\TMF[(j-1)(V_{\cK(i-1)} + N)]^{\cK(i-1)} & \TMF[j(V_{\cK(i-1)} + N)]^{\cK(i-1)}  & \TMF[(j+1)(V_{\cK(i-1)} + N)]^{\cK(i-1)} &
	}
	}
\end{align}
where each pair of consecutive horizontal and vertical arrows form a fiber sequence. 
Particularly important cases are the following.

\begin{ex}[$\TJF$]\label{ex_TJF_building}
Setting $\cK=U$ and $i = 1$ we get (here $\stab := \chi(V_{U(1)}) \cdot$)
	\begin{align}\label{eq_building_seq_TJF}
	\xymatrix{
		\TJF_{-1} \ar[r]^-{\stab} \ar@{=}[d]&	\TJF_0 \ar[r]^-{\stab} \ar[d]^-{\res_{U(1)}^e}& \TJF_1 \ar[r]^-{\stab}  \ar[d]^-{\res_{U(1)}^e}& \TJF_2 \ar[r]^-{\stab} \ar[d]^-{\res_{U(1)}^e}& \TJF_3 \ar[r]^-{\stab}  \ar[d]^-{\res_{U(1)}^e}& \cdots \\
		\TMF[1]&	\TMF &	\TMF[2]   & \TMF[4] & \TMF[6]   & \cdots
	}, 
\end{align}
where each pair of consecutive horizontal and virtical arrows form a fiber sequence
\begin{align}\label{seq_res_stab_TJF}
	\TJF_{k-1} \xrightarrow{\stab} \TJF_{k} \xrightarrow{\res_{U(1)}^e} \TMF[2k]. 
\end{align}
This fiber sequence is regarded as constructing $\TJF_{k}$ by attaching a single $2k$-dimensional $\TMF$-cell to $\TJF_{k-1}$. 
The sequence \eqref{eq_building_seq_TJF} is regarded as building $\TJF_{k}$ by starting from $\TJF_{1}\simeq \TMF$ (see Appendix Section \ref{subsec_cell_TJF_app}) and attaching even dimensional $\TMF$-cells one by one. 
We also employ the notation
\begin{align}\label{eq_TJFinfty}
	\TJF_\infty = \colim_k \left(\cdots \xrightarrow{\stab} \TJF_k \xrightarrow{\stab} \TJF_{k+1} \xrightarrow{\stab} \cdots \right) . 
\end{align}
For more on $\TJF$, see Appendix \ref{app_TJF}. 
\end{ex}

\begin{ex}[$\TEJF$]\label{ex_TEJF_building}
Similarly, in the case of $\cK=Sp$ and $i=1$, recalling our definition (Definition \ref{def_TEJF_app}) that $\TEJF_{2k} := \TMF[kV_{Sp(1)}]^{Sp(1)}$, we get (in this case, we set $\stab := \chi(V_{Sp(1)})$)
	\begin{align}\label{eq_building_seq_TEJF}
	\xymatrix{
		\TEJF_0 \ar[r]^-{\stab} \ar[d]^-{\res_{Sp(1)}^e}_-{\simeq}& \TEJF_2 \ar[r]^-{\stab}  \ar[d]^-{\res_{Sp(1)}^e}& \TEJF_4 \ar[r]^-{\stab} \ar[d]^-{\res_{Sp(1)}^e}& \TEJF_6 \ar[r]^-{\stab}  \ar[d]^-{\res_{Sp(1)}^e}& \cdots \\
		\TMF &	\TMF[4]   & \TMF[8] & \TMF[12]   & \cdots
	}, 
\end{align}
where each consecutive pair of horizontal and virtical arrows form a fiber sequence
\begin{align}\label{seq_res_stab_TEJF}
	\TEJF_{2k-2} \xrightarrow{\stab} \TEJF_{2k} \xrightarrow{\res_{Sp(1)}^e} \TMF[4k]. 
\end{align}
Here the equivalence $\TEJF_0 = \TMF^{Sp(1)} \simeq \TMF$ as indicated by the first vertical arrow in \eqref{eq_building_seq_TEJF} is the consequence of Fact \ref{factSU} below. 
This fiber sequence is regarded as constructing $\TEJF_{2k}$ by attaching a single $4k$-dimensional $\TMF$-cell to $\TEJF_{2k-2}$. 
The sequence \eqref{eq_building_seq_TEJF} is regarded as building $\TEJF_{2k}$ by starting from $\TEJF_{0} \simeq \TMF$ and attaching $4k$-dimensional $\TMF$-cells one by one. 
We study $\TEJF$ in more detail in Appendix \ref{app_TEJF}. We show, in Proposition \ref{prop_Sp(1)}, that we have (note that we are using $\HP^{k+1}$, NOT $\HP^{k+1}_+$)
\begin{align}
	\TEJF_{2k}\simeq \TMF \otimes \HP^{k+1}[-4], 
\end{align}
and the stabilization sequence \eqref{eq_building_seq_TEJF} is identified as the cell-attaching sequence of $\HP^k$. 
We also use the notation
\begin{align}\label{eq_TEJFinfty}
	\TEJF_\infty :=  \colim_k \left(\cdots \xrightarrow{\stab} \TEJF_{2k} \xrightarrow{\stab} \TEJF_{2k+2} \xrightarrow{\stab} \cdots \right) . 
\end{align}
For more on $\TEJF$, see Appendix \ref{app_TEJF}. 
\end{ex}

\paragraph{The internal structure in tangential Thom spectra}

Next, we introduce the internal structure maps in the tangential Thom spectra. We continue to set $\cK$ be any one of $U, SU, Sp, O, Spin$, where we set $N = 2, 2, 4, 1, 1$, respectively, and $i \in \Z_{\ge 1}$ ($i \ge 2 $ for $G=SU, Spin$), $j \in \Z$ as before. We consider
\begin{align}
\stab   &\colon	MT(\cK(i-1),j\overline{V}_{\cK(i-1)} )  \to MT(\cK(i),j\overline{V}_{\cK(i)} ), \label{eq_internal_dom_1}\\
\chi(V_{\cK(i)}) \cdot& \colon MT(\cK(i),j\overline{V}_{\cK(i)} ) \to MT(\cK(i),(j-1)\overline{V}_{\cK(i)} )[Ni]  \label{eq_internal_dom_2}
\end{align}
and call them the {\it internal structure maps} in the trio of tangential Thom spectra. 
Here, the map \eqref{eq_internal_dom_1} is the stabilization map \eqref{eq_stab_MTH_general} induced by the inclusion $\cK(i-1) \hookrightarrow \cK(i)$, and \eqref{eq_internal_dom_2} is the composition
\begin{align}\label{eq_def_res_MT}
\chi(V_{\cK(i)})\cdot&\colon MT(\cK(i),j\overline{V}_{\cK(i)} )
 \simeq (S^{-j\overline{V}_{\cK(i)}})_{h\cK(i)}  \\
&  \xrightarrow{\chi(V_{\cK(i)} ) \cdot}
(S^{-j\overline{V}_{\cK(i)}+V_{\cK(i)}})_{h\cK(i)} \simeq MT(\cK(i),(j-1)\overline{V}_{\cK(i)} )[Ni]  . \notag
\end{align}
By the analogy with the $\TMF$-case, we call \eqref{eq_internal_dom_2} as {\it restriction map} in the tangential Thom spectra in the trio. 
Geometric meaning of this map is explained after the next proposition.

Exactly similarly to Proposition \ref{prop_stab_res_TMF}, we get 
\begin{prop}[{The stabilization-restriction fiber sequence of tangential Thom spectra\footnote{See Remark \ref{rem_internal_status}.}}]\label{prop_stab_res_MT}
	In the setting above, the maps \eqref{eq_internal_dom_1} and \eqref{eq_internal_dom_2} form a fiber sequence
	\begin{align}
		MT(\cK(i-1),j\overline{V}_{\cK(i-1)} ) \xrightarrow{\stab} MT(\cK(i),j\overline{V}_{\cK(i)} ) \xrightarrow[\res]{\chi(V_{\cK(i)})} MT(\cK(i),(j-1)\overline{V}_{\cK(i)} )[Ni]  . 
	\end{align}
\end{prop}

\begin{proof}
The proof is exactly similar to that of Proposition \ref{prop_stab_res_TMF}. In this case, we apply $(-)_{h\cK(i)}$ to the sequence \eqref{eq_res_stab_proof} to get the result. 
\end{proof}

Thus we get the diagram consisting of the structure maps, 

\begin{align}\label{diag_building_MT}
	\scalebox{0.7}{
	\xymatrix@C=1em{
		\ar[r]^-{\stab} &MT(\cK(i-1),j\overline{V}_{\cK(i-1)} ) \ar[r]^-{\stab} \ar[d]^-{\res}_-{\chi(V_{\cK(i-1)})\cdot}   & MT(\cK(i),j\overline{V}_{\cK(i)} ) \ar[r]^-{\stab} \ar[d]^-{\res}_-{\chi(V_{\cK(i)})\cdot}   & MT(\cK(i+1),j\overline{V}_{\cK(i+1)} )\ar[d]^-{\res}_-{\chi(V_{\cK(i+1)})\cdot}   \ar[r]^-{\stab} &  \\
		&	MT(\cK(i-1),(j-1)\overline{V}_{\cK(i-1)} )[N(i-1)]  & MT(\cK(i),(j-1)\overline{V}_{\cK(i)} )[Ni]  & MT(\cK(i+1),(j-1)\overline{V}_{\cK(i+1)} )[N(i+1)] &
	}
}
\end{align}
where each pair of consecutive horizontal and vertical arrows form a fiber sequence. 
Now we explain the geometric meaning of those structure maps. By the Pontryagin-Thom isomorphism in Fact \ref{fact_PT}, applying $\pi_m$ to \eqref{eq_internal_dom_1} and \eqref{eq_internal_dom_2}, we get the maps in the tangential bordism groups, 
\begin{align}
\stab &\colon \Omega_m^{\left( \cK(i-1),j\overline{V}_{\cK(i-1)} \right) } \to \Omega_m^{\left( \cK(i),j\overline{V}_{\cK(i)} \right) } , \label{eq_geom_stab}\\
\res = \chi(V_{\cK(i)}) \cdot &\colon\Omega_m^{\left( \cK(i),j\overline{V}_{\cK(i)} \right) } \to \Omega_{m-Ni}^{\left( \cK(i),(j-1)\overline{V}_{\cK(i)} \right) } \label{eq_geom_res}
\end{align}
The geometric meaning of the stabilization map \eqref{eq_geom_stab} should be clear: a tangential $(\cK(i-1),j\overline{V}_{\cK(i-1)} )$-structure canonically induces a tangential $ (\cK(i),j\overline{V}_{\cK(i)} )$-structure by the inclusion $\cK(i-1) \hookrightarrow \cK(i)$. 
On the other hand, the restriction map \eqref{eq_geom_stab} is the interesting one, nicely explained as follows. 
By \eqref{eq_prelim_tangential}, an element of the tangential bordism group $ \Omega_m^{\left( \cK(i),j\overline{V}_{\cK(i)} \right) } $ is represented by a triple $(M, P, \psi)$, where $M$ is a closed $m$-dimensional manifold, $P$ is a principal $\cK(i)$-bundle and $\psi$ is an isomorphism of vector bundles over $M$, 
\begin{align}
\psi \colon TM \oplus \underline{\R}^L &\simeq (P \times_{\cK(i)} V_{\cK(i)}^{\oplus j}) \oplus \underline{\R}^{m+L-Nij} \\
&= V_P^{(1)} \oplus V_P^{(2)} \oplus V_P^{(3)} \oplus \cdots \oplus V_P^{(j)} \oplus \underline{\R}^{m+L-Nij}. 
\end{align}
with $L \ge 0$ a large enough integer; we also and denoted $V_P := P \times_{\cK(i)} V_{\cK(i)}$, and each $V_P^{(\bullet)}$ is a copy of $V_P$.  
Given such $(M, P, \psi)$, let us take a transverse section $s \in C^\infty(M; V_P^{(j)})$ of the $j$-th copy of $V_P$ in the splitting. Then, by the transversality, the zero locus $M' :=s^{-1}(0) \subset M$ is a smooth closed manifold of dimension $(m-Ni)$ with an isomorphism
\begin{align}
	\psi|_{TM'} \colon TM' \oplus \underline{\R}^L &\simeq V_P^{(1)} \oplus V_P^{(2)} \oplus V_P^{(3)} \oplus \cdots \oplus V_P^{(j-1)} \oplus \underline{\R}^{m+L-Nij}. 
\end{align}
This equips $M'$ with a tangential $\left( \cK(i),(j-1)\overline{V}_{\cK(i)} \right) $-structure.

\begin{prop}\footnote{See Remark \ref{rem_internal_status}.}\label{prop_geom_res}
	The map \eqref{eq_geom_res} is given by
	\begin{align}
 \chi(V_{\cK(i)}) \cdot \colon \left([M, P, \psi] \in  \Omega_m^{\left( \cK(i),j\overline{V}_{\cK(i)} \right) }\right) \mapsto \left([M',P_{M'}, \psi|_{TM'} ] \in \Omega_{m-Ni}^{\left( \cK(i),(j-1)\overline{V}_{\cK(i)} \right) } \right) , 
	\end{align}
	where the right hand side is the element just explained above. 
\end{prop}
\begin{proof}
	This is the direct consequence of applying the Pontryagin-Thom construction to the map \eqref{eq_def_res_MT}. 
\end{proof}

\begin{ex}[The case of $j=1$]\label{ex_res_MTH}
	The case of $j=1$ is of particular importance for us, especially in relation to Euler numbers and topological elliptic genera (Corollary \ref{cor_Jac_res_Euler} below).
	In this case, the restriction map in \eqref{eq_internal_dom_2} becomes
	\begin{align}\label{eq_res_MTG_j=1}
	 \chi(V_{\cK(i)}) \cdot \colon MT\cK(i) \to S[Ni],  
	\end{align}
	resulting in the map of bordism groups in Proposition \ref{prop_geom_res}
	\begin{align}\label{eq_generalized_Euler}
	\chi(V_{\cK(i)}) \cdot \colon \Omega_{m}^{\cK(i)} \to \Omega_{m-Ni}^{fr}, 
	\end{align}
	where $\Omega_*^\fr$ is the stably framed bordism group. In particular, if we set $m=Ni$, we have
	\begin{claim}\footnote{See Remark \ref{rem_internal_status}.}\label{claim_Euler}
		Let $M$ be a closed manifold with a \emph{strict} tangential $\cK(i)$-structure $\psi$ (Definition \ref{def_strict_str}---so that in particular $\dim_\R M = Ni$). 
		Then the map \eqref{eq_generalized_Euler} for $m=Ni$, 
		\begin{align}
		 \chi(V_{\cK(i)}) \cdot \colon \Omega_{Ni}^{\cK(i)} \to \Omega_{0}^{fr}=\Z, 
		\end{align}
		maps the element $[M, \psi]$ to its Euler number $\Eul(M) \in \Z$. 
	\end{claim}
	\begin{proof}
		This is a direct consequence of Proposition \ref{prop_geom_res}. The procedure in that proposition, applied to this case, produces the formula expressing the Euler number of $M$ in terms of vanishing points of generic vector fields. 
	\end{proof}
	
	\begin{rem}\label{rem_strict_Euler}
		The strictness assumption in Claim \ref{claim_Euler} is essential. Indeed, recall Example \ref{ex_spin_Sk}, where we intrduced two distinct tangential $Spin(k)$-structures on $S^k$: the one is the {\it stable} tangential $Spin(k)$-structure $\fraks_{\rm BB}^{Spin(k)}$ which is given by the blackboard framing, and the other is the {\it strict} tangential $Spin(k)$-structure $\fraks_{\rm str}^{Spin(k)}$. 
		
		Let $k$ be an even integer. We already know that $\Eul(S^k) = 2$. So Claim \ref{claim_Euler} applied here implies that
		\begin{align}
			\chi(V_{Spin(k)}) \cdot [S^k, \fraks_{\rm str}^{Spin(k)}] = 2 \in \Omega_0^\fr = \Z. 
		\end{align}
		On the other hand, since we already know that $[S^k, \fraks_{\rm BB}^{Spin(k)}] = 0 \in \Omega_{2k}^{Spin(k)}$, we have
		\begin{align}
				\chi(V_{Spin(k)}) \cdot [S^k, \fraks_{\rm str}^{Spin(k)}] = 0. 
		\end{align}
		This is not a contradiction, since we have $[S^k, \fraks_{\rm str}^{Spin(k)}] \neq [S^k, \fraks_{\rm BB}^{Spin(k)}] $ in $\Omega^{Spin(k)}_k$. 
		However, after stabilization those two tangential $Spin$-structures become bordant to each other. This example shows that our topological Elliptic genera are sensitive to {\it unstable} information. 
	\end{rem}
\end{ex}

\begin{rem}\label{rem_Landweber_Novikov}
	The restriction map \eqref{eq_geom_res} can be regarded as a variant of the {\it Landweber-Novikov operations} \cite{LandweberOperation}, \cite{NovikovCobordism}  on bordism homology theories. In general, for a multiplicative $B\cK \to BO$, given a map of the form  $\gamma \colon \Sigma^\infty B\cK_+ \to M\cK[d]$, by the universal Thom isomorphism for $\cK$-bundles, we can canonically associate an $M\cK$-module morphism $lf(\gamma) \colon M\cK \to M\cK[d]$, which is called the Landweber-Novikov operation associated to $\gamma$. 
	
	One concrete relation which we will use in our analysis of examples in Section~\ref{subsec_pi5MSp} is the following, concerning the case of $j=1$ explained in Example \ref{ex_res_MTH} above. Let $\cK$ be one of $U, SU, O, Spin, Sp$. 
	Denote by $\overline{e}_i \colon \Sigma^\infty B\cK_+ \to M\cK[Ni]$ the characteristic class which assigns
	\begin{align}
		\overline{e}_i (\xi) = e_i (-\xi) \in M\cK^{Ni}(X), 
	\end{align}
	for an $\cK$-vector bundle $\xi$ over $X$, where $\{e_l\}_{l=1}^\infty$ is the restriction of the standard generators $\{E_l\}_{l=1}^{\infty}$ of the $MU$, $MO$, $MSp$-characteristic classes in, e.g., \cite[(4.1)]{LandweberOperation}. 
	
	\begin{claim}\label{claim_lf_vs_euler}
		The following diagram commutes. 
		\begin{align}
			\xymatrix{
				MT\cK(i)  \ar[d]_-{\stab} \ar[rr]^-{\res:= \chi(V_{H(i)})\cdot}_-{\eqref{eq_res_MTG_j=1}} && S[Ni] \ar[d]^-{u} \\
				MT\cK \simeq M\cK \ar[rr]^-{lf({\overline{e}_i})}  && M\cK[Ni].
			}
		\end{align}
	\end{claim}
	The proof is straightforward by comparing the definitions of two horizontal arrows. 
\end{rem}

\paragraph{The compatibility of the topological elliptic genera with the internal structure maps}

Now we proceed to show that our topological elliptic genera are compatible with the internal structure maps introduced above. 

\begin{prop}[{Compatibility of the topological elliptic genera with the internal structure maps}]\label{prop_Jac_different_nk}
	Let $(\cG, \cH)$ be any one of $(U, SU), (Sp, Sp) \highlight{, (O, Spin)}$. 
	The following diagram commutes. 
	\begin{align}
		\xymatrix{
			MT(\cH(k-1),n\overline{V}_{\cH(k-1)} )  \ar[d]^-{\stab} \ar[rrr]^-{\Jac_{\cG(n)_{k-1}}} & &&\TMF[(k-1)V_{\cG(n)}]^{\cG(n)} \ar[d]^-{\chi(V_{\cG(n)}) \cdot}_-{\stab} \\
			MT(\cH(k),n\overline{V}_{\cH(k)} )   \ar[d]^-{\chi(V_{\cH(k)}) \cdot}_-{\res} \ar[rrr]^-{\Jac_{\cG(n)_k}} &&& \TMF[kV_{\cG(n)}]^{\cG(n)} \ar[d]^-{\res_{\cG(n)}^{\cG(n-1)}} \\
			MT(\cH(k),(n-1)\overline{V}_{\cH(k)} ) [Nk] \ar[rrr]^-{\Jac_{\cG(n-1)_k}} & &&\TMF[kV_{\cG(n-1)}+Nk]^{\cG(n-1)}
		}
	\end{align}
\end{prop}
The compatibility with the stabilization maps immediately implies, for example, that the $U(1)$-and $Sp(1)$-Jacobi orientations stabilize to give the maps (see \eqref{eq_TJFinfty} and \eqref{eq_TEJFinfty})
\begin{align}
	\Jac_{U(1)_\infty} \colon MTSU(\infty) \simeq  MSU \to \TJF_\infty, \label{eq_JacU(1)_infty}\\
	\Jac_{Sp(1)_\infty}\colon MTSp(\infty) \simeq  MSp \to \TEJF_\infty. \label{eq_JacSp(1)_infty}
\end{align}

Before proving Proposition \ref{prop_Jac_different_nk}, we deduce an important corollary of this proposition, which relates Euler numbers and topological elliptic genera. This is important in Section \ref{subsec_divisibility}, where we deduce interesting divisibility results of Euler numbers by way of our topological elliptic genera. 

\begin{cor}[{The restriction of $\Jac_{G(1)}$ is the Euler number}]\label{cor_Jac_res_Euler}
	Let $(\cG, \cH)$ be any one of $(U, SU)$, $(Sp, Sp)$ \colorbox{conjcolor}{, $(O, Spin)$}, and $k$ be a positive integer. The following diagram commutes. 
	\begin{align}
		\xymatrix{
			MT\cH(k)   \ar[d]^-{\chi(V_{\cH(k)}) \cdot}_-{\eqref{eq_res_MTG_j=1}} \ar[rrr]^-{\Jac_{G(1)_k}} &&& \TMF[kV_{G(1)}]^{G(1)} \ar[d]^-{\res_{G(1)}^{e}} \\
			S[Nk] \ar[rrr]^-{u} & &&\TMF[Nk]
		}
	\end{align}
	In particular, if $M$ is a closed manifold with a \emph{strict} tangential $\cH(k)$-structure $\psi$ (Definition \ref{def_strict_str}---so that in particular $\dim_\R M = Nk$), the composition
	\begin{align}\label{eq_Jac_res}
		 \Omega_{Nk}^{\cH(k)} \stackrel{\PT}{\simeq} \pi_{Nk}MT\cH(k) \xrightarrow{\Jac_{G(1)_k}}\TMF[kV_{G(1)}]^{G(1)}
		  \xrightarrow{\res_{G(1)}^e}  \pi_0 \TMF. 
	\end{align}
	sends the class $	[M, \psi] \in  \Omega_{Nk}^{\cH(k)}$ to the Euler number $\Eul(M) \in \Z = \pi_0S \xhookrightarrow{u} \pi_0\TMF$. 
\end{cor}

\begin{rem}\label{rem_strict_Jac}
	As in Remark \ref{rem_strict_Euler}, the strictness assumption in the second statement is essential. 
\end{rem}

\begin{proof}[Proof of Corollary \ref{cor_Jac_res_Euler} admitting Proposition \ref{prop_Jac_different_nk}]
	The first claim follows from the $n=1$ case of Proposition \ref{prop_Jac_different_nk}, by noting that $\Jac_{G(0)_k}$ is the unit map. 
	The second claim follows from Claim \ref{claim_lf_vs_euler}. 
\end{proof}

The rest of this subsection is devoted to proving Proposition \ref{prop_Jac_different_nk}. 
It is in fact an easy corollary of the following proposition, which we also use in Section \ref{sec_duality} on the level-rank duality. 
\begin{prop}\label{prop_F_different_nk}
	Let $(\cG, \cH)$ be one of $(U, SU)$, $ (Sp, Sp)$ \colorbox{conjcolor}{, $ (O, Spin)$}. Consider the following diagram in $\Mod_\TMF$. 
	\begin{align}\label{diag_different_nk}
		\xymatrix@C=5em{
			&\TMF \ar[ld]_-{\mathcal{F}_{\cG(n-1)_k}} \ar[d]^-{\mathcal{F}_{\cG(n)_k}} \ar[rd]^-{\mathcal{F}_{\cG(n)_{k-1}}} & \\
			\TMF[(n-1)\overline{V}_{\cH(k)}-Nk]^{\cH(k)} \ar[r]^-{\chi(V_{\cH(k)})} \ar@{}[d]|{\otimes}& \TMF[n\overline{V}_{\cH(k)}]^{\cH(k)}\ar[r]^-{\res_{\cH(k)}^{\cH(k-1)}} \ar@{}[d]|{\otimes}&
			\TMF[n\overline{V}_{\cH(k-1)}]^{\cH(k-1)}\ar@{}[d]|{\otimes} \\
			\TMF[kV_{\cG(n-1)}+Nk]^{\cG(n-1)} & 
			\TMF[kV_{\cG(n)}]^{\cG(n)} \ar[l]^-{\res_{\cG(n)}^{\cG(n-1)}}&
			\TMF[(k-1)V_{\cG(n)}]^{\cG(n)} \ar[l]^-{\chi(V_{\cG(n)})}
		}
	\end{align}
	Here $N=2,  4, \highlight{1}$ for $(\cG, \cH) = (U, SU), (Sp, Sp), \highlight{(O, Spin)}$, respectively. 
	Then, the left and the right halves of the diagram \eqref{diag_different_nk} are {\it compatible}, in the sense of Section \ref{subsec_notations} \eqref{prelim_compatible}. 
	Equivalently, the following diagram commutes. 
	\begin{align}\label{diag_F'_different_nk}
		\xymatrix{
			\TMF[-n\overline{V}_{\cH(k-1)} - \Ad(\cH(k-1))]^{\cH(k-1)} \ar[d]^-{\tr_{\cH(k-1)}^{\cH(k)}}_-{\eqref{eq_tr_cofib}} \ar[rr]^-{\mathcal{F}'_{\cG(n)_{k-1}}}&& \TMF[(k-1)V_{\cG(n)}]^{\cG(n)} \ar[d]^-{\chi(V_{\cG(n)}) \cdot}\\
			\TMF[-n\overline{V}_{\cH(k)} - \Ad(\cH(k))]^{\cH(k)} \ar[rr]^-{\mathcal{F}'_{\cG(n)_k}}\ar[d]^-{\chi(V_{\cH(k)}) \cdot}&& \TMF[kV_{\cG(n)}]^{\cG(n)} \ar[d]^-{\res_{\cG(n)}^{\cG(n-1)}} \\
			\TMF[-(n-1)\overline{V}_{\cH(k)} +Nk- \Ad(\cH(k))]^{\cH(k)} \ar[rr]^-{\mathcal{F}'_{\cG(n-1)_k}} && \TMF[kV_{\cG(n-1)}+Nk]^{\cG(n-1)}
		}
	\end{align}
\end{prop}

\begin{proof}[Proof of Proposition \ref{prop_F_different_nk}]
	We focus on proving the compatibility in the left half of diagram \ref{diag_different_nk}, since the other half is proven in the exact same way. 
	Recall that the right half of diagram \ref{diag_different_nk} is in more detail written as
	\begin{align}
		\xymatrix{
			\TMF \ar[d]_-{\chi(V_{\cH(k)} \otimes_{\mathbb{K}} V^{(*)}_{\cG(n)})}  \ar[rd]^-{\chi(V_{\cH(k-1)} \otimes_{\mathbb{K}} V^{(*)}_{\cG(n)})} 
			\\
			\TMF[V_{\cH(k)} \otimes_{\mathbb{K}} V^{(*)}_{\cG(n)}]^{\cH(k) \times \cG(n)} \ar[d]^-{\simeq}_-{\sigma(\Theta_{k,n}, \fraks)}&
			\TMF[V_{\cH(k-1)} \otimes_{\mathbb{K}} V^{(*)}_{\cG(n)}]^{\cH(k-1) \times \cG(n)} \ar[d]^-{\simeq}_-{\sigma(\Theta_{k-1,n}, \fraks)}
			\\
			\TMF[n\overline{V}_{\cH(k)}]^{\cH(k)}\ar[r]^-{\res_{\cH(k)}^{\cH(k-1)}} \ar@{}[d]|{\otimes}&
			\TMF[n\overline{V}_{\cH(k-1)}]^{\cH(k-1)}\ar@{}[d]|{\otimes} \\ 
			\TMF[kV_{\cG(n)}]^{\cG(n)} &
			\TMF[(k-1)V_{\cG(n)}]^{\cG(n)} \ar[l]^{\chi(V_{\cG(n)}) \cdot}
		}
	\end{align}
	Here we set $\mathbb{K} := \C, \bH, \highlight{\R}$ for $(\cG, \cH) = (U, SU), (Sp, Sp) \highlight{, (O, Spin)}$, respectively, and we need ``$*$'' in the second row only for the case $(\cG, \cH) =(Sp, Sp)$ (see \eqref{def_dataSp}). 
	The compatibility we need to prove is the commutativity of the square corresponding to \eqref{diag_compatible_square}. This case we need to compare two morphisms $\TMF \to \TMF[n\overline{V}_{\cH(k-1)}]^{\cH(k-1)} \otimes_\TMF \TMF[kV_{\cG(n)}]^{\cG(n)} $. 
	To simplify the notation, we denote $V_a := V^{(*)}_{\cG(a)}$ and $V'_{b}:= V_{\cH(b)}$, $\cG_a := \cG(a)$, $\cH_b := \cH(b)$ and $\res_k^{k-1} :=\res_{\cH(k)}^{\cH(k-1)}$ below (only in this proof). We consider the diagram
	\begin{align}\label{diag_big}
		\scalebox{0.85}{
		\xymatrix{
			&\TMF \ar[ld]_-{\chi(V'_{k} \otimes_\mathbb{K} V_{n})\cdot} \ar[rd]^-{ \chi(V'_{k-1} \otimes_\mathbb{K} V_{n})\cdot} & \\
			\TMF[V'_{k} \otimes_\mathbb{K} V_{n}]^{\cH_k \times \cG_n} \ar[d]^-{\simeq}_-{\sigma(\Theta_{k,n}, \fraks)} \ar[r]_-{\res_{k}^{k-1}}&
			\TMF[(V'_{k-1} \oplus \underline{\mathbb{K}}) \otimes_\mathbb{K} V_n]^{\cH_{k-1} \times \cG_n} \ar[d]^-{\simeq}_-{\sigma(\Theta_{k-1,n}, \fraks)}&
			\TMF[V'_{k-1} \otimes_\mathbb{K} V_{n}]^{\cH_{k-1} \times \cG_n} \ar[d]^-{\simeq}_-{\sigma(\Theta_{k-1,n}, \fraks)} \ar[l]^-{\cdot \chi(V_{n})}\\
			\TMF[n\overline{V'_k} \oplus kV_{n}]^{\cH_k \times \cG_n} \ar[r]_-{\res_k^{k-1}} &
			\TMF[n\overline{V'_{k-1}} \oplus kV_{n}]^{\cH_{k-1} \times \cG_n}
			&\TMF[n\overline{V'_{k-1}} \oplus (k-1)V_{n}]^{\cH_{k-1} \times \cG_n} \ar[l]^-{\cdot \chi(V_{n})}
		}
		}
	\end{align}
	By definition of all the morphisms in the diagram \eqref{diag_big}, it is easy to verify that the diagram commutes. Since what we need to prove is the equivalence of the outer compositions in the diagram \eqref{diag_big}, this completes the proof of Proposition \ref{prop_F_different_nk}. 
\end{proof}

\begin{proof}[Proof of Proposition \ref{prop_Jac_different_nk}]
	The claimed compatibility easily follows from the commutativity of \eqref{diag_F'_different_nk} and the definition of topological elliptic genera. 
\end{proof}

\subsection{The relation with Ando-French-Ganter \cite{Ando_2008}}\label{subsec_vs_JacOri}

In \cite{Ando_2008}, Ando, French and Ganter construct, given any ring spectrum $E$ with a ring homomorphism $s \colon MU\langle 2m + 2 \rangle \to E$ for a positive integer $m$, a morphism
\begin{align}
	\delta s \colon MU \langle 2m \rangle \to \Map(\CP^\infty_{-\infty}, E), 
\end{align}
where $\CP^\infty_{-\infty} := \lim_{k \to \infty} \CP^\infty_{-k}$ with
\begin{align}
	\CP^{\infty}_{-k} := (\CP^{\infty})^{-kV_{U(1)}}.
\end{align}
For their construction, we do NOT need any equivariant structure on $E$. 
Applied to the case of $E = \TMF$ with the sigma orientation $\sigma \colon MU\langle 6 \rangle \to \TMF$, we get 
\begin{align}\label{eq_delta_sigma}
	\sigma^\sharp  \colon MSU \to \Map(\CP^\infty_{-\infty} , \TMF).
\end{align}
It is the universal version\footnote{Precisely speaking, \cite{Ando_2008} specifically treats the case of elliptic spectra associated to an elliptic curve over $\mathrm{Spec}(R)$ with $R$ being an ordinary ring, but we can cirtainly apply their construction to the universal elliptic spectrum $\TMF$. } of what was called the {\it Jacobi orientation} of elliptic cohomology theories in \cite{Ando_2008}. 
In this subsection, we explain that our $U(1)$-topological elliptic genera can be regarded as a {\it genuine} and {\it unstable} version of \eqref{eq_delta_sigma} (Proposition \ref{prop_JacOri_vs_TEG}).

In general, for a genuine $G$-equivariant spectrum $E$, we have the commutative diagram in $\Spectra$, which is a version of {\it Tate square} \cite[Section~17.3.5.4]{hill2020equivariant},
\begin{align}\label{diag_equiv_htpy_U(1)}
	\xymatrix{
		E[-\Ad(G)]_{hG} \simeq \left( EG_+ \wedge E \right)^{G} \ar[rr]^-{\rm Nm} \ar@{=}[d]&& E^{G} \ar[d]^-{\zeta} \ar[rr] && \left(  \widetilde{EG}  \wedge E\right)^{G}\ar[d]^-{\rho}\\
	E[-\Ad(G)]_{hG} \simeq \left(  EG_+ \wedge F(EG_+, E) \right)^{G}\ar[rr]^-{\rm Nm}  && E^{hG} \simeq\left( F(EG_+, E)\right)^{G}  \ar[rr] && E^{tG}
	}
\end{align}
where the rows are fiber sequences (isotropy separation sequence and the norm cofiber sequence). 
Here $E^{G}$, $E^{hG}$ and $E^{t G}$ denote the genuine, homotopy and Tate fixed point spectra, respectively, and $\widetilde{EG}:= \cofib(EG_+ \to S^0)$. The middle and right vertical arrows are understood as the generalized Atiyah-Segal completion maps. 
In the case of $E = \TMF \in \Spectra^{U(1)}$, since $U(1)$ acts trivially on the underlying spectrum, we have
$\TMF_{hU(1)} \simeq \TMF \otimes \CP^\infty_+$, $\TMF^{hU(1)} \simeq \lMap(\CP_+^\infty, \TMF)$ and 
$\TMF^{tU(1)} \simeq \lMap(\CP^\infty_{-\infty}, \TMF)$. 
Moreover, the Norm map in the upper row is given by taking the colimit $k \to \infty$ (with respect to the stabilization sequence \eqref{eq_building_seq_TJF}) of the first arrow in the fiber sequence \eqref{eq_prf_cellTJF}, 
\begin{align}
		\TMF \otimes \CP^{k-1}_+[-1] \to \TMF^{U(1)} \to \TJF_k . 
\end{align} 
This means that we have
\begin{align}
	(\TMF \wedge \widetilde{EU(1)})^{U(1)} \simeq \TJF_\infty \stackrel{\eqref{eq_TJFinfty}}{:=}\colim_k \left(\cdots \xrightarrow{\stab} \TJF_k \xrightarrow{\stab} \TJF_{k+1} \xrightarrow{\stab} \cdots \right) , 
\end{align}
and the diagram \eqref{diag_equiv_htpy_U(1)} is identified as
\begin{align}\label{diag_geom_tate_TMF}
		\xymatrix{
		\TMF \otimes \CP^\infty_+[-1] \ar[rr]^-{\rm Nm} \ar@{=}[d]&& \TMF^{U(1)} \ar[d]^-{\zeta} \ar[rr] && \TJF_\infty \ar[d]^-{\rho}\\
	\TMF \otimes \CP^\infty_+[-1] \ar[rr]^-{\rm Nm}  && \lMap(\CP^\infty_+, \TMF)  \ar[rr] && \lMap(\CP^\infty_{-\infty}, \TMF)
	}
\end{align}

Now we can state the relation between our topological elliptic genera and Ando-French-Gepner's Jacobi orientation. Recall that our $U(1)$-topological elliptic genera stabilize to give \eqref{eq_JacU(1)_infty}
\begin{align}
\Jac_{U(1)_\infty} \colon MSU \to \TJF_\infty
\end{align}

\begin{prop}\label{prop_JacOri_vs_TEG}
	The Jacobi orientation $\delta \sigma$ in \eqref{eq_delta_sigma} factors as
	\begin{align}
		\sigma^\sharp = \rho \circ \Jac_{U(1)_\infty}  \colon MSU \xrightarrow[\eqref{eq_JacU(1)_infty}]{\Jac_{U(1)_\infty}} \TJF_\infty \xrightarrow[\eqref{diag_geom_tate_TMF}]{\rho} \lMap(\CP^\infty_{-\infty}, \TMF). 
	\end{align}
\end{prop}

\begin{proof}
	This directly follows from comparing our construction and that of \cite{Ando_2008}, especially Section 8 of their paper. 
\end{proof}

\section{The character formula}\label{sec_char_formula}

\noindent Note: The contents of Sections \ref{sec_char_formula}, \ref{sec_duality}, and \ref{sec_application} 
can be read independently of each other, and the reader may find it useful to skip to their section of interest.

In this section, we deduce the integration formula for the composition
\begin{align}
	e \circ \Jac_\cD \colon \pi_\bullet MT(H, \tau_H) \to \pi_\bullet \TMF[\tau_G]^G \xrightarrow{\eqref{eq_character}} \MF[\tau_G]^G|_{\deg  = \bullet}, 
\end{align}
producing $G$-equivariant integral Modular Forms (Definition \ref{def_JF_MFG}). 
We first derive the general formula in Proposition \ref{prop_formula_rational_general}, and specialize to the case of the $U$-and $Sp$-topological elliptic genera in Section \ref{sec_ex} to derive the concrete formula (Proposition \ref{prop_ch_formula_JacU(n)}, Proposition \ref{prop_formula_JacSp_rational}). In this section, we always assume that {\it $G$ and $H$ are connected and $\pi_1 G$ and $\pi_1 H$ are torsion-free}.
\begin{rem}\label{rem_variable}
	The formula we produce in this section is written in terms of functions in variables $z_i$ and $x_j$. In this subsection, we always use the convention that $z_i$ are associated to $G$ and $x_j$ are associated to $H$. They are defined after the choice of the maximal torus, and play the following a priori two distinct roles; 
	\begin{itemize}
		\item variables of equivariant Modular Forms as explained in Definition \ref{def_JF_MFG}, 
		\item generators of the ordinary cohomology ring of the maximal torus. 
	\end{itemize} 
	The two are canonically identified by the map \eqref{eq_CHD_C_U(1)}, but keeping track of the equivalence is essential in the following. 
\end{rem}

Recall that $\Jac_\cD$ is defined using the {\it coevaluation map} in Definition \ref{def_FD}, 
\begin{align}\label{eq_FD_charformula}
	   \mathcal{F}_{\mathcal{D}}  
	\colon \TMF
	\xrightarrow{\chi(V_\phi) \cdot} \TMF[V_\phi]^{G \times H} 
	\stackrel{\sigma(\Theta_\cD, \fraks)}{\simeq}\TMF[\tau_G ]^G \otimes_\TMF \TMF[\tau_H]^H . 
\end{align}
and used the canonical pairing between $\TMF[\tau_H]^{H}$ and $MT(H, \tau_H)$. 
At this point, it is convenient to convert \eqref{eq_FD_charformula} into equivariant Modular Forms. 
For simplicity, let us from now on assume at least one of $\pi_* \TMF_\Q[\tau_G]^G$ or $\pi_* \TMF_\Q[\tau_H]^H$ are flat over $\pi_* \TMF_\Q = \MF_\Q$\footnote{The formula in this section still works without this assumption. This assumption is satisfied for all the cases of the groups of our interest as we remarked in Remark \ref{rem_tensor_flat}.}
We have (here we denote $\MF_\Q[V]^G := \MF[V]^G \otimes \Q$ and so on)
\begin{align}
	\ch(\cF_\cD) \colon \MF \xrightarrow{\Phi_{V_{\phi}} \cdot }  \MF_\Q[V_\phi]^{G \times H} &=\MF_\Q[\res_G^{G \times H}\tau_G + \res_H^{G \times H} \tau_H]^{G \times H} \\
&= \MF_\Q[\tau_G ]^G \otimes_{\MF_\Q} \MF_\Q[\tau_H]^H . 
\end{align}
Here, we have used $\ch(\chi(V)) = \Phi_V$ in \eqref{eq_Phi_V=chiV} for the first arrow. For the second, recall that we have defined the ring of equivariant Modular Forms so that we have the {\it equality} between the source and target. The third equality uses 
The second arrow in \eqref{eq_FD_charformula} is converted into this equality because of Fact \ref{fact_sigma} 
\eqref{compatibility_AS_sigma}.

To get the integration formula in terms of characteristic polynomials, we need to translate the $\TMF$-valued characteristic classes to the rational ordinary cohomology. In our case, we denote the Chern-Dold character map for $\TMF$ by
\begin{align}
\CHD \colon \TMF \to H\MF_\Q, 
\end{align}
where $H\MF^\Q$ is the ordinary cohomology theory with coefficients in the $\Z$-graded abelian group $\MF_\Q $. 

Since we are assuming $H$ is connected, the element $\tau_H \in RO(H)$ is equipped with an ($SO$-)orientation $\mathfrak{o}$. 
Then we have the composition
\begin{align}\label{eq_CHD1}
\pi_*	\TMF[\tau_H]^H  \xrightarrow{\CHD} H^{-*}(BH^{-\tau_H}; \MF_\Q) \stackrel{\lambda(\tau_H, \frako)}{\simeq} H^{-*+\dim \tau_H}(BH; \MF_\Q), 
\end{align}
where we denote by $\lambda(\tau_H, \frako)$ the Thom isomorphism in the ordinary cohomology induced by the orientation $\frako$. 
Furthermore, if $H$ is connected, we choose a maximal torus $U(1)^r \simeq T \subset H$ with the Weyl group $W$ to identify
\begin{align}\label{eq_CHD2}
	H^{-*+\dim \tau_H}(BH; \MF_\Q)\simeq H^{-*+\dim \tau_H}(BU(1)^r; \MF_\Q)^W \simeq \left. \left( \MF_\Q[[x_1, x_2, \cdots, x_r]]\right)^W \right|_{\deg = *-\dim \tau_H}, 
\end{align}
where we have used the convention that $x \in H^2(BU(1); \Q)$ denotes the Thom class of the fundamental representation $V_{U(1)}$.
We denote the composition of \eqref{eq_CHD1} and \eqref{eq_CHD2} by
\begin{align}\label{eq_frakK}
\mathfrak{K}_H \colon	\pi_* \TMF[\tau_H]^H \to \left. \left( \MF_\Q[[x_1, x_2, \cdots, x_r]]\right) ^W \right|_{\deg = *-\dim \tau_H}
\end{align}
Since the second arrow in \eqref{eq_CHD1} canonically factors through $\MF[\tau_H]^H$, we can also define 
\begin{align}\label{eq_frK'}
	\mathfrak{K}'_H \colon  \left.\MF[\tau_H]^H \right|_{\deg = *} \to \left. \left( \MF_\Q[[x_1, x_2, \cdots, x_r]]\right) ^W \right|_{\deg = *-\dim \tau_H}
\end{align}
so that we have $\frK_H = \frK'_H \circ \ch$. 
Now we can state the general characteristic class formula. 
\begin{prop}[{Characteristic class formula for $\ch \circ \Jac_\cD$}]\label{prop_formula_rational_general}
	The characteristic polynomial associated to $\ch \circ \Jac_\cD \colon  \pi_\bullet MT(H, \tau_H) \to \MF[\tau_G]^G|_{\deg = \bullet}$ is given by 
	\begin{align}
	\left( \id_{\MF[\tau_G]^G} \otimes\mathfrak{K}'_H \right) 	\Phi_{V_\phi} \in \MF[\tau_G]^G \otimes \left( \Q[[x_1, \cdots, x_r]]\right)^W, 
	\end{align}
	where we are regarding $\Phi_{V_\phi} $ as an element in $\MF[\tau_G ]^G \otimes_\MF \MF[\tau_H]^H $. 
\end{prop}
\begin{proof}
	A priori, the characteristic class is obtained by the formula $$\left( \ch \otimes \frK_H \right) \circ \eqref{eq_CHD1}. $$
	Converting the equivariant Modular Forms from the beginning and using \eqref{eq_CHD2}, we get the result. 
\end{proof}

Let us work out how $\frK'_H$ looks like. First we work on the most fundamental case where $H=U(1)$ and $\tau_H = nV_{U(1)}$ for an integer $n$. The Thom isomorphism in the ordinary cohomology is identified as follows. 
\begin{align}\label{eq_x^n}
	\xymatrix{
H^*(BU(1)^{-nV_{U(1)}}; \Q) \ar[rr]^-{\lambda(nV_{U(1)}, \frako)} \ar@{=}[d] && H^*(BU; \Q)\ar@{=}[d] \\
x^{-n}\Q[[x]] \ar[rr]^-{x^n \cdot} && \Q[[x]]
}
\end{align}
The Chern-Dold character map is simply taking the formal expansion at the origin of the elliptic curves, 
\begin{align}\label{eq_CHD_C_U(1)}
\CHD_\C \colon \TMF[nV_{U(1)}]^{U(1)} \xrightarrow{} \Gamma(\cE_\C; \mathcal{O}(ne))
\xhookrightarrow{\ev_{\widehat{\cE}_\C}} x^{-n}\MF^{\C}[[x]] = H^*(BU(1)^{-nV_{U(1)}}; \MF_\C). 
\end{align}
where $x$ is regarded as the coordinate of the elliptic curve (which we had been denoted as $z$, but here we intentionally use a different letter). 
On the other hand, if we were to factor through Jacobi Forms, we should include the multiplication by the Theta function
\begin{align}\label{eq_CHDU1}
\CHD_\C \colon	\TMF[nV_{U(1)}]^{U(1)} \xrightarrow{e_\JF} \JF_n  \xrightarrow{a(x, \tau)^{-n}} \Gamma(\cE_\C; \mathcal{O}(ne))
	\xrightarrow{\ev_{\widehat{\cE}_\C}} x^{-n}\MF_{\C}[[x]] , 
\end{align}
by Lemma \ref{lem_AvsL}. Thus, composing the latter two arrows of \eqref{eq_CHDU1} and \eqref{eq_x^n}, the map $\frK'_{U(1)}$ \eqref{eq_frK'}  in this case becomes the composition
\begin{align}
	\frK'_{U(1)} \colon \MF[nV_{U(1)}]^{U(1)} = \JF_n \xrightarrow{\cdot \left( \frac{x}{a(x, \tau)} \right)^n } \MF_\Q[[x]]. 
\end{align}

In the case $H=SU(k)$, we follow the conventional approach that, rather than using the maximal torus of $SU(k)$, first regard $SU(k) \subset U(k)$ and use the maximal torus $U(1)^{k} \hookrightarrow U(k)$ to identify
\begin{align}\label{eq_H*BSU}
	H^*(BSU(k); \Q) =\left(  \Q[[x_1, x_2, \cdots, x_k]]/(x_1 + \cdots + x_k)\right) ^{\Sigma_k}. 
\end{align}
Then the map $\frK'_{SU(k)}$ for $\tau_H = nV_{SU(k)}$ becomes (see Example \ref{ex_SU(n)_JF})
\begin{align}
&	\frK'_{SU(k)} =\cdot \prod_{1 \le j \le k}  \left( \frac{x_j}{a(x_j, \tau)}\right) ^{k} \\
	& \quad \colon \MF[nV_{SU(k)}]^{SU(k)} \xrightarrow{\otimes \Q}\left( \frac{\bigotimes^{\MF_\Q}_{1 \le j \le k} \JF_n \otimes \Q}{(x_1+x_2+\cdots +x_k)} \right)^{\Sigma_k} 
	\to	\left(\frac{ \MF_\Q[[x_1, \cdots, x_k]]}{(x_1 + x_2 + \cdots +x_k)}\right)^{\Sigma_k}
\end{align}
where the tensor product is formed over the graded ring $\MF$.

\begin{prop}[{The formula for the characteristic polynomial of $\ch \circ \Jac_{SU(k)_n}$}]\label{prop_ch_formula_JacU(n)}
The characteristic polynomial $$\mathcal{K}_{U(n)_k} \in H^*(BSU(k); \Q) \otimes \MF[kV_{U(n)}]^{U(n)}$$ of the composition $\ch \circ \Jac_{U(n)_k} \colon \pi_\bullet MT(SU(k), n\overline{V}_{SU(k)})\to \MF[kV_{U(n)}]^{U(n)}|_{\deg = \bullet}$ is given by the formula
\begin{align}
\mathcal{K}_{U(n)_k} \left( \{x_j\}_{1 \le j \le k}\right) & := \prod_{1 \le i \le n, \ 1 \le j \le k}\frac{x_j\theta(z_i+x_j, q)}{\theta(x_j, q)} \\
&= \left( \prod_{i} e^{z_i/2}\right)^k \cdot
\left( \prod_{j} \frac{x_j }{1-e^{-x_j}}\right)^n 
\prod_{\substack{m \ge 1 ,\ i, j}} \frac{(1-q^{m}e^{z_i+x_j})(1-q^{m-1}e^{-z_i-x_j})}{(1-q^me^{x_j})(1-q^me^{-x_j})}. 
\end{align}
Here, $\{z_i\}_{i}$ are the variables of $U(n)$-equivariant Modular Forms given by the canonical choice of the maximal torus $U(1)^n \hookrightarrow U(n)$, and $\{x_j\}_{j}$ are the variables of $H^*(BSU(k); \Q)$ in \eqref{eq_H*BSU}. 
\end{prop}

\begin{proof}
We have
\begin{align}
	\Phi_{V_{U(n) } \otimes V_{SU(k)}} = \prod_{i,j} a(z_i + x_j) \in \MF[V_{U(n) } \otimes V_{SU(k)}]^{U(n) \times SU(k)}
\end{align}
where we recall that $a(z, \tau) = \Phi_{V_{U(1)}}(z, \tau) = \theta(z, \tau)/\eta(\tau)^3$ \eqref{eq_notation_a}. 
The formula in Proposition \ref{prop_formula_rational_general} becomes
\begin{align}
	\prod_{j}  \left( \frac{x_j}{a(x_j, \tau)}\right) ^{k}   \cdot \prod_{i,j} a(z_i + x_j, \tau) = 
	\prod_{1 \le i \le n, \ 1 \le j \le k}\frac{x_j\theta(z_i+x_j, q)}{\theta(x_j, q)}. 
\end{align}
\end{proof}

By Proposition \ref{prop_ch_formula_JacU(n)} we get the following integration formula for the character of the $U$-topological elliptic genera, 
\begin{cor}[{The integration formula for $\ch \circ \Jac_{U(n)_k}$}]\label{cor_Jac_integration_formula_SU_U}
	For $[M, \psi] \in \Omega_{m}^{(BSU(k), nV_{SU(k)})}$ we have
	\begin{align}\label{eq_integration_formula_Jac_SU_U}
		 \ch \circ \Jac_{U(n)_k}  [M, \psi] 
		=  \left(\prod_{1 \le i \le n}y_i^{\frac12}\right)^k \cdot \int_M \mathrm{Todd}(TM)^n \wedge \Ch\left(\otimes_{1 \le i \le n}\mathbb{TM}_{q, y_i}\right) ,
	\end{align}
	where we used the variable $y_i = e^{2\pi \sqrt{-1} z_i}$, and set (in the formula below all the tensor/exterior products are over $\C$, )
	\begin{align}
		\mathbb{TM}_{q, y} := \bigotimes_{m \ge 0} \wedge_{-q^{m}y^{-1}}T^*M \otimes \bigotimes_{m \ge 1} \wedge_{-q^{m}y} TM \otimes  \bigotimes_{m \ge 1} \mathrm{Sym}_{q^{m}}T^*M \otimes \bigotimes_{m \ge 1} \mathrm{Sym}_{q^{m}} TM . 
	\end{align}
\end{cor}

When $n=1$, this specializes to the formula \eqref{eq_intro_Jac_clas} for the classical elliptic genera $\Jac_{\clas}$ of tangential $SU(k)$-manifolds. 

\begin{rem}[Comparison with other literatures]\label{ref_comparizon_ell}
	In some literatures including \cite{Ando_2008} and \cite{totaro2000}, the elliptic genus for a tangential $SU(k)$-manifold $M$ is defined to be 
	\begin{align}
		a^{-k} \cdot \Jac_{\clas}(M) \in \Gamma(\cE_\C; \cO(ke) \otimes \omega^\bullet). 
	\end{align}
	This is because they define elliptic genera as a map from the stable $SU$-bordism group. See Section \ref{subsec_vs_JacOri}. 
\end{rem}

The character formula for the $Sp$-topological elliptic genus directly folllows from the above result on the $U$-topological elliptic genus. This is because the map of equivariant Modular Forms,
\begin{align}\label{eq_MFSP_MFU}
{\res_{Sp(n)}^{U(n)}} \colon	\MF[kV_{Sp(n)}]^{Sp(n) } \to \MF[2kV_{U(n)}]^{U(n)}
\end{align}
is an injection, as we have seen in Example \ref{ex_Sp(n)}: the $Sp(n)$-equivariant Modular Forms are the $U(n)$-equivariant Modular Forms even in all the variables $z_i$.
By the above injectivity and the functoriality of the topological elliptic genera with respect to the external structure map, as in Proposition \ref{prop_Jac_compatibility_external}, we see that the composition
\begin{align}\label{eq_Sp_Jac_ch}
	\ch \circ \Jac_{Sp(n)_k} \colon MT(Sp(k), n\overline{V}_{Sp(k)}) \to \MF[kV_{Sp(n)}]^{Sp(n)}
\end{align}
is simply given by retaining the $SU(2k)$-structure underlying the $Sp(k)$-structure and applying the formula we have obtained for the $U$-topological elliptic genus. 
Thus we get the following. 

\begin{prop}[The formulas for $\ch \circ \Jac_{Sp(n)_k}$]\label{prop_formula_JacSp_rational}
	The restriction along $Sp(k) \hookrightarrow SU(2k)$ of the element $\mathcal{K}_{U(2n)_k}$ in Proposition \ref{prop_ch_formula_JacU(n)} is contained in $H^*(BSp(k); \Q) \otimes \MF[kV_{Sp(n)}]$, which gives the characteristic polynomial of the composition \eqref{eq_Sp_Jac_ch}, 
	\begin{align}
		\mathcal{K}_{Sp(n)_k} = \res_{SU(2k)}^{Sp(k)} \mathcal{K}_{U(n)_k} \in H^*(BSp(k); \Q) \otimes \MF[kV_{Sp(n)}]^{Sp(n)} \subset H^*(BSU(2k); \Q) \otimes \MF[2kV_{U(n)}]^{U(n)}
	\end{align}
	Here, we have used the injectivity of \eqref{eq_MFSP_MFU}. 
	In particular, for $[M, \psi] \in \Omega_m^{(BSp(k), nV_{Sp(k)})}$, the integration formula for $\ch \circ \Jac_{Sp(n)_k}[M, \psi]$ is simply given by retaining the $SU(2k)$-structure underlying the $Sp(k)$-structure and applying the formula \eqref{eq_integration_formula_Jac_SU_U}. 
\end{prop}

\section{Level-rank duality isomorphisms in $\TMF$}\label{sec_duality}
As explained in Section \ref{subsec_construction}, in the general settings there, we get a composition of $\TMF$-module morphisms
\begin{align}
    \mathcal{F}_{\mathcal{D}} \colon \TMF \xrightarrow{\chi(V_\phi) } \TMF[ V_\phi]^{G \times H} \stackrel{\sigma(\Theta_\cD, \fraks) }{\simeq}\TMF[\tau_G]^G \otimes_{\TMF} \TMF[\tau_H]^H. 
\end{align}
In the setting of the trio we presented in Section \ref{sec_ex}, we expect the above map to be related to the {\it level-rank duality} in physics. While initially discovered in the context of affine Lie algebras and conformal field theory \cite{frenkel2006representations,nakanishi1992level}, the level-rank duality can be formulated in the closely related frameworks of Chern-Simons theories \cite{naculich1990group,mlawer1991group,naculich2007level,Hsin:2016blu} and tensor categories \cite{ostrik2014,ostrik2020symplectic}.
In this section, we verify mathematically that, indeed among our trio, in the cases of $(\cG, \cH) = (U, SU)$ and $(Sp, Sp)$, the map $\cF_\cD$ exhibits the duality in $\TMF$-module spectra:\footnote{{The authors acknowledge Du Pei and Lennart Meier for providing the idea of the contents in this section.}}
\begin{align}
	\TMF[kV_{Sp(n)}]^{Sp(n)} &\stackrel{\mbox{\tiny dual} }{\longleftrightarrow}	\TMF[n\overline{V}_{Sp(k)}]^{Sp(k)} \\
		\TMF[kV_{U(n)}]^{U(n)} &\stackrel{\mbox{\tiny dual  }}{\longleftrightarrow}	\TMF[n\overline{V}_{SU(k)}]^{SU(k)}  \quad \mbox{in }\Mod_\TMF. 
\end{align}

\begin{rem}\label{rem_levelrank}
	In this article, we do not go further into the level-rank duality itself, especially with physical explorations, though the authors certainly encourage explorations in this direction. Nevertheless, we include the relevant mathematical proofs here, since these results show that our generalized topological elliptic genera are highly nontrivial. 
\end{rem}

We heavily use the following fact, which will appear in an upcoming paper by Gepner-Meier \cite{GepnerMeierNEW}:
\begin{fact}[{Gepner-Meier, \cite{GepnerMeierNEW}}]\label{factSU}~
	\begin{enumerate}
		\item For any positive integer $k$, the restriction map provides an isomorphism, 
		\begin{align}
			\res_{SU(k)}^e \colon \TMF^{SU(k)} \simeq \TMF \label{eq_fact_SU} \\
		 \res_{Sp(k)}^e \colon \TMF^{Sp(k)} \simeq \TMF \label{eq_fact_Sp}
		\end{align}
		\item For any positive integer $k$, the restriction map along $\det \colon U(k) \to U(1)$ provides an isomorphism, 
		\begin{align}
			\TMF^{U(k)} \stackrel{\res_{\det}}{\simeq} \TMF^{U(1)} \stackrel{(\res_{U(1)}^e, \tr_{U(1)}^e)}{\simeq} \TMF \oplus \TMF[1]. 
		\end{align}
	\end{enumerate}
\end{fact}

The rest of this section is organized as follows. In Section \ref{subsec_duality_Sp} and \ref{subsec_duality_SU_U}, we show the duality statement for $(Sp, Sp)$ and $(SU, U)$, respectively. Section \ref{subsec_duality_lem} is devoted to the proof of a general lemma on the duality in symmetric monoidal categories, which we use in the proofs of main theorems in the earlier subsections. 

\subsection{The level-rank duality between $Sp(n)_k$ and $Sp(k)_n$}\label{subsec_duality_Sp}

First, we analyze the case of $\cD=Sp(n)_k$, where the argument is simpler than the case of $\cD=U(n)_k$. 
We show the following. 

\begin{thm}[{The level-rank duality between $Sp(n)_k$ and $Sp(k)_n$}]\label{thm_duality_Sp}
	Let $n$ and $k$ be nonnegative integers. The coevaluation map in Definition \ref{def_FD} applied to the data $\cD = Sp(n)_k$ in Definition \ref{def_Jac_trio}, 
	\begin{align}
		\mathcal{F}_{Sp(n)_k} \colon \TMF &\xrightarrow{\chi(V_{Sp(n)} \otimes_\bH V^*_{Sp(k)})} \TMF[V_{Sp(n)} \otimes_\bH V^*_{Sp(k)}]^{Sp(n) \times Sp(k)} \\ &\stackrel{\sigma(\Theta_{Sp(n)_k}, \fraks)}{\simeq} \TMF[kV_{Sp(n)}]^{Sp(n)} \otimes_\TMF \TMF[n\overline{V}_{Sp(k)}]^{Sp(k)},  
	\end{align}
exhibits
$\TMF[kV_{Sp(n)}]^{Sp(n)}$ and $\TMF[n\overline{V}_{Sp(k)}]^{Sp(k)}$ as duals to each other in $\Mod_\TMF$,
\end{thm}

\begin{proof}
	We prove the theorem by induction. First, as the base step, we check that the statement holds when either one of $n$ or $k$ is $0$; but this is simply implied by Fact \ref{factSU} (1), \eqref{eq_fact_Sp}. 
	
	Now, recall the following diagram for $n \ge 1$ and $k \ge 1$ in \eqref{diag_different_nk} specialized to our case. 
	\begin{align}\label{diag_proofduality_Sp}
		\xymatrix{
			&\TMF \ar[ld]_-{\mathcal{F}_{Sp(n-1)_k}} \ar[d]^-{\mathcal{F}_{Sp(n)_k}} \ar[rd]^-{\mathcal{F}_{Sp(n)_{k-1}}} & \\
			\TMF[(n-1)\overline{V}_{Sp(k)}-4k]^{Sp(k)} \ar[r]_-{\chi(V_{Sp(k)}) \cdot} \ar@{}[d]|{\otimes}& \TMF[n\overline{V}_{Sp(k)}]^{Sp(k)}\ar[r]_-{\res} \ar@{}[d]|{\otimes}&
			\TMF[n\overline{V}_{Sp(k-1)}]^{Sp(k-1)}\ar@{}[d]|{\otimes} \\
			\TMF[kV_{Sp(n-1)}+4k]^{Sp(n-1)} & 
			\TMF[kV_{Sp(n)}]^{Sp(n)} \ar[l]^-{\res}&
			\TMF[(k-1)V_{Sp(n)}]^{Sp(n)} \ar[l]^-{\chi(V_{Sp(n)})\cdot}
		}
	\end{align}
	
	 Here the second and third rows are the stabilization-restriction fiber sequences in Proposition \ref{prop_stab_res_TMF}. 
	By Proposition \ref{prop_F_different_nk}, both the left and the right half of the diagram \eqref{diag_proofduality_SU_U} is {\it compatible}, in the sense of Section \ref{subsec_notations} \eqref{prelim_compatible}. 
	Using this result and a general Lemma \ref{lem_duality} below, we prove the statement of Theorem by induction on $n$, and within that, we induct on $k$. The base case $n=0$ has already been checked above. 
	
	Now, suppose that we have verified the claim for all $(n, k) \in [0, N-1] \times [0, \infty)$. Then let us set $n=N$, and prove the statement inductively in $k \ge 0$. The base case $k=0$ has already been checked above. Assuming the case for $(n, k) = (N, K-1)$ is proven, we apply Lemma \ref{lem_duality} to the compatible diagram \ref{diag_proofduality_SU_U}, we get the desired statement for the case $(n, k) = (N, K)$. 
	This finishes the inductive step and completes the proof of Theorem \ref{thm_duality_Sp}. 
\end{proof}

\subsection{The level-rank duality between $U(n)_k$ and $SU(k)_n$}\label{subsec_duality_SU_U}
We now move on to the case of $\cD=U(n)_k$. 
The inductive strategy is exactly the same as the case of $\cD=Sp(n)_k$ proved above, but the verification of the base case is a little more complicated.

Before proceeding to prove the duality between $\TMF[kV_{U(n)}]^{U(n)}$ and $\TMF[n\overline{V}_{SU(k)}]^{SU(k)}$ for general $n$ and $k$, we start by proving the extreme case, which will be used as a part of base step in our inductive proof of the general statement (Theorem \ref{thm_duality_U_SU}). 

\begin{prop}[{The level-rank duality between $U(n)$ and $SU(1) = e$}]\label{prop_base_case_Un}
Let $n$ be a nonnegative integer. The map
\begin{align}
    \chi(V_{U(n)}) \cdot \colon \TMF \to \TMF[V_{U(n)}]^{U(n)}
\end{align}
is an equivalence in $\Mod_\TMF$. 
\end{prop}
\begin{proof}
We already know the case of $n=1$ by Appendix~\ref{subsec_cell_TJF_app}. The stabilization-restriction fiber sequence in Proposition \ref{prop_stab_res_TMF} looks like
\begin{align}\label{seq_TJF1}
	\xymatrix{
		\TMF[1] \ar[r]^-{\tr_e^{U(1)}} \ar@{=}[d] &  \TMF^{U(1)} \ar[r]^-{\chi(V_{U(1)}) } \ar[d]^-{\simeq} & \TMF[V_{U(1)}]^{U(1)} \ar[r]^-{\res_{U(1)}^e}_-{0} & \TMF[2] \\
	\TMF[1]	&\TMF[1] \oplus \TMF \ar[r] \ar[l] & \TMF \ar[u]^-{\simeq}_-{\chi(V_{U(1)})\cdot } &
	}
\end{align}
i.e., split at $\TMF^{U(1)}$, and the third vertical arrow provides the isomorphism claimed in the proposition for $n=1$. 

Now, for each integer $n \ge 2$, consider the inclusion of the standard maximal torus $\iota_n \colon \bT^n \hookrightarrow U(n)$ where we denoted $\bT := U(1)$. It induces the restriction map
\begin{align}\label{eq_f_n}
   \res_{\iota_n} \colon \TMF[V_{U(n)}]^{U(n)} \to \TMF\left[\bigoplus_{i=1}^{n}V_{\bT_i}\right]^{\bT^{n}} 
\end{align}
Here, the we indicated the $i$-th copy of $\bT$ in the group $\bT^{n} $ by $\bT_i$. 
The following diagram commutes. 
\begin{align}
    \xymatrix{
    &\TMF \ar[ld]_-{\chi(V_{U(n)})\cdot} \ar[d]^-{\chi(V_{\bT})^{\otimes n}\cdot}_-{\simeq}  \\
    \TMF[V_{U(n)}]^{U(n)} \ar[r]^-{\res_{\iota_n}} & \TMF\left[\bigoplus_{i=1}^{n}V_{\bT_i}\right]^{\bT^{n}}
    }
\end{align}
and the right vertical arrow is an isomorphism because of the statement of the proposition already checked for $n=1$ above (i.e., \eqref{seq_TJF1}). Thus, to prove the proposition, it is enough to prove that \eqref{eq_f_n} is an isomorphism. We prove it by induction on $n$. 

The base case $n=1$ is trivial. So assume that \eqref{eq_f_n} for $n-1$ is isomorphism. We use the following commutative diagram, 
\begin{align}
    \xymatrix{
    \TMF^{U(n)} \ar[r]^-{\chi(V_{U(n)}) \cdot } \ar[d]_-{\simeq \text{\tiny  by Fact }\ref{factSU} (2)}^-{\res_{\det}} & \TMF[V_{U(n)}]^{U(n)} \ar[dd]^-{\res_{\iota_n}}\ar[rr]^-{\res_{U(n)}^{U(n-1)}} && \TMF[V_{U(n-1)}+2]^{U(n-1)} \ar[dd]^-{\res_{\iota_{n-1}}} \\
    \TMF^{\bT} \ar[d]_-{\simeq \text{\tiny  by \eqref{seq_TJF1}}}^-{\chi(V_{\bT})^{\otimes(n-1)} \cdot}  &&& \\
    \TMF\left[\bigoplus_{i=1}^{n-1}V_{\bT_i}\right]^{\bT^{n-1} \times \bT} \ar[r]_-{ \chi(V_{\bT})} &\TMF\left[\bigoplus_{i=1}^{n}V_{\bT_i}\right]^{\bT^{n}} \ar[rr]^-{\res_{\bT^n}^{\bT^{n-1}}}_-{0} && \TMF\left[\bigoplus_{i=1}^{n-1}V_{\bT_i} + 2\right]^{\bT^{n-1}}
     }
\end{align}
The top row is a fiber sequence by Proposition \ref{prop_stab_res_TMF}, and the bottom row is also a fiber sequence by tensoring $\TMF\left[\bigoplus_{i=1}^{n-1}V_{\bT_i}\right]^{\bT^{n-1}}$ to the sequence \eqref{seq_TJF1}. 
The left vertical arrows are equivalences as indicated. 
The right vertical arrow is an equivalence by the induction hypothesis. Thus the middle vertical arrow is an equivalence, and this completes the inductive step for $n$. This completes the proof of Proposition \ref{prop_base_case_Un}. 
    
\end{proof}

\begin{thm}[{The level-rank duality between $U$ and $SU$}]\label{thm_duality_U_SU}
Let $n$ and $k$ be integers with $n \ge 0$ and $k \ge 1$. Then
$\TMF[kV_{U(n)}]^{U(n)}$ and $\TMF[n\overline{V}_{SU(k)}]^{SU(k)}$ are duals to each other in $\Mod_\TMF$, and the coevaluation map in Definition \ref{def_FD} applied to the data $\cD = U(n)_k$ in Definition \ref{def_Jac_trio}, 
    \begin{align}
    \mathcal{F}_{U(n)_k} \colon \TMF &\xrightarrow{\chi(V_{U(n)} \otimes_\C V_{SU(k)})} \TMF[V_{U(n)} \otimes_\C V_{SU(k)}]^{U(n) \times SU(k)} \\ &\stackrel{\sigma(\Theta_{U(n)_k}, \fraks)}{\simeq} \TMF[kV_{U(n)}]^{U(n)} \otimes_\TMF \TMF[n\overline{V}_{SU(k)}]^{SU(k)},  
\end{align}
is the coevaluation map of the duality. 
\end{thm}

\begin{proof}
    We use an inductive argument, which is exactly parallel to the proof of Theorem \ref{thm_duality_Sp}. 
    For this case, we use the following diagram for $n \ge 1$ and $k \ge 2$ in \eqref{diag_different_nk} specialized to our case. 
      \begin{align}\label{diag_proofduality_SU_U}
        \xymatrix{
        &\TMF \ar[ld]_-{\mathcal{F}_{U(n-1)_k}} \ar[d]^-{\mathcal{F}_{U(n)_k}} \ar[rd]^-{\mathcal{F}_{U(n)_{k-1}}} & \\
        \TMF[(n-1)\overline{V}_{SU(k)}-2k]^{SU(k)} \ar[r]_-{\chi(V_{SU(k)}) \cdot} \ar@{}[d]|{\otimes}& \TMF[n\overline{V}_{SU(k)}]^{SU(k)}\ar[r]_-{\res} \ar@{}[d]|{\otimes}&
        \TMF[n\overline{V}_{SU(k-1)}]^{SU(k-1)}\ar@{}[d]|{\otimes} \\
        \TMF[kV_{U(n-1)}+2k]^{U(n-1)} & 
        \TMF[kV_{U(n)}]^{U(n)} \ar[l]^-{\res}&
        \TMF[(k-1)V_{U(n)}]^{U(n)} \ar[l]^-{\chi(V_{U(n)})\cdot}
        }
    \end{align}

    Here the second and third rows are the stabilization-restriction fiber sequences in Proposition \ref{prop_stab_res_TMF}. 
    By Proposition \ref{prop_F_different_nk}, both the left and the right half of the diagram \eqref{diag_proofduality_SU_U} is {\it compatible}, in the sense of Section \ref{subsec_notations} \eqref{prelim_compatible}. 
Using this result and a general Lemma \ref{lem_duality} below, we prove the desired statement by induction on $n$, and within that, we induct on $k$. The base case $n=0$ easily follows by Fact \ref{factSU}. 

Now, suppose that we have verified the claim for all $(n, k) \in [0, N-1] \times [1, \infty)$. Then let us set $n=N$, and prove the statement inductively in $k \ge 1$. The base case $k=1$ is done by Proposition \ref{prop_base_case_Un}. Assuming the case for $(n, k) = (N, K-1)$ is proven, we apply Lemma \ref{lem_duality} to the compatible diagram \ref{diag_proofduality_SU_U}, we get the desired statement for the case $(n, k) = (N, K)$. 
This finishes the inductive step, and completes the proof of Theorem \ref{thm_duality_U_SU}. 

\end{proof}

\subsection{A lemma on duality}\label{subsec_duality_lem}
Here, we prove a general lemma which was used in our inductive proof of Theorem \ref{thm_duality_Sp} and Theorem \ref{thm_duality_U_SU} above. 
\begin{lem}\label{lem_duality}
    Let $R$ be an $E_\infty$ ring spectrum, and suppose that we are given two fiber sequences in $\Mod_R$, 
    \begin{align}
        a_1 \xrightarrow{\alpha} a_2 \xrightarrow{\beta} a_3, \\
        b_1 \xleftarrow{\gamma} b_2 \xleftarrow{\delta} b_3. \label{bbb}
    \end{align}
    Assume that $a_i$ and $b_i$ are dual to each other in $\Mod_R$ for $i = 1$ and $3$, with coevaluation maps
    \begin{align}
        \mathrm{coev}_i \colon R \to a_i \otimes_R b_i, \ i=1, 3. 
    \end{align}
    Furthermore, assume that we are given a morphism
    \begin{align}
        f \colon R \to a_2 \otimes_R b_2
    \end{align}
    with which both the left and the right half of the following diagram
    \begin{align}\label{diag_dualitylem}
        \xymatrix{
        &R \ar[ld]_-{\coev_1} \ar[d]^-{f} \ar[rd]^-{\coev_3}& \\
        a_1 \ar@{}[d]|{\otimes_R} \ar[r]^-{\alpha}& a_2 \ar@{}[d]|{\otimes_R} \ar[r]^-{\beta} & a_3 \ar@{}[d]|{\otimes_R}\\
        b_1 & b_2 \ar[l]^-{\gamma} & b_3\ar[l]^-{\delta}
        } 
    \end{align}
    is compatible in the sense of Section \ref{subsec_notations} \eqref{prelim_compatible}. 
    Then $a_2$ and $b_2$ are duals to each other (in particular they are dualizable), and $f$ is the coevaluation map associated to the duality. 
\end{lem}
\begin{proof}
    Since we know that $b_1$ and $b_3$ are dualizable and we have a fiber sequence \eqref{bbb}, we conclude that $b_2$ is also dualizable. For dualizable objects $x$, let us denote its dual by $D_R(x)$ and use notation $\Hom_{\Mod_R}(1, x \otimes_R y) \simeq \Hom_{\Mod_R}(D_R(x), y), \ g \mapsto g' $. 
    Then the compatibility of the diagram \eqref{diag_dualitylem} is equivalent to the commutativity of the following diagram. 
    \begin{align}
        \xymatrix{
        D_R(b_1) \ar[r] \ar[d]^-{\simeq}_{\coev_1'} & D_R(b_2) \ar[r] \ar[d]_{f'} &D_R(b_3) \ar[d]^-{\simeq}_{\coev_3'} \\
        a_1 \ar[r] & a_2 \ar[r] & a_3
        }
    \end{align}
    Since the rows are fiber sequences and $\coev_i'$ are equivalences for $i=1, 3$, we see that $f'$ is an equivalence, exhibiting the duality between $a_2$ and $b_2$. This completes the proof of Lemma \ref{lem_duality}. 
\end{proof}

\section{Applications}\label{sec_application}
In the Introduction, we explained why we can expect our genuinely equivariant topological elliptic genera to be more interesting than the classical elliptic genera. 
In this section, we give examples to show that all the expected interesting phenomena listed there indeed happen. 

\subsection{The first interesting example: the detection of the $2$-torsion in $\pi_{5}MSp$}\label{subsec_pi5MSp}
First, we give the easiest example which illustrates the various interesting aspects of our topological elliptic genera. Specifically, we construct an example which simultaneously realizes the following items in the Introduction: (1) detecting torsions, (3) detecting unstable elements, and (4) detecting the difference between $Sp$ and $SU$. 
%We deal with a family of $2$-torsion elements in $\pi_{8k-3}MSp$ constructed by Alexander \cite{AlexanderIndecomposable}. We start by explaining the case $k=1$ in detail.

The manifold we consider is the standard generator $\mu_1$ of $\pi_5 MSp$ \cite{ray1972symplectic} \cite{AlexanderIndecomposable}. It is represented by a $5$-sphere $S^5$ equipped with a nontrivial tangential $Sp(1)=SU(2)$-structure as follows.   
Recall that we have $\pi_5 BSp(1) \simeq \Z/2$. Take a representative $P \colon S^5 \to BSp(1)$ in the nontrivial class. We know that the composition
$J \circ P \colon S^5 \to BO$ is nullhomotopic, with a unique nullhomotopy $\psi$ up to homotopy. We trivialize the stable tangent bundle of $S^5$ in the usual way, so that the triple $(S^5, P, \psi)$ is a tangential $Sp(1)=SU(2)$-manifold. 

\begin{defn}
	We define $\widetilde{\mu}_1 \in \pi_5 MTSp(1) = \pi_5 MTSU(2)$ to be the element represented by the triple $(S^5, P, \psi)$ above. 
\end{defn}

Let us consider the following commutative diagram. 

\begin{align}\label{diag_mu1}
	\xymatrix@C=6em{
	MSU \ar[d]^-{\Jac_{U(1)_\infty}} & MTSU(2) \ar[l]_-{\stab_{MTSU}} \ar[d]^-{\Jac_{U(1)_2}} \ar@{=}[r] & MTSp(1) \ar[d]^-{\Jac_{Sp(1)_1}} \ar[r]^-{\stab_{MTSp}}  & MSp \ar[d]^-{\Jac_{Sp(1)_\infty}}  \\
	\TJF_\infty & \TJF_2 \ar[l]^-{\stab_\TJF}  & \TEJF_2 \ar[l]^-{\res_{Sp(1)}^{U(1)}}_-{\simeq \  \mbox{\tiny by Cor.}\ref{cor_TEJF2=TJF2}} \ar[r]_-{\stab_{\TEJF}} & \TEJF_\infty
	}
\end{align}

Here, four of the horizontal arrows are stabilization maps in the internal structure of the trio. 
Let us investigate the images of $\widetilde{\mu}_1$ in the $\pi_5$ of each component of the diagram \eqref{diag_mu1}. 
This element is known to show an interesting behavior in the bordism groups in the upper row, as follows. 

\begin{fact}[{Bordism classes of images of $\widetilde{\mu_1}$}]\label{fact_mu1_bordism}
	\begin{align}
			\widetilde{\mu}_1 \neq 0 &\in \pi_5 MTSp(1) = \pi_5 MTSU(2) \simeq \Z/2, \\
			\mu_1 :=	\stab_{MTSp} (\widetilde{\mu}_1) \neq 0 &\in \pi_5 MSp \simeq \Z/2, \\
					\stab_{MTSU} (\widetilde{\mu}_1) = 0 &\in \pi_5 MSU = 0
	\end{align}

\end{fact}

The goal of this subsection is to show that all the vertical arrows in \eqref{diag_mu1} are injective, so that the topological elliptic genera exactly detect this behavior (Proposition \ref{prop_mu1_JacSp} and Corollary \ref{cor_mu1}). 
First, the bottom row of \eqref{diag_mu1} is understood as follows. 

\begin{prop}\label{prop_mu1_codom}
	\begin{enumerate}
		\item The following restriction map is an isomorphism. 
		\begin{align}\label{eq_pi5TEJF2}
		\res_{Sp(1)}^{e} \colon \pi_5 \TEJF_2 \to  \pi_1 \TMF = \eta \cdot \pi_0 \MF/(2\eta). 
		\end{align}
		In this subsection, we denote by $\widehat{\eta} \in \pi_5 \TEJF_2$ the unique element that maps to $\eta$ under the isomorphism \eqref{eq_pi5TEJF2}. 
		\item We have
		\begin{align}
	 \widehat{\eta} \notin \ker\left( 	\stab_{\TEJF} \colon \pi_5 \TEJF_2 \to \pi_5 \TEJF_\infty \right) . 
		\end{align}
	\item We have $\pi_5 \TJF_\infty = 0$. 
		\end{enumerate}
\end{prop}

\begin{proof}
(1) follows from $\TEJF_2 \simeq \TMF/\nu$ in \eqref{eq_TEJF2_app}. 
	For (2), we use Proposition \ref{prop_Sp(1)} in Appendix which gives an identification
	\begin{align}
		\TEJF_{2k} \simeq \TMF \otimes \HP^{k+1}[-4], \quad \TEJF_\infty \simeq \TMF \otimes \HP^\infty[-4],  
	\end{align}
	By this identification, the stabilization map $\stab_\TEJF \colon \TEJF_{2k}  \to \TEJF_\infty$ corresponds to the map induced by the inclusion $i \colon \HP^{k+1} \hookrightarrow \HP^\infty$. 
	Consider the following commutative diagram. 
	\begin{align}\label{diag_tejf_TEJF}
	\xymatrix{
	\eta \cdot \pi_0 \MF/(2\eta) =	\pi_1 \TMF && \pi_5 \TEJF_2 \ar[ll]_-{\res_{Sp(1)}^e} ^-{\simeq} \ar[rr]^-{\stab_\TEJF}&& \pi_5 \TEJF_\infty \\
		\Z \eta /(2\eta) =	\pi_1 \tmf \ar@{^{(}->}[u]^-{[\Delta^{-24}]}  & & \pi_5 \tmf \otimes \HP^2[-4] \ar[ll]^-{(\HP^{2} \to S^8)_*}^-{\simeq} \ar[rr]^-{(\HP^2 \hookrightarrow \HP^\infty)_*}_-{\simeq} \ar@{^{(}->}[u]^-{[\Delta^{-24}]}   &&  \pi_5 \tmf \otimes \HP^{\infty}[-4]\ar@{^{(}->}[u]^-{[\Delta^{-24}]}  \\
	}
	\end{align}
	Here, the top left horizontal arrow is an isomorphism by (1) of this proposition, proved above. So the bottom left horizontal arrow is also an isomorphism. Moreover, the bottom right horizontal arrow is also an isomorphism, because the map $\HP^2 \hookrightarrow \HP^{\infty}$ is $10$-connected. The vertical arrows are injective. 
By this diagram and the definition of $\widehat{\eta}\in\pi_5 \TEJF_2$, we get (2). 
(3) follows from $\pi_5 \TJF_3 = 0$ which is easily checked by \eqref{eq_TJF_3}. This completes the proof of Proposition \ref{prop_mu1_codom}. 
\end{proof}

We can now specify the images of $\widetilde{\mu}_1$ in the bottom row of \eqref{diag_mu1}. 

\begin{prop}\label{prop_mu1_JacSp}
	We have
	\begin{align}\label{eq_prop_mu1_Sp1}
		\Jac_{Sp(1)_1}(\widetilde{\mu}_1) = \widehat{\eta}  \in   \pi_5 \TEJF_2 \simeq  \widehat{\eta}\cdot \pi_0 \MF/(2 \widehat{\eta}). 
	\end{align}
\end{prop}
\begin{proof}
We use Corollary \ref{cor_Jac_res_Euler}, which gives us the commutative diagram
\begin{align}
	\xymatrix{
		\pi_5	MTSp(1) \simeq \Z\widetilde{\mu}_1/(2 \widetilde{\mu}_1) \ar[rr]^-{\Jac_{Sp(1)_1}}  \ar[d]^-{\res = \chi(V_{Sp(1)}) \cdot}_-{\eqref{eq_res_MTG_j=1}}&& \pi_5 \TEJF_2 = \widehat{\eta}\cdot \pi_0 \MF/(2\widehat{\eta})  \ar[d]_-{\simeq}^-{\res_{Sp(1)}^e \colon \widehat{\eta} \mapsto \eta } \\
		\pi_1	S \simeq \Z\eta/(2\eta) \ar@{>->}[rr]^-{u} && \pi_1 \TMF = \eta \cdot \pi_0 \MF/(2\eta)\\
	}
\end{align}
The right vertical arrow is an isomorphism by Proposition \ref{prop_mu1_codom} (1), and maps $\widehat{\eta}$ to $\eta$. The isomorphism in the upper left corner used Fact \ref{fact_mu1_bordism}. 
Furthermore, we claim that the left vertical arrow maps $\widetilde{\mu}_1$ to $\eta \in \pi_1 S$: this is not difficult to prove directly using Proposition \ref{prop_geom_res}, but we may also use Claim \ref{claim_lf_vs_euler} and the classical result in \cite{AlexanderIndecomposable} that the corresponding Landweber-Novikov operation applied to $\mu_1 = \stab_{MTSp}(\widetilde{\mu}_1) \in \pi_5 MSp$ is the element $\eta \in \pi_1 MSp$. 
This means that the left vertical arrow is an isomorphism. This completes the proof of Proposition \ref{prop_mu1_JacSp}. 
\end{proof}

\begin{cor}\label{cor_mu1}
	We have
	\begin{align}
		\Jac_{U(1)_2}(\widetilde{\mu}_1) = \widehat{\eta} \neq 0  &\in   \pi_5 \TJF_2 \simeq  \widehat{\eta}\cdot \pi_0 \MF /(2 \widehat{\eta}), \\
		\Jac_{Sp(1)_\infty} (\mu_1) = \stab_\TEJF(\widehat{\eta}) \neq 0   &\in \pi_5 \TEJF_\infty , \\
	\Jac_{U(1)_\infty}  \circ	\stab_{MTSU} (\widetilde{\mu}_1) = 0 &\in \pi_5 \TJF_\infty = 0
	\end{align}
	In particular, all the vertical arrows of diagram \ref{diag_mu1} are injective. 
\end{cor}

\begin{proof}
	The three equalities follow from Proposition \ref{prop_mu1_JacSp}, commutativity of the diagram \eqref{diag_mu1} and Proposition \ref{prop_mu1_codom}. The last claim uses Fact \ref{fact_mu1_bordism}. 
\end{proof}

\begin{rem}[The update on arXiv:v2]
	In the arXiv version 1 of this article, we claimed the generalization of the results of this subsection to a family of $2$-torsion elements $\mu_k \in \pi_{8k-3} MSp$ claimed to exist by the work of Alexander \cite{AlexanderIndecomposable}. However, we had not been aware that the original paper of Alexander contained an error and a correction was made in \cite{alexander1977correction}. Our previous analysis used the wrong construction of $Sp$-manifolds which actually do not work. Since the generalization was not essential for the main purpose of this paper, we chose to delete our analysis on the torsion elements in $\pi_{8k-3}MSp$. However, the authors believe that it is still interesting to analyze the images of topological elliptic genera for torsion elements in $\pi_* MSp$. We hope to adress it in future works. 
\end{rem}

\subsection{Divisibility constraints for Euler numbers}\label{subsec_divisibility}

In this subsection, we present a major application of our topological elliptic genus: novel divisibility constraints on Euler numbers. 
This corresponds to the items (2) and (5) in the Introduction.
The main result is Theorem \ref{thm_divisibility_constraints_concrete}, the idea behind which is to use the relation with Euler numbers and $U$ and $Sp$-topological elliptic genera as shown in Corollary \ref{cor_Jac_res_Euler} above.

\subsubsection{The divisibility constraints via the topological elliptic genera}\label{subsubsec_divisibility_top}
Recall the stabilization-restriction fiber sequence in Proposition \ref{prop_stab_res_TMF} for $U(1)$ and $Sp(1)$ \eqref{eq_building_seq_TJF}, \eqref{eq_building_seq_TEJF} :
\begin{align}\label{seq_building_trio}
		\xymatrix{
	  \TJF_{k} \ar[r]^-{\res}  &\TMF[2k] \ar[r]^-{x(k) \cdot} & \TJF_{k-1}[1]   \ar[r]^-\stab  & \TJF_k[1]\\
		 	\TEJF_{2k} \ar[r]^-{\res}  & \TMF[4k] \ar[r]^-{y(k) \cdot} &  \TEJF_{2k-2}[1]    \ar[r]^-\stab & \TEJF_{2k}[1]
	}
\end{align}
Here we have defined
\begin{align}\label{eq_xyz}
	x(k) \in \pi_{2k-1}\TJF_{k-1} , \quad y(k) \in \pi_{4k-1} \TEJF_{2k-2}
\end{align}
to be the element that specifies the cofibers of the restriction maps in \eqref{seq_building_trio}. We call them {\it attaching element} (c.f., Examples \ref{ex_TJF_building}, \ref{ex_TEJF_building}). 
Let us introduce the following notations. 
\begin{defn}\label{def_divisibility}
	For each positive integer $k$, define $d_{SU}(k), d_{Sp}(k)\in \Z_{>0} \cup \{\infty\}$ to be the order of the elements $x(k)$ and $y(k)$ in \eqref{eq_xyz}, respectively. 
\end{defn}

The integers $d_{SU}(k)$ and $d_{Sp}(k)$ capture information of the image of the first arrows in \eqref{seq_building_trio} as follows. 
\begin{prop}\label{prop_d_im}
For each positive integer $k$, we have the following. 
\begin{align}
d_{SU}(k) \cdot \Z =\mathrm{im}	\left( \res_{U(1)}^e \colon \pi_{2k}\TJF_{k} \to  \pi_0 \TMF \right) \bigcap \mathrm{im}\left(  u \colon \Z \hookrightarrow \pi_0 \TMF \right),\label{eq_d_im_SU} \\
d_{Sp}(k) \cdot \Z =\mathrm{im}	\left( \res_{Sp(1)}^e \colon \pi_{4k}\TEJF_{2k} \to  \pi_0 \TMF \right) \bigcap \mathrm{im}\left(  u \colon \Z \hookrightarrow \pi_0 \TMF \right). \label{eq_d_im_Sp}
\end{align}
\end{prop}
\begin{proof}
	This is a direct consequence of the fact that the sequences in \eqref{seq_building_trio} are fiber sequences. 
\end{proof}

The following is the first main result of this subsection. 

\begin{thm}[{Genuine divisibility constraints of the Euler numbers}]\label{thm_divisibility_constraints_genuine}
	Let $k$ be any positive integer. 
	\begin{enumerate}
		\item For any closed manifold $M$ which admits a \emph{strict} tangential $SU(k)$-structure (Definition \ref{def_strict_str}; in particular we necessarily have $\dim_\R M = 2k$), its Euler number $\Eul(M)$ is divisible by $d_{SU}(k)$. 
		\item For any closed manifold $M$ which admits a \emph{strict} tangential $Sp(k)$-structure (so that $\dim_\R M = 4k$), $\Eul(M)$ is divisible by $d_{Sp}(k)$. 
		\end{enumerate}
\end{thm}

\begin{proof}
	The proof is exacty the same for both (1) and (2). Let $(\cG, \cH)$ be $(U, SU)$ and $(Sp, Sp)$ for the cases (1) and (2), and set $N = 2, 4$, respectively. Given an $Nk$-dimensional manifold $M$ with a strict tangential $\cH(k)$-structure $\psi$, by Corollary \ref{cor_Jac_res_Euler} we have
	\begin{align}
\Eul(M) =	\res_{G(1)}^e \circ	\Jac_{G(1)_k}[M, \psi]  \in \Z \hookrightarrow \pi_0 \TMF. 
	\end{align}
	In particular, we have
	\begin{align}
		\Eul(M) \in \mathrm{im}\left( \res_{G(1)}^e \colon \pi_{Nk}\TMF[kV_{G(1)}]^{G(1)} \to  \pi_0 \TMF \right) \bigcap \mathrm{im}\left(  u \colon \Z \hookrightarrow \pi_0 \TMF \right) . 
	\end{align}
From this Proposition \ref{prop_d_im}, we get the desired result. 
\end{proof}

This theorem allows us to deduce divisibility constraints of Euler numbers by analyzing the numbers $d_\cH(k)$, which is purely a question about the trio of equivariant TMF. 

T.Bauer has worked out the computation and determined the numbers $d_\Sp(k)$ and $d_\SU(k)$. The results are summarized in Propositions \ref{prop_y(k)_exact} and \ref{prop_x(k)_app}, and the proof will be given in \cite{BauerComputationTJF}. Especially, we have
\begin{align}
	d_\Sp(k)  = \frac{24}{ \gcd(k, 24)}. 
\end{align}
These results rely on the computation using spectral sequences. 
However, actually we can give an elementary proof of the lower estimate (Proposition \ref{prop_y(k)_app})
	\begin{align}\label{eq_dSp}
	\left.	\frac{24}{ \gcd(k, 24)} \right| d_{Sp}(k). 
\end{align}
This estimate is enough for the proof of the divisiblity result in Theorem \ref{thm_divisibility_constraints_concrete} (1). 
In order for the self-containedness, we include the proof of the result \eqref{eq_dSp} in Appendix \ref{subsec_cell_TEJF}. 
%
%For the $Sp$-case, T.Bauer has 
%in the Appendix, Proposition \ref{prop_y(k)_app}, we show that, for any positive integer $k$, we have
%	\begin{align}
%			\left.	\frac{24}{ \gcd(k, 24)} \right| d_{Sp}(k). 
%	\end{align}
%For the $SU$-case, the numbers $d_{SU}(k)$ are completely determined by Proposition \ref{prop_x(k)_app} in the Appendix. 

Thus we get the following concrete divisibility results. 

\begin{thm}[{Concrete divisibility constraints on the Euler numbers}]\label{thm_divisibility_constraints_concrete}
	~
	\begin{enumerate}
			\item Let $k$ be any positive integer. For any closed manifold $M$ which admits a strict tangential $Sp(k)$-structure (so that $\dim_\R M = 4k$), its Euler number $\Eul(M)$ satisfies
			\begin{align}
					\left.	\frac{24}{ \gcd(k, 24)} \,\right|\, \Eul(M). 
				\end{align}
		\item For any closed manifold $M$ which admits a {\it strict} tangential $SU(k)$-structure, its Euler number $\Eul(M)$ satisfies the following. 
		\begin{enumerate}
			\item If $k=1$, we have $\Eul(M)=0$. 
			\item If $k=2$, we have $24 \,|\, \Eul(M)$. 
			\item For $k \ge 2$, we have
			\begin{align}
			\left.	2^{\alpha(k)} \cdot 3^{\beta(k)} \,\right|\, \Eul(M)
			\end{align}
			with
			\begin{align}
				\alpha(k) = \begin{cases}
					3 & k \equiv 1, 2, 5 \pmod 8 \\
					2 & k \equiv 6, 7 \pmod 8 \\
					1 & k \equiv  3, 4 \pmod 8, \\
					0 & k \equiv  0 \pmod 8. 
				\end{cases}
				\quad 
				\beta(k) = \begin{cases}
					1 & k \equiv 1, 2 \pmod 3 \\
					0 & k \equiv 0 \pmod 3. 
				\end{cases}
			\end{align}
		\end{enumerate}
		\end{enumerate}
\end{thm}

\begin{proof}\footnote{The proof of Proposition \ref{prop_x(k)_app} about the exact determination of the numbers $d_{SU}(k)$ is deferred to the work by T.Bauer \cite{BauerComputationTJF} in preparation. However, in this paper, we have provided some estimates, with self-contained proof, of the numbers $d_{SU}(k)$ in Claim \ref{claim_dSU_easy}. This implies the divisibility result in Theorem \ref{thm_divisibility_constraints_concrete} (2) except for the case of $k \equiv 2$ (mod $8$) with $k \ge 10$. For those cases, Claim \ref{claim_dSU_easy} gives the estimate off by a factor of $2$ compared to the estimate in Theorem \ref{thm_divisibility_constraints_concrete}.    }
	This is obtained by combining Theorem \ref{thm_divisibility_constraints_genuine}, Propositions \ref{prop_y(k)_app} and \ref{prop_x(k)_app}. 
\end{proof}

\begin{rem}[{$K3$ and its Enriques involutions}]
	\label{rem_enriques}
The divisibility for $Sp(1) = SU(2)$-manifolds is saturated by the Euler number of $K3$ surfaces. A subset of $K3$ surfaces enjoy certain fixed-point-free Enriques involutions, the quotients by which give the surfaces $\frac12 K3$. 
While a $\frac12 K3$ is not an $SU$-manifold (the Enriques quotient does not preserve the complex structure of $K3$) and hence outside the domain of \eqref{eq_intro_Jac_clas}, the formula \eqref{eq_char_Jac_clas} was originally formulated for almost complex manifolds, and when applied to $\frac12 K3$ gives $\phi_{0,1}$. However, because $\TJF$ is defined as the genuinely $U(1)$-equivariant twisted $\TMF$, a local $U(1)$ action coming from a nonintegrable almost complex structure of $\frac12 K3$ is ``not good enough'', and indeed $\phi_{0,1}$ does not lie in the image of $e_\JF$ given in \eqref{eq_e_JF_intro}. What is true is that a double cover of the almost complex structure of $\frac12 K3$ gives a global $U(1)$ action, and so $\phi_{0,1}|_{y \to y^2}$ does lie in the image of $e_\JF$ with $k=4$ and $m=8$.
\end{rem}

\subsubsection{Comparison with classical divisibility constraints}\label{subsubsec_divisibility_naive}

The classical elliptic genera \eqref{eq_intro_Jac_clas} already imply divisibility constraints on Euler numbers, which we call the {\it classical divisibility constraints}. This section explains this and compares those constraints with our divisibility results in Section \ref{subsubsec_divisibility_top}. We show that indeed our divisibility constraints strictly refine the classical constraints, especially for strict tangential $SU(k)$-manifolds with $k \equiv 2$ (mod $8$) and strict tangential $Sp(k)$-manifolds for all $k$.

The classical elliptic genera $\Jac_{\clas}$ in \eqref{eq_intro_Jac_clas} take values in integral {\it weak} Jacobi forms \cite{Gritsenko:1999fk}. We use the notation $\jF_k$ introduced in Section \ref{subsec_notations} \eqref{notation_JF}. 
Let us define 
\begin{defn}
	For each nonnegative integer $k$,
	 define $d_{\clas}(k) \in \Z_{\ge 0}$ by 
	\begin{align}
		d_{\clas}(k) := \gcd  \left\{\mathrm{im}\left( \ev_{z=0} \colon  \jF_{k}|_{\deg = 2k} \to \mf|_{\deg = 0} = \Z \right) \right\}
	\end{align}
\end{defn}
Correspondingly to Theorem \ref{thm_divisibility_constraints_genuine}, we get
\begin{prop}[Classical divisibility constraints]\label{prop_naive}
	For any \underline{strict} tangential $SU(k)$-manifold $M$ (of real dimension $2k$), its Euler number $\Eul(M)$ is divisible by $d_{\clas}(k)$. 
\end{prop}
\begin{proof}
	This follows from the classical fact \cite{Gritsenko:1999fk} that
	\begin{align}
		\Eul(M) = \ev_{z = 0} \circ \Jac_\clas(M). 
	\end{align}
\end{proof}

\begin{rem}[{No further classical divisibility for $Sp$} ]\label{rem_naive_Sp}
	Here it is important to remark that the classical elliptic genera do not give any further refinement of the divisibility results for $Sp$-manifolds, since for any $k \in \Z$, all the elements in $ \jF_{2k}|_{\deg = 4k}$ are even in the elliptic coordinate $z$. 
\end{rem}

The collection $\oplus_{k\in\Z_{\ge 0}} \left( \jF_k|_{\deg = 2k}\right) $ of jacobi forms with weight $0$ forms a $\Z$-graded subring of $\oplus_{k} \JF_k$, and expressed as, 
\begin{align}\label{eq_jF_generator}
	\oplus_{k} \left( 	\jF_{k}|_{\deg = 2k} \right)  = \Z[\phi_{0,1}, \phi_{0,\frac32}, \phi_{0,2}, \phi_{0,4}] /\left(4\phi_{0,4}=\phi_{0, 1} \phi_{0 ,\frac32}^2 - \phi_{0,2}^2\right)\subset \eqref{eq_JF_generators}, 
	\end{align}
where the lower indices of each generator correspond to its weight and index. 
The number $d_{\clas}(k)$ is explicitly computable by looking at the generators in \eqref{eq_jF_generator}. Under $\ev_{z=0}$, the generators are mapped as
\begin{align}\label{eq_ev_z=0}
	\phi_{0,1} \mapsto 12, \quad \phi_{0,\frac32} \mapsto 2, \quad \phi_{0,2} \mapsto 6, \quad \phi_{0,4} \mapsto 3,
\end{align}
We can deduce

\begin{prop}[{Computation of the classical divisibility constraints}]\label{prop_naive_constraint_concrete}
	We have
	\begin{enumerate}
		\item We have $d_{\clas} (1) = \infty$. 
		\item For even integer $k = 2k'$ with $k' \ge 1$, we have
		\begin{align}
		d_{\clas}(2k')  = 	\frac{12}{ \gcd(k', 12)}
		\end{align}
		\item For odd integer $k = 2k' + 3$ with $k' \ge 0$, we have
		\begin{align}
		d_{\clas}(2k'+3) =	\frac{24}{ \gcd(k', 12)} . 
		\end{align}
	\end{enumerate}
\end{prop}
\begin{proof}
	This follows from elementary computation using \eqref{eq_ev_z=0}. Details are left to the reader. 
\end{proof}

Now let us compare the classical divisibility constraints $d_{\clas}$ in Proposition \ref{prop_naive_constraint_concrete} with our divisibility constraints in Theorem \ref{thm_divisibility_constraints_concrete} and \ref{thm_divisibility_constraints_genuine}. 
We observe that, for the $SU$-case, we have
\begin{align}
	d_{SU}(k) = \begin{cases}
		2 d_{\clas}(k) & k \equiv 2 \pmod 8, \\
		d_{\clas} (k) & k \not\equiv 2 \pmod 8. 
	\end{cases}
\end{align}
On the other hand, we also see that our result for the $Sp$-case strictly refines the divisibility constraints by the factor of $2$ for {\it all} $k$, (also see Remark \ref{rem_naive_Sp}). 
\begin{align}
	d_{Sp}(k) = 2 d_{\clas}(2k) \quad \forall k. 
\end{align}
To the best knowledge of the authors, this refined divisibility result was not known in the literature. 

\begin{rem}[{Gritsenko's results by \cite{Gritsenko:1999fk}} ] 
	In \cite[Theorem~2.4]{Gritsenko:1999fk} (see also the review article \cite[Proposition~3.1 and the text below]{gritsenko2020modified}), Gritsenko gives the following divisibility results by classical methods. 
	\footnote{Both \cite[Theorem~2.4]{Gritsenko:1999fk} and the review article \cite[Proposition~3.1 and the text below]{gritsenko2020modified} contained misprints, and the correct statement is presented in our main text. 
		The authors thank V. Gritsenko for confirming this.}
	For any almost complex manifold $M$ of even complex dimension $k=2k'$ with $k' \ge 1$, such that the rational first Chern class $c_1(M)_\Q \in H^2(M; \Q)$ vanishes, its Euler number $\Eul(M)$ satisfies
	\begin{align}
		\left.	\frac{12}{ \gcd(k', 12)} \,\right|\, \Eul(M).  
	\end{align}
	If furthermore $k \equiv 2$ (mod $8$) and the integral first Chern class $c_1(M) \in H^2(M; \Z)$ vanishes, making $M$ an $SU(k)$-manifold, then we further have \begin{align}
		8 \,|\, \Eul(M) \quad \mbox{if } k \equiv 2 \pmod 8
	\end{align}
	Restricted to the $SU(k)$-manifolds, we see that this divisibility result coincides with our statement in Theorem \ref{thm_divisibility_constraints_concrete} in the case $k$ even.
\end{rem}

\begin{rem}[Classical divisibility constraints for irreducible hyperk\"ahler manifolds of low dimensions]
	If we restrict ourselves to irreducible hyperk\"ahler manifolds, which furnish a very special class of strict tangential $Sp(k)$-manifolds, we can use the known divisibility constraints on the Hodge numbers to refine the classical divisibility constraints obtained in Proposition \ref{prop_naive_constraint_concrete}. We illustrate it by the case of $k=2$ and $k=3$. 
	As we will see, for those cases we achieve our divisibility constraints in Theorem \ref{thm_divisibility_constraints_concrete} (1). But the method is already complicated there, and gets more and more complicated as $k$ is increased. 
	Moreover, we emphasize that such analysis does not work for general strict tangential $Sp$-manifolds, since we cannot write Euler numbers in terms of an almost complex version of the Hodge numbers. 
	This should be compared with our simple and conceptual proof of the corresponding divisibility, treating general strict tangential $Sp$-manifolds all at once. The authors believe this illustrates the power of topological refinements of classical concepts. 

We begin with the following facts:
\begin{itemize}
	\item An irreducible hyperk\"ahler manifold $M$ of real dimension $4k$ has Hodge numbers \cite{enoki1990compact}
	\[
	h^{0,q} = \begin{cases} 1, & q \text{ is even and } 0 \le q \le 2k,
		\\
		0, & \text{otherwise},
	\end{cases}
	\]
	and symmetry $h^{p,q} = h^{q,p} = h^{p,2k-q}$. 
	\item The constant Fourier component (in $\tau$) of the elliptic genus of a complex manifold $M$ of real dimension $2k$ is related to its Hodge numbers by 
	\begin{align}
		\Jac_{\clas}(M) = \sum_{p=0}^k c_p y^{p-\frac{k}{2}} + \mathcal{O}(q),
		\quad
		c_p = \sum_{q=0}^k (-)^{p+q} h^{p,q}.
	\end{align}
\end{itemize}
\paragraph{The case of $k=2$} 
For $k=2$, there are three independent Hodge numbers $h^{1,1}, h^{1,2}, h^{2,2}$. The Betti and Euler numbers are
\ie
    b_0 = b_8 &= 1, \\
    b_1 = b_7 &= 0, \\
    b_2 = b_6 &= 2 + h^{1,1} \\
    b_3 = b_5 &= 2h^{1,2} \\
    b_4 &= 2 + 2h^{1,1} + h^{2,2}  \\
    \Eul(M) &= 2 + 2b_2 - 2b_3 + b_4 = 8 + 4h^{1,1} - 4h^{1,2} + h^{2,2}.
\fe
The elliptic genus is written as
\ie
    \Jac_{\clas}(M) = 3\phi_{0,1}^2 + \left( \frac{\Eul(M)}{6} - 72 \right) \phi_{0,2},
\fe
which by \eqref{eq_ev_z=0} already shows
\begin{align}\label{eq_hyp2_6}
6 \,|\, \Eul(M).
\end{align}
Expanding the elliptic genus in $q,y$, we find
\ie
    2h^{1,1}-h^{1,2} &= \left( \frac{\Eul(M)}{6} - 12 \right),
    \\
    2 - 2h^{1,2} + h^{2,2} &= \left( \frac{2\,\Eul(M)}{3} + 18 \right),
\fe
and eliminating $\Eul(M)$ gives
\ie
    h^{2,2} &= 64 + 8h^{1,1} - 2h^{1,2}.
\fe
By \cite[Corollary~8.1]{wakakuwa1958riemannian}, $4 \,|\, b_3 = 2h^{1,2}$, so $4 \,|\, \Eul(M)$. Combining this with \eqref{eq_hyp2_6}, we deduce
\begin{align}
	12 \,|\, \Eul(M). 
\end{align}
This divisibility result coincides with our divisibility result in Theorem \ref{thm_divisibility_constraints_concrete} (1). 

\paragraph{The case of $k=3$} For $k=3$, we compute
\ie
    b_0 = b_{12} &= 1, \\
    b_1 = b_{11} &= 0, \\
    b_2 = b_{10} &= 2 + h^{1,1} \\
    b_3 = b_9 &= 2h^{1,2} \\
    b_4 = b_8 &= 2 + 2h^{1,3} + h^{2,2} \\
    b_5 = b_7 &= 2h^{1,2} + 2h^{2,3} \\
    b_6 &= 2 + 2h^{1,1} + 2h^{2,2} + h^{3,3} \\
    \Eul(M) &= 2 + 2b_2 - 2b_3 + 2b_4 - 2b_5 + b_6 \\
    &= 12 + 4 h^{1,1} - 8 h^{1,2} + 4 h^{2,2} + 4 h^{1,3} - 4 h^{2,3} + h^{3,3}.
\fe
The elliptic genus is parameterized by $A \in \Z$ as
\ie
    \Jac_{\clas}([M]) = 4\phi_{0,1}^3 + A \phi_{0,1} \phi_{0,2} + \left( \frac{\Eul(M)}{4} - 18A - 1728 \right) \phi_{0,3}. 
\fe
Expanding the elliptic genus in $q,y$, we find
\ie
    2h^{1,1} - 2h^{1,2} + h^{1,3} &= 120 + A, \\
    2 - 2h^{1,2} + 2h^{2,2} - h^{2,3} &= - 516 - 4A + \frac{\Eul(M)}{4}, \\
    2h^{1,3} - 2h^{2,3} + h^{3,3} &= 784 + 6A + \frac{\Eul(M)}{2}.
\fe
Eliminating $A$ using the last two equations gives
\ie
    -14 - 6h^{1,2} + 6h^{2,2} - 3h^{2,3} + 4h^{1,3} - 4h^{2,3} + 2h^{3,3} &= \frac{7\,\Eul(M)}{4}.
\fe
Furthermore, we have $4 \,|\, b_3, b_5$ by \cite[Corollary~8.1]{wakakuwa1958riemannian}, which implies $2 \,|\, h^{2,3}$.
This shows we have

\begin{align}
8 \,|\, \Eul(M). 
\end{align}
This divisibility result coincides with our divisibility result in Theorem \ref{thm_divisibility_constraints_concrete} (1). 

\end{rem}

\appendix

\section{A user guide to $\TJF$}\label{app_TJF}
The theory of {\it Topological Jacobi Forms} is developed in an upcoming work by Bauer-Meier \cite{BauerMeierTJF}. It is defined as the genuinely $U(1)$-equivariant twisted $\TMF$, and regarded as a spectral refinement of the classical ring of integral Jacobi Forms. 
It is an essential tool for us, being the domain of the $U(1)$-topological elliptic genus $\Jac_{U(1)_k} \colon MTSU(k) \to \TJF_k$. 
In this section, we collect the necessary results on $\TJF_k$, which we heavily use in the main text. 

\subsection{Definition and basic properties}\label{subsec_TJF_defprop_app}

We employ the following as the definition of Topological Jacobi Forms. 
\begin{defn}[{$\TJF_k$}]\label{def_TJF}
	Let $k$ be any integer. We define
	\begin{align}
	\TJF_k := \TMF[kV_{U(1)}]^{U(1)} , 
	\end{align}
	where $V_{U(1)}$ is the fundamental representation of $U(1)$. 
\end{defn}
As explained in Section \ref{subsec_preliminary_TMF_G}, the right hand side is by definition identified as
\begin{align}
	\TMF[kV_{U(1)}]^{U(1)} = \Gamma\left( \cE, \cL(-kV_{U(1)})\right) ,  
\end{align}
where $\cE \to \moduli$ is the universal oriented elliptic curve (in the spectral algebro-geometric sense), and $\cL(-kV_{U(1)}) \in \QCoh(\cE)^\times$ is the result of applying the $U(1)$-equivariant elliptic cohomology functor \eqref{eq_EllG} to the corresponding representation sphere. 
As explained in \cite{BauerMeierTJF}, it is easy to verify that we have a canonical isomorphism $\cL(-kV_{U(1)}) \simeq \cO_\cE(ke)= \cO_\cE(e)^{\otimes k} \in \QCoh(\cE)^\times$, where $\cO_{\cE}(e)$ is the (SAG-version of the) sheaf of meromorphic functions on $\cE$ having poles of order at most $1$ at the zero section $e \colon \moduli \to \cE$. Thus we have a canonical identification
\begin{align}\label{eq_TJF_app_ke}
	\TJF_k := \TMF[kV_{U(1)}]^{U(1)} \simeq \Gamma(\cE, \cO_{\cE}(ke)), 
\end{align}

The equivariant Euler class of the fundamental representation,
\begin{align}
	\chi(V_{U(1)}) \in \pi_0 \TJF_1
\end{align}
is of particular importance. As we will see in Section \ref{subsec_JF_app} \eqref{eq_a=chi_app}, this element corresponds to one of the generators $a :=\phi_{-1, \frac12} = \theta_{11}(z, q)/\eta(q)^3$ of the integral Jacobi Forms.\footnote{Also of physical importance because it is expected to correspond to the complex Majorana fermion. }
It is also important to note that we have the multiplication map
\begin{align}
	\cdot \colon \TJF_k \otimes_\TMF \TJF_m \to \TJF_{k+m}, 
\end{align}
so that $\oplus_k \TJF_k$ can be regarded as a $\Z$-graded ring object in $\Mod_\TMF$. 

Recall that we have seen in our main text (Section \ref{subsubsec_internal} Example \ref{ex_TJF_building}) that the {\it internal structure maps} for the trio of equivariant $\TMF$ specialize to produce the {\it stabilization sequence} of $\TJF$ (see Section \ref{subsec_notations} \eqref{notation_chi} for the notation $\chi(V_{U(1)})$), 
\begin{align}\label{eq_building_seq_TJF_app}
	\xymatrix@C=3em{
		\TJF_{-1} \ar[r]^-{\stab}_-{\chi(V_{U(1)}) \cdot} \ar@{=}[d]&	\TJF_0 \ar[r]^-{\stab}_-{\chi(V_{U(1)}) \cdot} \ar[d]^-{\res_{U(1)}^e}& \TJF_1 \ar[r]^-{\stab}_-{\chi(V_{U(1)}) \cdot} \ar[d]^-{\res_{U(1)}^e}& \TJF_2 \ar[r]^-{\stab}_-{\chi(V_{U(1)}) \cdot} \ar[d]^-{\res_{U(1)}^e}& \TJF_3 \ar[r]^-{\stab}_-{\chi(V_{U(1)}) \cdot}  \ar[d]^-{\res_{U(1)}^e}& \cdots \\
		\TMF[1]&	\TMF &	\TMF[2]   & \TMF[4] & \TMF[6]   & \cdots
	}, 
\end{align}
where each consecutive pair of horizontal and vertical arrows forms a fiber sequence which we call the {\it stabilization-restriction fiber sequence}. 
This sequence is regarded as building $\TJF_{k}$ by attaching even dimensional $\TMF$-cells one by one. 
In view of the identification \ref{eq_TJF_app_ke}, the algebro-geometric meaning of this sequence is nicely understood by the following commutative diagram, 
\begin{align}\label{seq_res_stab_TJF_app}
	\xymatrix{
		\TJF_{k-1} \ar[r]^-{\stab}_-{\chi(V_{U(1)})\cdot} \ar@{=}[d]& \TJF_{k} \ar[r]^-{\res_{U(1)}^e} \ar@{=}[d] & \TMF[2k] \ar@{=}[d] \ar[r]^-{x(k)}& \TJF_{k-1}[1]\\
		\Gamma(\cE, \cO_{\cE}((k-1)e)) \ar[r] & \Gamma(\cE, \cO_{\cE}(ke)) \ar[r] & \Gamma(\moduli, \omega^{-k}) &, 
	}
\end{align}
where the first bottom horizontal arrow comes from the canonical map $\cO_\cE((k-1)e) \to \cO_\cE(ke)$, and the second one is the residue pairing. 
We have defined, in \eqref{eq_xyz}, the {\it attaching element}
\begin{align}\label{eq_x(k)_app}
	x(k) \in \pi_{2k-1} \TJF_{k-1}
\end{align}
to be the cofiber of $\res_{U(1)}^e$ in \eqref{seq_res_stab_TJF_app}. This is the attaching map of the top $\TMF$-cell of $\TJF_k$, which can also be identified with the transfer map $\tr_e^{U(1)}$ (see \eqref{diag_selfdual_TJF} below). The analysis of this element played a key role in our application to Euler numbers in Section \ref{subsec_divisibility}. 
Moreover, it is important to note that the stabilization-restriction sequence is {\it dual} to that of $k$ replaced by $1-k$, in the sense that the following diagram commutes by Proposition \ref{prop_selfdual_stabres}. 
\begin{align}\label{diag_selfdual_TJF}
	\xymatrix@C=5em{
	\TJF_{k-1} \ar[r]^-{\stab}_-{\chi(V_{U(1)})\cdot} \ar[d]^{\simeq}& \TJF_{k} \ar[r]^-{\res}\ar[d]^{\simeq}& \TMF[2k] \ar@{=}[d] \ar[r]_-{x(k)}^-{\tr} & \TJF_{k-1}[1] \ar[d]^{\simeq}\\
	D(\TJF_{1-k})[1] \ar[r]^-{D(\stab)}_-{\chi(V_{U(1)})\cdot} & D(\TJF_{-k})[1] \ar[r]_-{D(x(-k) )}^-{D(\tr)} & \TMF[2k] \ar[r]^-{D(\res)} & D(\TJF_{1-k})[2]\\
	}
\end{align}
Here we have used the dualizability of equivariant $\TMF$ in \eqref{eq_twisted_dual}. 
In particular, the commutativity of the right square allows us to identify the top right horizontal arrow with the transfer map as indicated in the diagram. 

\subsection{The cell structure}\label{subsec_cell_TJF_app}
The following explicit knowledge of the cell structure of $\TJF$ is key to our analysis. 

\begin{fact}\label{fact_TJF_cellstr}
	Let $k \ge 1$ be any positive integer. 
	Let $\tr \colon \Sigma \bC P^\infty_+ \to S^0$ be the $U(1)$-equivariant transfer map. Define
	\begin{align}\label{eq_def_Pm}
		P_k := \cofib\left( \tr|_{\Sigma \CP^{k-1}} \colon \Sigma \bC P^{k-1} \to S^0\right)
	\end{align}
	Then we have an isomorphism of $\TMF$-modules
	\begin{align}\label{eq_TJF_cellstr}
		\TJF_k \simeq \TMF \otimes P_k.
	\end{align}
	Moreover, the isomorphism \eqref{eq_TJF_cellstr} is compatible with the stabilization-restriction fiber sequence in \eqref{seq_res_stab_TJF_app} in the sense that the following diagram commutes, 
	\begin{align}\label{eq_attaching_P}
		\xymatrix@C=3em{
			\TJF_{k-1} \ar[r]^-{\stab}_-{\chi(V_{U(1)})\cdot} & \TJF_{k} \ar[r]^-{\res_{U(1)}^e} & \TMF[2k]  \ar[r]^-{x(k)} & \TJF_{k-1}[1]\\
			P_{k-1} \ar[u]^-{\TMF \otimes - } \ar@{^{(}->}[r] & P_k \ar@{->>}[r] \ar[u]^-{\TMF \otimes - } & S^{2k} \ar[u]^-{\TMF \otimes - } \ar[r]& P_{k-1}[1]\ar[u]^-{\TMF \otimes - }
		}
	\end{align}
	where the bottom row is the cofiber sequence induced by the standard inclusion $\CP^{k-2} \hookrightarrow \CP^{k-1}$.
\end{fact}

$\TJF_{k}$ for negative $k$ is also understood by using the above Fact \ref{fact_TJF_cellstr}. We have (here $D$ denotes the dual in $\Mod_\TMF$ and $D_S$ denotes the dual in $\Spectra$, see Notation \ref{subsec_notations} \eqref{notation_dual})
\begin{align}
	\TJF_k \simeq D(\TJF_{-k})[1] \simeq \TMF \otimes D_S(P_{-k})[1]
\end{align}
by the dualizability of equivariant $\TMF$ in \eqref{eq_twisted_dual}. 

Here, we give the sketch of the proof of this fact, in order to make the meaning of this cell structure clear, and also to prepare for the analogous argument showing the cell structure of $\TEJF$ in Proposition \ref{prop_Sp(1)} below.  

\begin{proof}[Sketch of Proof of Fact \ref{fact_TJF_cellstr} \cite{BauerMeierTJF}]
	Consider the following cofiber sequence of pointed $U(1)$-spaces, 
	\begin{align}
		S(kV_{U(1)})_+ \to S^{0} \xrightarrow{\chi(kV_{U(1)})} S^{kV_{U(1)}}. 
	\end{align}
	Apply the $U(1)$-equivariant $\TMF$-homology functor $\left( \TMF \otimes - \right)^{U(1)}$ to get the fiber sequence
	\begin{align}\label{eq_prf_cellTJF}
	\TMF \otimes \Sigma \CP^{k-1}_+ \to \TMF^{U(1)} \to \TJF_k 
	\end{align} 
	where we have used that $S(kV_{U(1)})$ is a free $U(1)$-space with $\CP^{k-1} = S(kV_{U(1)})/U(1)$, so that we can apply the Adams isomorphism 
	\begin{align}
		\left( \TMF \otimes S(kV_{U(1)})_+\right)^{U(1)}  \simeq \TMF \otimes \Sigma \CP^{k-1}_+. 
	\end{align}
	We know by \cite{GepnerMeier} (also see Fact \ref{factSU} (2)), that we have $\TMF^{U(1)} \simeq \TMF \oplus \TMF[1]$, 
	and we can verify that the first arrow in \eqref{eq_prf_cellTJF} is given by tensoring $\TMF$ to the map
	\begin{align}
		\Sigma\CP^{k-1}_+ = \Sigma \CP^{k-1} \sqcup S^1  \xrightarrow{\tr \sqcup \id_{S^1}} S^0 \sqcup S^1. 
	\end{align}
	This gives the desired result. 
\end{proof}

The cell complex $P_k$ looks identical to $\Sigma \CP^{k-1}$, except for the lower dimensional cells. The stable attaching maps of $\CP^{k-1}$ can be read off from \cite[Theorem 5.2]{Mosher}. The cell diagram of $\TJF_k$ for $-1 \le k \le 6$ is depicted in Figure \ref{celldiag_TJF}. 
Each dot labeled by an integer $n$ denotes one $\TMF$-cell in degree $n$. 

\begin{figure}[H]
	\centering
	\begin{tikzpicture}[scale=0.5]
		
			\begin{scope}[shift={(-8, 0)}]
			\begin{celldiagram}
				\n{1}
				\foreach \y in {1} {
					\node [left] at (-0.5, \y) {$\y$};
				}
			\end{celldiagram}
			\node [] at (0, -2) {$\TJF_{-1}$};
		\end{scope}
		
			\begin{scope}[shift={(-4, 0)}]
			\begin{celldiagram}
				\n{0} \n{1}
				\foreach \y in {0,1} {
					\node [left] at (-0.5, \y) {$\y$};
				}
			\end{celldiagram}
			\node [] at (0, -2) {$\TJF_{0}$};
		\end{scope}
		
			\begin{scope}[shift={(0, 0)}]
			\begin{celldiagram}
				\n{0} 
				\foreach \y in {0} {
					\node [left] at (-0.5, \y) {$\y$};
				}
			\end{celldiagram}
			\node [] at (0, -2) {$\TJF_{1}$};
		\end{scope}
		
			\begin{scope}[shift={(4, 0)}]
			\begin{celldiagram}
				\nu{0}
				\n{0} \n{4} 
				\foreach \y in {0, 4} {
					\node [left] at (-0.5, \y) {$\y$};
				}
			\end{celldiagram}
			\node [] at (0, -2) {$\TJF_{2}$};
		\end{scope}
		\node [right, blue] at (5, 2) {$\nu$};
		
			\begin{scope}[shift={(8, 0)}]
			\begin{celldiagram}
				\nu{0} \eta{4}
				\n{0} \n{4} \n{6} 
				\foreach \y in {0, 4, 6} {
					\node [left] at (-0.5, \y) {$\y$};
				}
			\end{celldiagram}
			\node [] at (0, -2) {$\TJF_{3}$};
		\end{scope}
			\node [right, blue] at (9, 2) {$\nu$};
				\node [right, red] at (8, 5) {$\eta$};
				
					\begin{scope}[shift={(12, 0)}]
					\begin{celldiagram}
						\nu{0} \eta{4} \nu{4}
						\n{0} \n{4} \n{6} \n{8}
						\foreach \y in {0, 4, 8} {
							\node [left] at (-0.5, \y) {$\y$};
						}
					\end{celldiagram}
					\node [] at (0, -2) {$\TJF_{4}$};
				\end{scope}
				\node [right, blue] at (13, 2) {$\nu$};
				\node [right, red] at (12, 5) {$\eta$};
					\node [left, blue] at (11, 6) {$2\nu$};
		
		\begin{scope}[shift={(16, 0)}]
		\begin{celldiagram}
			\nu{0} \eta{4} \nu{4} \eta{8} \nu{6}
			\n{0} \n{4} \n{6} \n{8} \n{10}
			\foreach \y in {0, 4, 8, 10} {
				\node [left] at (-0.5, \y) {$\y$};
			}
		\end{celldiagram}
		\node [] at (0, -2) {$\TJF_{5}$};
	\end{scope}
	\node [right, blue] at (17, 2) {$\nu$};
	\node [right, red] at (16, 5) {$\eta$};
	\node [left, blue] at (15, 6) {$2\nu$};
		\node [left, red] at (16, 9) {$\eta$};
			\node [right, blue] at (17, 8) {$\nu$};
			
				\begin{scope}[shift={(20, 0)}]
				\begin{celldiagram}
					\nu{0} \eta{4} \nu{4} \eta{8} \nu{6} \nu{8}
					\n{0} \n{4} \n{6} \n{8} \n{10} \n{12}
					\foreach \y in {0, 4, 8, 12} {
						\node [left] at (-0.5, \y) {$\y$};
					}
				\end{celldiagram}
				\node [] at (0, -2) {$\TJF_{6}$};
			\end{scope}
			\node [right, blue] at (21, 2) {$\nu$};
			\node [right, red] at (20, 5) {$\eta$};
			\node [left, blue] at (19, 6) {$2\nu$};
			\node [left, red] at (20, 9) {$\eta$};
			\node [right, blue] at (21, 8) {$\nu$};
		\node [left, blue] at (19, 6) {$2\nu$};
			\node [left, blue] at (19, 10) {$3\nu$};
		
	\end{tikzpicture}
	\caption{The cell diagram of $\TJF_k$}\label{celldiag_TJF}
\end{figure}

This means that for example, we have
\begin{align}
	\TJF_0 &\simeq \TMF \oplus \TMF[1], \label{eq_TJF0}\\
	\TJF_1 &\simeq \TMF \ \mbox{where the isom is given by } \chi(V_{U(1)}) \colon \TMF \to \TJF_1 ,  \label{eq_TJF1}  \\
	\TJF_2 &\simeq \TMF \otimes (S^0 \cup_\nu S^4)= \TMF / \nu  \label{eq_TJF_2}, \\
	\TJF_3 &\simeq \TMF \otimes (S^0 \cup_\nu S^4 \cup_\eta S^6), \label{eq_TJF_3} 
\end{align}
and we can read off the attaching elements $x(k) \in \pi_{2k-1} \TJF_{k-1}$ in \eqref{eq_x(k)_app} as
\begin{align}
	x(1) &= (0, 1) \in \pi_1 \TJF_0 \simeq \pi_1 \TMF \oplus \pi_0 \TMF, \label{eq_x(1)}\\
	x(2) &= \nu  \in \pi_3 \TMF \stackrel{\chi(V_{U(1)}) \cdot}{\simeq} \pi_3 \TJF_1 \label{eq_x(2)}\\
	x(3) &= \widehat{\eta} \in \pi_5 \TJF_2, 
\end{align}
where $\widehat{\eta} \in \pi_5 \TJF_2$ is the element appeared in Section \ref{subsec_pi5MSp}, which is the unique element which maps to $\eta \in \pi_1 \TMF$ by the restriction map $\res_{U(1)}^e \colon \TJF_2 \to \TMF[4]$.

If we invert the prime $2$, we get the following simple result. 
\begin{prop}[{The structure of $\TJF_k$ after inverting $2$ \cite{LinTominagaYamashita}}]\label{prop_structure_TJF_2inverted}
	After inverting $2$, the $\TMF$-module structure of $\TJF_k$ is identified as follows.
	\begin{enumerate}
		\item The stabilization-restriction sequence \eqref{seq_res_stab_TJF_app} for $k=3$ splits at $\TJF_3$, 
		\begin{align}
			\xymatrix{
				\TJF_2 \ar[r]^-{\stab} & \TJF_3 \ar[r]^-{\res_{U(1)}^e} & \TMF[6] \ar@/^1ex/[l]^-{\mathfrak{b}}. 
			}
		\end{align}
		Here we have denoted $\frb \in \pi_6 \TJF_3$ which gives a splitting. This element is characterized by the Jacobi Form image as
		\begin{align}
			e_\JF(\frb) = \frac{1}{2}\phi_{0, \frac32}. 
		\end{align}
		\item For $k \ge 4$, setting $k' := \lfloor (m-1)/3 \rfloor$, there is an isomorphism of $\TMF$-modules,
		\begin{align}\label{eq_deco_TJFm}
			\TJF_k \simeq \TJF_{k-3k'}[6k'] \oplus \bigoplus_{i=0}^{k'-1} \TMF_1(2) [6i] .
		\end{align}
	\end{enumerate} 
\end{prop} 
We have more to say on this decomposition after our analysis of the relation between $\TEJF$ and $\TJF$ in Section \ref{subsec_app_TJFvsTEJF}.

The more detailed computation of homotopy groups of $\TJF_k$ will appear in upcoming works by Bauer-Meier and Tominaga \cite{BauerMeierTJF} \cite{Tominaga}. 
In Section \ref{subsec_divisibility} of this paper, we use the following computational result by Bauer \cite{BauerComputationTJF} on the order of the attaching element $x(k) \in \pi_{2k-1}\TJF_{k-1}$ introduced in \eqref{eq_x(k)_app}. 

\begin{prop}[{\cite{BauerComputationTJF}}]\label{prop_x(k)_app}
	The order $d_{SU}(k)$ (Definition \ref{def_divisibility}) of the The attaching element $x(k) \in \pi_{2k-1} \TJF_{k-1}$ in \eqref{eq_x(k)_app} is given as follows. 
	We have
	\begin{align}
		d_{SU}(1) = \infty,  
	\end{align}
	and for $k \ge 2$, we have
	\begin{align}\label{eq_dSU_app}
	d_{SU}(k) = 2^{\alpha(k)} \cdot 3^{\beta(k)}
 	\end{align}
 	with
 	\begin{align}
 		\alpha(k) = \begin{cases}
 			 3 & k \equiv 1, 2, 5 \pmod 8 \\
 			 2 & k \equiv 6, 7 \pmod 8 \\
 			 1 & k \equiv  3, 4 \pmod 8, \\
 			 0 & k \equiv  0 \pmod 8. 
 		\end{cases}
 		\quad 
 		\beta(k) = \begin{cases}
 			1 & k \equiv 1, 2 \pmod 3 \\
 			0 & k \equiv 0 \pmod 3. 
 		\end{cases}
 	\end{align}
\end{prop}

\begin{rem}[{Easy estimates of $d_{SU}(k)$}]\label{rem_estimate_dSU}
	The proof of Proposition \ref{prop_x(k)_app} depends on spectral sequence computations. But without such effort, we can give a substantial lower bound to $d_{SU}(k)$'s in an elementary way, using our knowledge of the cell structures of $\TJF$. Here we present such a result with a proof. 
	\begin{claim}[{Easy estimates of $d_{SU}(k)$}]\label{claim_dSU_easy}~
		We have
		\begin{align}\label{eq_dSU1_2}
			d_{SU}(1) = \infty, \quad d_{SU}(2) = 24 ,
		\end{align}
		and for $k \ge 3$, we have
		\begin{align}
	\left.	2^{\alpha'(k)} \cdot 3^{\beta'(k)} \right| d_{SU}(k) , 
		\end{align}
		with 
			\begin{align}
			\alpha(k) = \begin{cases}
				3 & k \equiv 1, 5 \pmod 8 \\
				2 & k \equiv 2, 6, 7 \pmod 8 \\
				1 & k \equiv  3, 4 \pmod 8, \\
				0 & k \equiv  0 \pmod 8. 
			\end{cases}
			\quad 
			\beta(k) = \begin{cases}
				1 & k \equiv 1, 2 \pmod 3 \\
				0 & k \equiv 0 \pmod 3. 
			\end{cases}
		\end{align}
	\end{claim}
	Note that, compared to the result in Proposition \ref{prop_x(k)_app}, the estimate in Claim \ref{claim_dSU_easy} is sharp except for the case $k \equiv 2$ (mod $8$) for $k \ge 10$, and off by the factor $2$ for those cases. 
	
	\begin{proof}[Proof of Claim \ref{claim_dSU_easy}]
		We use the knowledge of cell structures of $\TJF$ explained in Section \ref{subsec_cell_TJF_app}. 
		\eqref{eq_dSU1_2} follows by \eqref{eq_x(1)} and \eqref{eq_x(2)}. 
		Let us prove the case for $k \ge 3$. 
		First, the estimate on the $3$-valuation is easily obtained by Proposition \ref{prop_structure_TJF_2inverted}. 
		So let us focus on the $2$-valuation. We separate the case of $k$ even and odd. 
		
		First, let $k := 2k'$ be an even integer for $k' \ge 2$.
		Consider the composition
		\begin{align}\label{eq_divisibility_SU_proof_1}
			\TMF[4k'] \xrightarrow{x(2k') \cdot} \TJF_{2k'-1}[1] \xrightarrow{} \TJF_{2k'-1} / \TJF_{2k'-3}[1] \simeq \TMF/\eta[4k'-3], 
		\end{align}
		where we have used Fact \ref{fact_TJF_cellstr} that $\TJF_m \simeq \TMF \otimes P_m$ with $P_m := \mathrm{cofib}(\Sigma \CP^{m-1} \xrightarrow{\tr} S^0)$, and the composition \eqref{eq_divisibility_SU_proof_1} is obtained by tensoring $\TMF$ with the following, 
		\begin{align}\label{eq_divisibility_SU_proof_2}
			S^{4k'} \xrightarrow{} \Sigma P_{2k' - 1} \xrightarrow{} \Sigma P_{2k' - 1}/P_{2k'-3} \simeq \Sigma^3 \CP^{2k'-2} / \CP^{2k'-4} \simeq S^{4k' - 1} \cup_\eta S^{4k'-3}. 
		\end{align}
		By the commutativity of \eqref{eq_attaching_P}, we see that the composition \eqref{eq_divisibility_SU_proof_2} is the stable attaching map of the top cell of the truncated complex projective space $ \CP^{2k'-1}/\CP^{2k'-4}$. It is known \cite[Theorem 5.2 and its proof]{Mosher} (also see the cell diagrams in Figure \ref{celldiag_TJF}) that this map factors through the bottom cell as,
		\begin{align}
			S^{4k'} \xrightarrow{k'\nu}S^{4k'-3} \xrightarrow{} S^{4k' - 1} \cup_\eta S^{4k'-3}, 
		\end{align}
		after tensoring $\TMF$, it is easy to see by using $\eta^2 = 12\nu$ in $\pi_3 \TMF$ that the composition \eqref{eq_divisibility_SU_proof_1} represents an element of order exactly $\frac{12}{\gcd(k', 12)}$ in $\pi_{3}\TMF/\eta$. 
		This means that the order of the element $x(k)$ is divisible by this number, proving the case of $k$ even. 
		
		Finally let us prove the case of odd $k = 2k' + 3$ with $k' \ge 0$. Similarly to the above proof for the even case, let us consider the composition
		\begin{align}\label{eq_divisibility_SU_proof_3}
			\TMF[4k'+6] \xrightarrow{x(2k'+3) \cdot} \TJF_{2k'+2}[1] \xrightarrow{} \TJF_{2k'+2} / \TJF_{2k'}[1] \simeq \TMF[4k'+5]\oplus \TMF[4k'+3], 
		\end{align}
		where, as before using $\TJF_m \simeq \TMF \otimes P_m$, we see that \eqref{eq_divisibility_SU_proof_3} is obtained by tensoring $\TMF$ to
		\begin{align}\label{eq_divisibility_SU_proof_4}
			S^{4k'+6} \xrightarrow{} \Sigma P_{2k' +2} \xrightarrow{} \Sigma P_{2k' +2}/P_{2k'} \simeq \Sigma^3 \CP^{2k'+1} / \CP^{2k'-1}\simeq S^{4k' +5} \oplus S^{4k'+3}. 
		\end{align}
		Again, \eqref{eq_divisibility_SU_proof_4} is the stable attaching map of the top cell of $\CP^{2k'+2}/\CP^{2k'-1}$. It is known \cite[Theorem 5.2 and its proof]{Mosher} (also see the cell diagrams in Figure \ref{celldiag_TJF}) that this map represents the class $(\eta, k'\nu) \in \pi_1 S \oplus \pi_3 S$ for $k' \not\equiv 0 \pmod 4$ and $(\eta, 2k'\nu) \in \pi_1 S \oplus \pi_3 S$ for $k' \equiv 0 \pmod 4$. 
		After tensoring $\TMF$, we see that the composition \eqref{eq_divisibility_SU_proof_3} represesnts the class of order exactly $\frac{24}{\gcd(12, k')}$. This means that the order of the element $x(2k'+3) = x(k)$ is divisible by this number, proving the case of $k$ odd. This completes the proof of Claim \ref{claim_dSU_easy}. 
	\end{proof}
\end{rem}

\subsection{The relation with (classical) Jacobi Forms}\label{subsec_JF_app}
As explained in Section \ref{subsubsec_TMF_C}, Jacobi Forms are identified with the $U(1)$-equivariant Modular Forms. 
Recall our notation $\JF_k$ in \eqref{eq_def_JFk} (Note that we are imposing integrality, and note also for our degree convention). 
We have multiplication maps $\JF_k \otimes_\MF \JF_m \to \JF_{k+m}$, which makes $\JF_\bullet := \oplus_{k} \JF_k$ into a $\Z$-graded $\MF$-module ring. 
Concretely, we have
By \cite[Theorem 2.7]{gritsenko2020modified} we have the generator-relation expression, 
\begin{align}\label{eq_JF_generators}
	\JF_\bullet =  \MF[a:= \phi_{-1, \frac12}, \phi_{0,1}, \phi_{0, \frac32},\phi_{0,2}, \phi_{0,4}, E_{4,1}, E_{4,2}, E_{4,3}, E_{6,1}, E_{6,2}, E'_{6,3}]/\sim, 
\end{align}
where for the relation we refer to \cite{gritsenko2020modified}.
The notation $f_{w, i}$ denotes an elements of weight $w$ and index $i$, so that $f_{w, i} \in \JF_{2i}|_{\deg = 2w+4k}$, and we have employed the notation $a := \phi_{-1, \frac12}$ as introduced in \eqref{eq_notation_a}.

The generator $a= \phi_{-1, \frac12} \in \JF_1|_{\deg=0}$ in \eqref{eq_notation_a} is of particular importance. 
It vanishes at order $1$ at the zero section of the universal elliptic curve, and nowhere vanishing outside. This means that the multiplication by $a$ gives an isomorphism of line bundles
\begin{align}\label{eq_isom_Loo_pole}
a \cdot  \colon \cO_{\cE_\C}(e) \simeq  \cA(\xi) \otimes \omega^{-1} \quad  \mbox{ in } \Pic(\cE_\C), 
\end{align}
where $\xi = [\overline{V}_{U(1)}] \in [BU(1), P^4 BO] \simeq \Z$ is the generator and $\cA(\xi)$ is the corresponding Looijenga's Line bundle (Definition \ref{def_Loo}). 
Thus, for each nonnegative integer $k$ we have an isomorphism
\begin{align}
a^k \cdot \colon	\Gamma(\cE_{\C}; \cO_{\cE_\C}(ke)) \simeq \JF_k^\C . 
\end{align}

Now we can introduce the connection with Topological Jacobi Forms. We have a canonical map $e_\JF \colon \pi_\bullet \TJF_k \to  \JF_k|_{\deg = \bullet}$ which fits into the commutative diagram,  
\begin{align}
	\xymatrix{
	\pi_\bullet \TJF_k \ar[rr]^-{e_\JF}  \ar@{=}[d] && \JF_k |_{\deg = \bullet} \ar@{_{(}->}[d]^-{a^{-k}}\\
	\pi_\bullet \Gamma(\cE, \cO_\cE(ke)) \ar[rr]^-{(\cE_\C \to \cE)^*} &&\Gamma(\cE_{\C}; \cO_{\cE_\C}(ke) \otimes \omega^{\bullet/2})
	}
\end{align}
This allows us to regard $\oplus_k \TJF_k$ as a spectral refinement of the graded ring of integral Jacobi Forms. 

The stabilization-restriction fiber sequence \eqref{seq_res_stab_TJF_app} fits into the following commutative diagram, 

\begin{align}\label{diag_stabres_JF_TJF_app}
	\xymatrix{
\pi_\bullet	\TJF_{k-1} \ar[r]^-{\stab}_-{\chi(V_{U(1)})\cdot} \ar[d]^-{e_\JF}&\pi_\bullet \TJF_{k} \ar[rrr]^-{\res_{U(1)}^e} \ar[d]^-{e_\JF}&&&\pi_{\bullet-2k} \TMF \ar[d]^-{e_\MF} \\
\JF_{k-1}|_{\deg = \bullet} \ar[r]^-{a \cdot} \ar@{_{(}->}[d]^-{a^{-k+1}}& \JF_{k}|_{\deg = \bullet} \ar[rrr]^-{\ev_{z=0} \colon \phi(z, q) \mapsto \phi(0, q)} \ar@{_{(}->}[d]^-{a^{-k}}&&&\MF|_{\deg = \bullet - 2k} \ar@{_{(}->}[d]\\
	\Gamma(\cE_\C, \cO_{\cE_\C}((k-1)e)\otimes \omega^{\bullet/2}) \ar@{^{(}->}[r] & \Gamma(\cE_\C, \cO_{\cE_\C}(ke)\otimes \omega^{\bullet/2}) \ar[rrr]^-{\ev_{z=0} \circ (a^k \cdot - )} &&& \Gamma(\moduli_\C, \omega^{\bullet/2-k}) , 
}
\end{align}
where the bottom left arrow is the canonical inclusion. 
In particular, we have 
\begin{align}\label{eq_a=chi_app}
	a = \phi_{-1, \frac12} = e_\JF\left( \chi(V_{U(1)})\right)  \in \JF_1|_{\deg = 0}. 
\end{align}
This is a special case of \eqref{eq_Phi_V=chiV}. 

\section{On $\TEJF$:= the $Sp(1) = SU(2)$-equivariant twisted $\TMF$}\label{app_TEJF}
In the main body of this article, the spectrum $\TEJF_{2k}$, which is defined to be $Sp(1)$-equivariant twisted $\TMF$ and called {\it Topological Even Jacobi Forms}, appeared as the domain of the $Sp(1)$-topological elliptic genus $\Jac_{Sp(1)_k} \colon MTSp(k) \to \TEJF_{2k}$. 
In this section, we study this spectrum, which itself is of independent interest. 

\begin{rem}
The content of this Appendix \ref{app_TEJF} is an original new result of this paper. 
\end{rem}

\subsection{The definition}

\begin{defn}[{$\TEJF_{2k}$}]\label{def_TEJF_app}
	Let $k$ be any integer. We define
	\begin{align}\label{eq_def_TEJF_app}
		\TEJF_{2k} := \TMF[kV_{Sp(1)}]^{Sp(1)} , 
	\end{align}
	where $V_{Sp(1)}$ is the fundamental representation of $Sp(1)$. 
\end{defn}

\begin{rem}\label{rem_no_odd_TEJF}
	Note that we do NOT define $\TEJF_{m}$ for odd $m$. 
\end{rem}
We employ this terminology because $\TEJF_{2k}$ is regarded as refining the module of integral {\it even} Jacobi Forms. 
Recall that we have defined in Example \ref{ex_EJF} the sub-$\MF$-module $\EJF_{2k} \subset \JF_{2k}$ by
\begin{align}
	\EvJF_{2k}|_{\deg = m} &:= \{\phi(z, \tau)\in \JF_{2k}|_{\deg = m} \ | \ \phi(z, \tau) = \phi(-z, \tau) \} = (\pi_m \JF_{2k})^{\Z/2}, \\
		& = \begin{cases}
				 \JF_{2k}|_{\deg = m} & m \equiv 0 \pmod 4  \\
			0 & m \not\equiv 0 \pmod 4. 
		\end{cases}
\end{align}
where $\Z/2$ acts on $\JF_{2k}$ by $\phi(z, \tau) \mapsto \phi(-z, \tau)$. 
$\EJF_{2k}$ is identified as integral $Sp(1)$-equivariant Modular Forms as explained in Example \ref{ex_EJF}. 
We get the $Sp(1)$-equivariant character map $e_\EJF$ in \eqref{eq_char_EJF}, 
which is compatible with the character map for $\TJF$, so that the following diagram commutes. 
\begin{align}\label{diag_TEJF_EJF_TJF_JF_app}
	\xymatrix{
\pi_m	\TEJF_{2k} \ar[d]^-{e_\EJF} \ar[rrrr]^-{\res_{Sp(1)}^{U(1)}}&&&& \pi_m \TJF_{2k} \ar[d]^-{e_\JF}\\ \EJF_{2k}|_{\deg = m}\ar@{^{(}->}[rrrr]^-{\id \  \mbox{\tiny   for } m \equiv 0 \pmod 4}_-{0 \  \mbox{\tiny   for } m \not\equiv 0 \pmod 4} &&&&  \JF_{2k}|_{\deg  = m}
	}
\end{align}

As indicated above, the map from $\EJF_{2k}$ to $\JF_{2k}$ is just the inclusion of the direct summand. 
However, as we will see in Section \ref{subsec_app_TJFvsTEJF} below, the upper horizontal arrow in \eqref{diag_TEJF_EJF_TJF_JF_app} does not split; rather, we show that it fits into a fiber sequence involving another copy of $\TEJF$ (Proposition \ref{prop_TEJF_TJF_TEJF6}). This creates nontrivial torsion elements in cokernels of the upper horizontal arrow in \eqref{diag_TEJF_EJF_TJF_JF_app}, which is exactly the origin of the refined divisibility result of Euler numbers for $Sp$-manifolds in the main text Section \ref{subsec_divisibility}. 

As explained in Section \ref{subsubsec_twist}, the $RO(G)$-graded $\TMF$ are special cases of twisted genuinely $G$-equivariant $\TMF$. In general, it is expected that genuinely $G$-equivariant $\TMF$ can be twisted by a map $BG \to P^4 BO$. In the case $G=Sp(1)$, we are lucky enough that the representations $kV_{Sp(1)} \in \RO(Sp(1))$ exhaust all the expected twists as follows.
\begin{lem}\label{lem_twist_Sp(1)}
	We have 
	\begin{align}\label{eq_BSp1_P4BO}
		[BSp(1), P^4BO] \simeq H^4(BSp(1); \Z) \simeq \Z,
	\end{align}
	and the element $\tw(\overline{V}_{Sp(1)}) \in [BSp(1), P^4 BO]$
	represents a generator of \eqref{eq_BSp1_P4BO}. 
\end{lem}
\begin{proof}
	The first claim follows from $Sp(1)$ being a compact connected simply connected simple Lie group, and the second claim follows from the fact that the second Chern class of $V_{Sp(1)}$ is the generator of $H^4(BSp(1); \Z)$. 
\end{proof}
This entitles us to say that $\TEJF_{2k}$ completes the list of all the geometrically twisted $Sp(1)$-equivariant $\TMF$.
\begin{rem}\label{rem_twist_convention}
	Concretely, $\TEJF_{2k}$ is identified with what is often called ``$Sp(1)$-equivariant $\TMF$ twisted by $k \tau \in H^4(BSp(1); \Z)$'' where $\tau$ is a generator of $H^4(BSp(1); \Z) \simeq \Z$. The integer $k$ is also often called the ``level''. But we need to be careful about the degree, since $\TEJF_{2k} =\left(  \TMF \otimes S^{kV_{Sp(1)}}\right)^{Sp(1)} $ and $S^{kV_{Sp(1)}}$ is of dimension $4k$; for example, to compare with the degree convention of \cite{tachikawa2023anderson}, we have
	\begin{align}
		\pi_m \TEJF_{2k}  = \TMF_{Sp(1)}^{4k-m+k\tau} \mbox{ in  \cite{tachikawa2023anderson}.} 
	\end{align}
\end{rem}

\subsection{Basic properties}

The structure of $\TEJF_\bullet$ is parallel $\TJF_\bullet$, reviewed in Section \ref{subsec_TJF_defprop_app}. 
The Euler class of the fundamental representation\footnote{
For a physical interpretation of the genuinely equivariant Euler class $\chi(V) \in \TMF[V]^G$, see Remark \ref{rem_PhiV=chiV}. In particular, the element $\chi(V_{Sp(1)})$ is supposed to be physically interpreted as ``quaternionic $1$-dimensional Majorana fermion''. 
}
\begin{align}
	\chi(V_{Sp(1)}) \in \pi_0 \TEJF_2
\end{align}
restricts by $\res_{Sp(1)}^{U(1)}$ to $\chi(V_{U(1)})^2 \in \pi_0 \TJF_2$ , so that we have
\begin{align}
e_\EJF\left( 	\chi(V_{Sp(1)}) \right) = a^2 = \left( \phi_{-1, \frac12} \right)^2 =  \left( \theta_{11}(z, \tau)/\eta(\tau)^3\right)^2 \in  \EJF_2|_{\deg = 0}. 
\end{align}
where we are using the notation \ref{eq_notation_a}. 
We have the multiplication map
\begin{align}
	\cdot \colon \TEJF_{2k} \otimes_\TMF \TEJF_{2m} \to \TEJF_{2k+2m}
\end{align}
so that $\oplus_{k} \TEJF_{2k}$ can be regarded as a evenly graded ring object in $\Mod_\TMF$. 

The {\it internal structure map} relating our trio in the main text introduced in Section \ref{subsubsec_internal} specializes to give the following {\it stabilization sequence} of $\TEJF$ (see Example \ref{ex_TEJF_building}),

\begin{align}\label{eq_building_seq_TEJF_app}
	\xymatrix@C=4em{
		\TEJF_0 \ar[r]^-{\stab}_-{ \chi(V_{Sp(1)})\cdot } \ar[d]^-{\res_{Sp(1)}^e}_-{\substack{\simeq \\ \mbox{\tiny Fact} \ \ref{factSU}}}& \TEJF_2 \ar[r]^-{\stab}_-{ \chi(V_{Sp(1)})\cdot }  \ar[d]^-{\res_{Sp(1)}^e}& \TEJF_4 \ar[r]^-{\stab}_-{ \chi(V_{Sp(1)})\cdot } \ar[d]^-{\res_{Sp(1)}^e}& \TEJF_6 \ar[r]^-{\stab}_-{ \chi(V_{Sp(1)})\cdot }  \ar[d]^-{\res_{Sp(1)}^e}& \cdots \\
		\TMF &	\TMF[4]   & \TMF[8] & \TMF[12]   & \cdots
	}, 
\end{align}
where each pair of consecutive horizontal and vertical arrows form a fiber sequence which we call the {\it stabilization-restriction fiber sequence} which fits into the following commutative diagram (c.f. \eqref{diag_stabres_JF_TJF_app}), 
\begin{align}\label{seq_res_stab_TEJF_app}
		\xymatrix@C=4em{
		\TEJF_{2k-2} \ar[r]^-{\stab}_-{\chi(V_{Sp(1)})\cdot} \ar@{~>}[d]^-{e_{\EJF}\circ \pi_*}& \TEJF_{2k} \ar[r]^-{\res_{Sp(1)}^e}\ar@{~>}[d]^-{e_{\EJF}\circ \pi_*}& \TMF[4k] \ar@{~>}[d]^-{e_{\MF}\circ \pi_*}\ar[r]^-{y(k)} & \TEJF_{2k-2}[1]\\
		\EJF_{2k-2} \ar[r]^-{a^2 \cdot } & \EJF_{2k}\ar[r]^-{ev_{z=0}} & \MF[4k]&, 
	}
\end{align}
The sequence \eqref{eq_building_seq_TEJF_app} is regarded as building $\TEJF_{2k}$ by attaching $4k$-dimensional $\TMF$-cells one by one. 
We have defined, in \eqref{eq_xyz}, the {\it attaching element}
\begin{align}\label{eq_y(k)_app}
	y(k) \in \pi_{4k-1} \TEJF_{2k-2}
\end{align}
to be the cofiber of $\res_{Sp(1)}^e$ in \eqref{seq_res_stab_TEJF_app}. This is the attaching map of the top $\TMF$-cell of $\TEJF_{2k}$, which can also be identified with the transfer map $\tr_e^{Sp(1)}$ (see \eqref{diag_selfdual_TEJF} below). The analysis of this element played a key role in our application to Euler numbers in Section \ref{subsec_divisibility}. 

We also use the following duality result: 

\begin{lem}\label{lem_negative_TEJF}
	\begin{enumerate}
		\item The virtual representation 
		$\theta := \Ad(Sp(1)) - 2V_{Sp(1)} \in \RO(Sp(1))$ admits a $BU\langle 6 \rangle$-structure $\fraks$, and the choice is unique up to contractible choice. 
		
		\item For any integer $k$, the composition
		\begin{align}\label{eq_duality_pairing_TEJF_app}
			\TEJF_{2k} \otimes_\TMF \TEJF_{-2k-4}[5]  \xrightarrow{ \cdot } \TEJF_{-4}[5] \stackrel{ \sigma(\theta, \fraks)}{\simeq}\TMF[-\Ad(Sp(1))]^{Sp(1)} \xrightarrow{\tr_{Sp(1)}^e} \TMF
		\end{align}
		exhibits the following duality in $\Mod_\TMF$ (Here $D$ denotes the dual in $\Mod_\TMF$), 
		\begin{align}\label{eq_TEJF_duality_app}
				\TEJF_{2k} \simeq D(\TEJF_{-2k-4})[-5]. 
		\end{align}
		Here, the equivalence $\sigma(\theta, \fraks)$ in \eqref{eq_duality_pairing_TEJF_app} is the $Sp(1)$-equivariant Thom isomorphism (Fact \ref{fact_sigma}) induced by the $BU\langle 6 \rangle$-structure in (1). 
	\end{enumerate}

\end{lem}

\begin{proof}
	(1) follows by checking the second Chern class. 
	(2) follows from the general duality statement of equivariant $\TMF$ in \eqref{eq_twisted_dual}. 
\end{proof}

At this point, we note that the stabilization-restriction sequence in \eqref{seq_res_stab_TEJF_app} is {\it dual} to that of $k$ replaced by $-k-1$, in the sense that the following diagram commutes by Proposition \ref{prop_selfdual_stabres}. 
\begin{align}\label{diag_selfdual_TEJF}
	\xymatrix@C=4em{
		\TEJF_{2k-2} \ar[r]^-{\stab}_-{\chi(V_{Sp(1)})\cdot} \ar[d]^{\simeq}& \TEJF_{2k} \ar[r]^-{\res}\ar[d]^{\simeq}& \TMF[4k] \ar@{=}[d] \ar[r]_-{y(k)}^-{\tr} & \TEJF_{2k-2}[1] \ar[d]^{\simeq}\\
		D(\TEJF_{-2k-2})[-5] \ar[r]^-{D(\stab)}_-{\chi(V_{Sp(1)})\cdot} & D(\TEJF_{-2k-4})[-5] \ar[r]_-{D(y(-k-2) )}^-{D(\tr)} & \TMF[4k] \ar[r]^-{D(\res)} & D(\TEJF_{-2k-2})[-4]\\
	}
\end{align}
Here we have used Lemma \ref{lem_negative_TEJF}. 
In particular, the commutativity of the right square allows us to identify the top right horizontal arrow with the transfer map as indicated in the diagram. 

\subsection{The cell structure of $\TEJF_{2k}$}\label{subsec_cell_TEJF}
In this subsection, we determine the structure of $\TEJF_{2k}$ as a $\TMF$-module. 
As we will see, $\TEJF_{2k}$ turns out to have a surprizingly simple structure;

\begin{prop}\label{prop_Sp(1)}
	\begin{enumerate}
		\item
	For any integer $k \ge -1$, we have an equivalence of $\TMF$-modules, 
	\begin{align}\label{eq_TEJF_cellstr}
	\TEJF_{2k}:=	\TMF[kV_{Sp(1)}]^{Sp(1)} \simeq \TMF \otimes \HP^{k+1}[-4]
	\end{align}
	Here we note that we are using $\HP^{k+1}$, NOT $\HP^{k+1}_+$. In particular, for $k=-1$ we have
	\begin{align}\label{eq_TEJF-2=0}
		\TEJF_{-2} = 0.  
	\end{align}
	
		Moreover, the isomorphism \eqref{eq_TEJF_cellstr} is compatible with the stabilization-restriction fiber sequence in \eqref{seq_res_stab_TEJF_app} in the sense that the following diagram commutes, 
	\begin{align}\label{eq_attaching_HP_app}
		\xymatrix@C=3em{
			\TEJF_{2k-2} \ar[r]^-{\stab}_-{\chi(V_{Sp(1)})\cdot} & \TEJF_{2k} \ar[r]^-{\res_{Sp(1)}^e} & \TMF[4k]  \ar[r]^-{y(k)} & \TEJF_{2k-2}[1]\\
			\HP^k[-4] \ar[u]^-{\TMF \otimes - } \ar@{^{(}->}[r] & \HP^{k+1}[-4] \ar@{->>}[r] \ar[u]^-{\TMF \otimes - } & S^{4k} \ar[u]^-{\TMF \otimes - } \ar[r]^-{\widetilde{y}(k)}& \HP^{k}[-3]\ar[u]^-{\TMF \otimes - }
		}
	\end{align}
	where the bottom row is the cofiber sequence induced by the standard inclusion $\HP^{k} \hookrightarrow \HP^{k+1}$, and we have denoted by $\widetilde{y}(k)$ the stable attaching map of the top cell of $\HP^{k+1}$.
	\item 
	For $k \le -2$, we have
	\begin{align}
		\TEJF_{2k} \stackrel{\eqref{eq_TEJF_duality_app}}{\simeq} D(\TEJF_{-2k-4})[-5] \stackrel{\eqref{eq_TEJF_cellstr}}{\simeq} \TMF \otimes D_S(\HP^{-k-1})[-1]. 
	\end{align}
	\end{enumerate}
\end{prop}

\begin{rem}[{Various equivalent cell structures for $\TEJF_{2k}$ (remark added for arXiv v2)}]\label{rem_various_cell}
    In the upcoming work by T.Bauer and the second author \cite{BauerYamashitaTEJF}, we show that actually there are other, a priori different-looking $\TMF$-cell complexes which turns out to be all equivalent to $\TEJF_{2k}$: 
    \ie
\TEJF_{2k} \simeq \TMF \otimes \HP^{k+1}[-4] \simeq \TMF \otimes D_S(\HP_{-k}^0), 
    \fe
    as well as
    \ie
\TEJF_{\infty} \simeq \TMF \otimes \HP^{\infty}[-4] \simeq \TMF \otimes D_S(\HP_{-\infty}^0) \simeq \TMF \sslash \nu
    \fe
    where the $\sslash$ denotes the $\mathbb{E}_1$-quotient, apearing in \cite{devalapurkar:hodge}. 
    In the paper \cite{BauerYamashitaTEJF}, we will give independent proof of the all the isomorphisms above, without using the proof provided here. In particular, \cite{BauerYamashitaTEJF} contains a proof for Proposition~\ref{prop_Sp(1)} above, which is conceptually better than the one provided here. 
\end{rem}

\begin{proof}
	The proof is parallel to the proof of Fact \ref{fact_TJF_cellstr} by Bauer-Meier sketched there. 
	Let $k \ge -1$. Consider the following cofiber sequence of pointed $Sp(1)$-spaces, 
	\begin{align}
		S((k+2)V_{Sp(1)})_+ \to S^0 \to S^{(k+2)V_{Sp(1)}}. 
	\end{align}
	Tensoring $S^{-\Ad_{Sp(1)}}$ to the above sequence and applying the $Sp(1)$-equivariant $\TMF$-homology functor $\left(\TMF \otimes (-)\right)^{Sp(1)} $, we get a fiber sequence
	\begin{align}\label{seq_Sp1}
		\left(\TMF \otimes S^{-\Ad_{Sp(1)}} \otimes S((k+2)V_{Sp(1)})_+ \right)^{Sp(1)} 
		\to \TMF[-\Ad_{Sp(1)}]^{Sp(1)} \to \TMF[(k+2)V_{Sp(1)}-\Ad_{Sp(1)}]^{Sp(1)}. 
	\end{align}
	By the Adams isomorphism and the fact that $S((k+2)V_{Sp(1)})_+ / Sp(1) \simeq \HP^{k+1}$, we get
	\begin{align}\label{eq_adams_Sp1}
		\left(\TMF \otimes S^{-\Ad_{Sp(1)}} \otimes S((k+2)V_{Sp(1)})_+ \right)^{Sp(1)} \simeq \TMF \otimes \HP_+^{k+1}
	\end{align}
	On the other hand, we claim that we have
	\begin{align}\label{eq_Sp1_duality}
		\TMF[-\Ad_{Sp(1)}]^{Sp(1)} \simeq D(\TMF^{Sp(1)}) \simeq D(\TMF) \simeq \TMF. 
	\end{align}
	Here $D(-)$ denotes the dual object in $\Mod_\TMF$. The first equivalence follows from Fact \ref{fact_dualizability_TMF} and the second equivalence follows from Fact \ref{factSU} (1). 
	Moreover, we have
	\begin{align}\label{eq_Sp1ad}
		\TMF[(k+2)V_{Sp(1)}-\Ad_{Sp(1)}]^{Sp(1)} \simeq \TMF[kV_{Sp(1)}+5]^{Sp(1)}, 
	\end{align}
	since $[\overline{\Ad_{Sp(1)}}] = 2 \cdot [\overline{V_{Sp(1)}}] \in [BSp(1), BO\langle 0, \cdots, 4 \rangle]$ and $\dim_\R \Ad_{Sp(1)} = 3$. 
	Rewriting the fiber sequence \eqref{seq_Sp1} by the isomorphisms \eqref{eq_adams_Sp1}, \eqref{eq_Sp1_duality} and \eqref{eq_Sp1ad}, we get a fiber sequence
	\begin{align}
		\TMF \otimes \HP_+^{k+1} \xrightarrow{\rm{ev}_+} \TMF \to \TMF[kV_{Sp(1)}+5]^{Sp(1)}
	\end{align}
	Here the first arrow can be identified by the evaluation at the basepoint because of the following observation\footnote{The arXiv v1 of this article contained a mistake in the logic here, but the result is unchanged.}:  
    In general, for any ring spectrum $A$ and $M \in \Mod_A$, suppose an $A$-module morphism
    \ie
f \colon A \oplus M \to A
    \fe
    satisfies that
    \ie \label{eq_section}
\id_A = \left(A \xhookrightarrow{\id_A \oplus 0} A \oplus M \xrightarrow{f} A \right). 
    \fe
    Then $f$
    is equivalent, by an $A$-module automorphism of $A \oplus M$, to the morphism $\id_A \oplus 0$. 
    We apply this observation to $A = \TMF$ and $M = \TMF \oplus \HP^{k+1}$ with the morphism $f$ being the first map of \eqref{seq_Sp1}. This map satisfies the above condition \eqref{eq_section} by the case $k=-1$. 
    % it factors through the case for $k=-1$. 
    This implies the first statement of Proposition \ref{prop_Sp(1)} (1). The second statement of (1) follows directly from our construction of the isomorphism \eqref{eq_TEJF_cellstr}. (2) is obtained by combining the duality statement in Lemma \ref{lem_negative_TEJF} and (1) of this proposition which we have just proved. 
	This completes the proof of Proposition \ref{prop_Sp(1)} 
\end{proof}

The stable attaching map $\widetilde{y}(k)$ in \eqref{eq_attaching_HP_app} of $\HP^{k+1}$ is classically known (e.g., \cite{Mukai}), and not difficult to prove, to satisfy the following. 
\begin{fact}\label{fact_HP_attaching}
	For each positive integer $k$, the composition
	\begin{align}
		S^{4k+3} \xrightarrow{\widetilde{y}(k)} \HP^{k} \to \HP^{k}/\HP^{k-1} \simeq S^{4k}
	\end{align}
	stably represents the element $k\nu \in \pi_3 S = \Z \nu / (24\nu)$. 
\end{fact}

This gives us the following result, which is the key to our application for the divisibility of Euler numbers in Section \ref{subsec_divisibility}. 
\begin{prop}\label{prop_y(k)_app}
	The attaching element $y(k) \in \pi_{4k-1} \TEJF_{2k-2}$ in \eqref{eq_y(k)_app} satisfies
	\begin{align}
	\res_{Sp(1)}^e (y(k)) = k\nu \in \pi_3 \TMF \simeq \Z\nu / (24\nu). 
	\end{align}
	In particular, the order $d_{Sp}(k)$ (Definition \ref{def_divisibility}) of the element $y(k)$ satisfies
	\begin{align}\label{eq_dSp_app}
	\left.	\frac{24}{ \gcd(k, 24)} \right| d_{Sp}(k). 
\end{align}
\end{prop}

\begin{proof}
	This follows directly from Fact \ref{fact_HP_attaching} and the commutativity of diagram \eqref{eq_attaching_HP_app}. 
\end{proof}

Actually, T.Bauer has proved that our estimate in Proposition \ref{prop_y(k)_app} is actually exact. The result with proof will appear in \cite{BauerComputationTJF}. We do not use this stronger statement in the main body of this paper, but for completeness we present the result here: 

\begin{prop}[{\cite{BauerComputationTJF}}]\label{prop_y(k)_exact}
	For each positive integer $k$, we have
	\begin{align}
		d_{Sp}(k) = \frac{24}{ \gcd(k, 24)}
	\end{align}
\end{prop}

So the cell diagram looks as shown in Figure \ref{celldiag_TEJF} for lower $k$. For example, we have
\begin{align}
	\TEJF_{-6} &\simeq \TMF\otimes (S^{-9} \cup_\nu S^{-5}) = \TMF/\nu[-9]\\
	\TEJF_{-4} &\simeq \TMF\otimes S^{-5} = \TMF[-5], \\
	\TEJF_{-2} &= 0, \\
	\TEJF_{0} & \simeq \TMF, \\
	\TEJF_{2} &\simeq \TMF\otimes (S^0 \cup_\nu S^4) = \TMF/\nu, \label{eq_TEJF2_app}\\
	\TEJF_{4} &\simeq \TMF \otimes \left( S^0 \cup_\nu S^4 \cup_{2\nu} S^8\right) . \label{eq_TEJF4}
\end{align}

By Proposition \ref{prop_d_im} and \eqref{eq_dSp_app}, we get
\begin{align}
\mathrm{im}	\left( \res_{Sp(1)}^e \colon \pi_{4k}\TEJF_{2k} \to  \pi_0 \TMF \right) \bigcap \mathrm{im}\left(  u \colon \Z \hookrightarrow \pi_0 \TMF \right) \subset 	\frac{24}{ \gcd(k, 24)} \Z. 
\end{align}
This is used in deducing the divisibility constraints of Euler numbers of tangential $Sp$-manifolds (Theorem \ref{thm_divisibility_constraints_concrete}). 

\begin{figure}[H]
	\centering
	\begin{tikzpicture}[scale=0.5]
			\begin{scope}[shift={(-16, 0)}]
			\begin{celldiagram}
				\nu{-9} \nu{-13}
				\n{-5} \n{-9} \n{-13}
				\foreach \y in {-5, -9, -13} {
					\node [left] at (-0.5, \y) {$\y$};
				}
			\end{celldiagram}
			\node [] at (0, -2) {$\TEJF_{-8}$};
			\node [blue] at (0, -7) {$\nu$};
			\node [blue] at (0, -11) {$2\nu$};
		\end{scope}
		
			\begin{scope}[shift={(-12, 0)}]
			\begin{celldiagram}
				\nu{-9}
				\n{-5} \n{-9}
				\foreach \y in {-5, -9} {
					\node [left] at (-0.5, \y) {$\y$};
				}
			\end{celldiagram}
			\node [] at (0, -2) {$\TEJF_{-6}$};
				\node [blue] at (0, -7) {$\nu$};
		\end{scope}
		
	\begin{scope}[shift={(-8, 0)}]
	\begin{celldiagram}
		\n{-5} 
		\foreach \y in {-5} {
			\node [left] at (-0.5, \y) {$\y$};
		}
	\end{celldiagram}
	\node [] at (0, -2) {$\TEJF_{-4}$};
\end{scope}
			\begin{scope}[shift={(-4, 0)}]
		\node [] at (0, 0) {$\varnothing $};
			\node [] at (0, -2) {$\TEJF_{-2} $};
		\end{scope} 
		
		\begin{scope}[shift={(0, 0)}]
			\begin{celldiagram}
				\n{0} 
				\foreach \y in {0} {
					\node [left] at (-0.5, \y) {$\y$};
				}
			\end{celldiagram}
			\node [] at (0, -2) {$\TEJF_{0}$};
		\end{scope}
		
		\begin{scope}[shift={(4, 0)}]
			\begin{celldiagram}
				\nu{0}
				\n{0} \n{4} 
				\foreach \y in {0, 4} {
					\node [left] at (-0.5, \y) {$\y$};
				}
			\end{celldiagram}
			\node [] at (0, -2) {$\TEJF_{2}$};
		\end{scope}
		\node [blue] at (4, 2) {$\nu$};
		
		\begin{scope}[shift={(8, 0)}]
			\begin{celldiagram}
				\nu{0} \nu{4}
				\n{0} \n{4} \n{8} 
				\foreach \y in {0, 4, 8} {
					\node [left] at (-0.5, \y) {$\y$};
				}
			\end{celldiagram}
			\node [] at (0, -2) {$\TEJF_{4}$};
				\node [blue] at (0, 2) {$\nu$};
			\node [ blue] at (0, 6) {$2\nu$};
			
		\end{scope}
		
			\begin{scope}[shift={(12, 0)}]
			\begin{celldiagram}
				\nu{0} \nu{4} \nu{8} 
				\n{0} \n{4} \n{8} \n{12} 
				\foreach \y in {0, 4, 8, 12} {
					\node [left] at (-0.5, \y) {$\y$};
				}
			\end{celldiagram}
			\node [] at (0, -2) {$\TEJF_{6}$};
			\node [blue] at (0, 2) {$\nu$};
		\node [ blue] at (0, 6) {$2\nu$};
			\node [blue] at (0, 10) {$3\nu$};
		\end{scope}
		
			\begin{scope}[shift={(16, 0)}]
			\begin{celldiagram}
				\nu{0} \nu{4} \nu{8} \nu{12}
				\n{0} \n{4} \n{8} \n{12} \n{16}
				\foreach \y in {0, 4, 8, 12, 16} {
					\node [left] at (-0.5, \y) {$\y$};
				}
			\end{celldiagram}
			\node [] at (0, -2) {$\TEJF_{8}$};
			\node [blue] at (0, 2) {$\nu$};
		\node [ blue] at (0, 6) {$2\nu$};
		\node [blue] at (0, 10) {$3\nu$};
			\node [blue] at (0, 14) {$4\nu$};
		\end{scope}

	\end{tikzpicture}
	\caption{The cell diagram of $\TEJF_{2k}$}\label{celldiag_TEJF}
\end{figure}

\subsubsection{$\TEJF$ at odd primes}\label{subsubsec_TEJF_p=odd}
If we invert the prime $2$, $\TEJF_k$'s look even more simple. First, if we localize at a prime $p \ge 5$, Proposition \ref{prop_Sp(1)} simply gives a decomposition
\begin{align}
	\left( \TEJF_{2k}\right)_{(p)} \simeq \bigoplus_{i = 0}^{k}\TMF_{(p)}[4i]. \quad p \ge 5. 
\end{align}

Now consider the case including $p = 3$. We get, for each odd prime $p$, 
\begin{align}\label{eq_TEJF4=TMF12}
	\left( \TEJF_{4}\right) _{(p)} \stackrel{\eqref{eq_TEJF4}}{\simeq} \TMF \otimes (S^0 \cup_\nu S^4 \cup_{2\nu} S^{8})_{(p)} \simeq \TMF_1(2), 
\end{align}
where $\TMF_1(2)$ is the $\TMF$ with level-$2$ structure \cite{HillLawson}. 
We know that $\pi_* \TMF_1(2)$ is non-torsion, concentrated in $* \equiv 0$ (mod $4$). 
In particular, the connecting element $y(3) \in \pi_{11}\left( \TEJF_4\right) _{(p)}$ in the stabilization sequence \eqref{seq_res_stab_TEJF_app} for $k=3$ is zero, so the sequence splits at $\TEJF_{6}$, 
\begin{align}\label{eq_TEJF_6_split}
	\xymatrix@C=5em{
		\left( \TEJF_{4}\right) _{(p)} \ar[r]^-{\stab}_-{\chi(V_{Sp(1)})\cdot} & \left( \TEJF_{6} \right) _{(p)}\ar[r]^-{\res_{Sp(1)}^e}& \TMF_{(p)}[12] \ar@/^1ex/[l]^-{\mathfrak{c}}
	}
\end{align}
here we denoted an element $\mathfrak{c} \in \pi_{12} \left( \TEJF_6\right) _{(p)}$ which gives a splitting. Note that the character $e_\EJF(\frc) \in \EJF_6|_{\deg = 12}$ of this element should satisfy
\begin{align}
\ev_{z=0} \circ	e_\EJF(\frc) = 1. 
\end{align} 
By inspecting the generators of $\EJF_6|_{\deg = 12}$ and using \eqref{eq_ev_z=0}, we find that we neccesarily have
\begin{align}
	e_\EJF (\frc) =\left(  \frac{\phi_{0, \frac32}}{2}\right)^2.  
\end{align}

\begin{prop}[{$\TEJF$ localized at prime $3$}]\label{prop_TEJF_p=3}
	For each integer $k \ge 3$, we have the follwoing decomposition of $\left( \TEJF_{2k}\right) _{(3)}$ as a $\TMF_{(3)}$-module: Setting $k' := \lfloor (k+1)/3 \rfloor$, the map
	\begin{align}\label{eq_TEJFp=3}
( \frc^{k'}\cdot) \oplus  \ \bigoplus_{i=0}^{k'-1} \left( ( \stab)^{k-3i-2} \circ \frc^i \cdot \right)   \colon \left( \TEJF_{2(k-3k')}\right) _{(3)}[12k'] \oplus \bigoplus_{i=0}^{k'-1} \left( \TEJF_4\right) _{(3)}[12i] \to \left( \TEJF_{2k}\right) _{(3)}. 
	\end{align}
	is an equivalence of $\TMF$-modules. 
	Here, the map consists of the multiplication by the element $\frc \in \pi_{12} \TEJF_6$ given in \eqref{eq_TEJF_6_split}. 
	This means that, using \eqref{eq_TEJF4=TMF12}, we have an isomorphism of $\TMF_{(3)}$-modules, 
	\begin{align}
	\left( 	\TEJF_{2k}\right) _{(3)} \simeq \bigoplus_{i=0}^{k'-1} \TMF_1(2)[12i] \bigoplus \begin{cases}
		\TMF_{(3)}[12k'] & k \equiv 0 \pmod 3, \\
		\left( \TMF/\nu\right) _{(3)} [12k'] & k \equiv 1 \pmod 3, \\
		0 & k \equiv 2 \pmod 3. 
	\end{cases}
	\end{align}
	In particular, the torsions in the homotopy groups are given by
	\begin{align}
\left( 	\pi_\bullet 	\left( 	\TEJF_{2k}\right) _{(3)} \right) _{\rm tors} \simeq \begin{cases}
\left( \pi_{\bullet - 12k'}	\TMF_{(3)} \right)_{\rm tors}  & k \equiv 0 \pmod 3, \\
\left( \pi_{\bullet - 12k'}	\left( \TMF/\nu\right) _{(3)} \right) _{\rm tors} & k \equiv 1 \pmod 3, \\
	0 & k \equiv 2 \pmod 3. 
\end{cases}
	\end{align}
\end{prop}

\begin{proof}
	The proof is analogous to the proof of Proposition \ref{prop_structure_TJF_2inverted} in \cite[Appendix A]{LinTominagaYamashita}, so we only give a sketch here and leave the details to the reader. We claim that the map
	\begin{align}\label{eq_prf_TEJF_p=3}
	\frc \cdot \oplus \left( \stab \right)^{k-2}  \colon \left( \TEJF_{2k-6}\right) _{(3)}[12] \oplus \left( \TEJF_4 \right) _{(3)}\to \left( \TEJF_{2k}\right)_{(3)} 
	\end{align}
	is an equivalence for any $k \ge 2$. This claim is shown by the induction on $k$, using the fact that \eqref{eq_prf_TEJF_p=3} is compatible with the stabilization-restriction fiber sequence \eqref{seq_res_stab_TEJF_app}. The first statement of the Proposition follows by applying this claim repeatedly. The remaining statements follow from the fact that $\pi_* \TMF_1(2)$ is torsion-free. 
\end{proof}

\subsection{Comparison to $\TJF$}\label{subsec_app_TJFvsTEJF}

In this subsection, as promised in the paragraph after the diagram \eqref{diag_TEJF_EJF_TJF_JF_app}, we study the restriction map
\begin{align}
	\res_{Sp(1)}^{U(1)} \colon \TEJF_{2k} \to \TJF_{2k}. 
\end{align}
The statement uses the Euler class of the adjoint representation of $Sp(1)$, 
\begin{align}\label{eq_chiAd}
	\chi\left(\Ad(Sp(1))\right) \in \pi_0 \TMF[\Ad(Sp(1))]^{Sp(1)} \simeq  \pi_{5} \TEJF_{4}, 
\end{align}
where we have used the string orientation of $\Ad(Sp(1)) - 2V_{Sp(1)}$ and the $Sp(1)$-equivariant sigma orientation.

\begin{prop}\label{prop_TEJF_TJF_TEJF6}
	For each integer $n \in \Z$, we have the following fiber sequence of $\TMF$-modules. 
	\begin{align}\label{seq_TEJF_TJF_TEJF6}
	\TEJF_{2n-4}[5] \xrightarrow{\chi(\Ad(Sp(1))) \cdot }	\TEJF_{2n} \xrightarrow{\res_{Sp(1)}^{U(1)}} \TJF_{2n}  \to \TEJF_{2n-4}[6]
	\end{align}
\end{prop}

\begin{proof}
	We follow a similar strategy as the proof of Proposition \ref{prop_stab_res_TMF}. 
	First observe that we have an isomorphism of $Sp(1)$-spaces,  
	\begin{align}
		Sp(1)/U(1) \simeq S(\Ad(Sp(1))). 
	\end{align}
	Thus we have the following cofiber sequence of pointed $Sp(1)$-spaces, 
	\begin{align}
		(Sp(1)/U(1))_+ \to S^0 \xrightarrow{\chi(\Ad(Sp(1)))} S^{\Ad(Sp(1))}. 
	\end{align}
	Taking the smash product with $S^{-nV_{Sp(1)}}$, we get the following fiber sequence of $Sp(1)$-spectra, 
	\begin{align}\label{eq_Sp/U_proof}
		(Sp(1)/U(1))_+\wedge S^{-nV_{Sp(1)}} \to S^{-nV_{Sp(1)}} \xrightarrow{\chi(\Ad(Sp(1)))} S^{-nV_{Sp(1)}+ \Ad(Sp(1))}. 
	\end{align}
	By \eqref{eq_IndRes} we have an isomorphism of $Sp(1)$-spectra, 
	\begin{align}
	 \Ind_{U(1)}^{Sp(1)}\left(S^{\res_{Sp(1)}^{U(1)} (nV_{Sp(1)} )}\right) \simeq (Sp(1)/U(1))_+\wedge S^{-nV_{Sp(1)}}, 
	\end{align}
	Thus, applying $Sp(1)$-equivariant $\TMF$-cohomology to \eqref{eq_Sp/U_proof}, we get a fiber sequence
	\begin{align}
		\xymatrix{
			\TMF\left[nV_{Sp(1)} -\Ad(Sp(1))\right]^{Sp(1)}  \ar[rr]^-{\chi(\Ad(Sp(1)))\cdot }  \ar[d]^{\simeq} && \TMF\left[nV_{Sp(1)}\right]^{Sp(1)} \ar@{=}[d] \ar[rr]^-{\res_{Sp(1)}^{U(1)}} && \TMF\left[\res_{Sp(1)}^{U(1)} (nV_{Sp(1)} )\right]^{U(1)} \ar[d]^-{\simeq} \\
			\TEJF_{2n-4}[5] \ar[rr]^-{\chi(\Ad(Sp(1)))\cdot } && \TEJF_{2n} \ar[rr]^-{\res_{Sp(1)}^{U(1)}}& &\TJF_{2n}
		}
	\end{align}
\end{proof}

\begin{cor}\label{cor_TEJF2=TJF2}
	The restriction map
	\begin{align}
		\res_{Sp(1)}^{U(1)} \colon \TEJF_2 \to \TJF_2 
	\end{align}
	gives an isomorphism between $\TEJF_2 \simeq \TMF/\nu$ in \eqref{eq_TEJF2_app} and $\TJF_2 \simeq \TMF/\nu$ in \eqref{eq_TJF_2}. 
\end{cor}
\begin{proof}
	This follows from Proposition \ref{prop_TEJF_TJF_TEJF6} applied to $k = 1$ and the fact that $\TEJF_{-2} = 0$ in Proposition \ref{prop_Sp(1)}. 
\end{proof}

\begin{cor}\label{cor_TEJF_TJF_invert2}
	If we invert the prime $2$, the fiber sequence \eqref{seq_TEJF_TJF_TEJF6} splits at $\TJF_{2k}$, 
so that we have an isomorphism of $\TMF$-modules, 
\begin{align}\label{eq_TEJF_TJF_dec_odd}
\TEJF_{2k} \oplus \TEJF_{2k-4}[6] \simeq \TJF_{2k} . 
\end{align}
\end{cor}
\begin{proof}
This is because, after inverting $2$, we have $\pi_5 \TEJF_4 =0$ by Section \ref{subsubsec_TEJF_p=odd}. Thus the element \eqref{eq_chiAd} vanishes and get the desired splitting. 
\end{proof}

\begin{rem}
	Propositions \ref{prop_TEJF_p=3} and Corollary \ref{cor_TEJF_TJF_invert2} explain the decomposition of $\TJF$ at odd prime in Proposition \ref{prop_structure_TJF_2inverted} in a nice way. Namely, the $\TMF_1(2)$'s appearing in the decomposition \eqref{eq_deco_TJFm} is most naturally regarded as $\TEJF_2$. The components labeled by even $i$ correspond to the first direct summand $\TEJF_{2k}$ in \eqref{eq_TEJF_TJF_dec_odd}, and those labeled by odd $i$ correspond to the second direct summand $\TEJF_{2k-4}[6]$.
\end{rem}

\section{A toy model: The topological $\mathbb{G}_m$-genera}\label{sec_app_toymodel}

In this section, we give a toy model of the construction of the main body of this article\footnote{The authors thank Thomas Schick for suggesting this toy model.}. 
We replace the genuinely equivariant $\TMF$ with the genuinely equivariant $\KO$-theory with the standard equivariance. The construction here should be regarded as being obtained by replacing elliptic curves by the multiplicative group $\mathbb{G}_m$ in the construction, so we name them as {\it topological $\mathbb{G}_m$-genera}. 
We construct a morphism of spectra of the form
\begin{align}
	\Jac^{\KO}_{\cD^\KO} \colon MT(H, \tau_H) \to \KO[\tau_G]^{G}, 
\end{align}
where $\cD^\KO$ is a set of data introduced below, and $G, H, \tau_G, \tau_H$ are included as ingredients of the data $\cD^\KO$. 

\subsection{The definition of $\Jac^{\KO}$}\label{subsec_def_JacKO}
We start with the main construction of this section, which is compared to Section \ref{subsec_construction} in the main part. 
Assume we are given a set of data, which we symbolically denote by $\mathcal{D}^\KO$. 
\begin{itemize}
	\item Fix compact Lie groups $G$ and $H$, together with $\tau_G \in \RO(G)$ and $\tau_H \in \RO(H)$. 
	\item Fix an integer $d$ and a group homomorphism $\phi \colon G \times H \to O(d)$. We denote the corresponding $d$-dimensional orthogonal representation by $V_\phi \in \Rep_{O}(G \times H)$. 
	\item We assume that $\dim \tau_H = 0$ and $d = \dim \tau_G$\footnote{This assumption is technical. In general, we can just add trivial representations to $\tau_G$ or $\tau_H$ to reduce to this case. }. 
	\item Fix a {\it spin} structure $\mathfrak{s}$ on the virtual representation
	\begin{align}
		\Theta := V_\phi - \res_G^{G \times H}(\tau_G)- \res_H^{G \times H}(\tau_H) \in \RO(G \times H). 
	\end{align}
	I.e., we assume that the composition
	\begin{align}
		BG \times BH \xrightarrow{\Theta} BO \to P^2 BO
	\end{align}
	is nullhomotopic, and $\mathfrak{s}$ is a choice of its nullhomotopy.\footnote{The Postnikov truncation $P^2 BO$ of $BO$ is the obstruction space of spin structure. We have a fibration
		\begin{align}
			BSpin \to BO \to P^2BO. 
		\end{align}
	}
\end{itemize}
We can regard $\Theta$ as a vector bundle over $BH$ with a $G$-action, where the space $BH$ is equipped with the trivial $G$-action. Then the spin structure $\fraks$ above induces the $G$-equivariant spin structure on the virtual vector bundle $\Theta$ on $BH$. The $G$-equivariant Atiyah-Bott-Shapiro orientation gives us the following equivalence of $G$-equivariant $\KO$-module spectra, 
\begin{align}\label{eq_ABS_Theta}
\ABS(\Theta, \fraks)\colon	\KO \otimes BH^{V_\phi - \tau_H} \simeq \KO \otimes BH_+ \otimes S^{\tau_G}. 
	\end{align}

\begin{defn}[{Definition of $\Jac^{\KO}_{\cD^\KO}$}]\label{def_Jac_KO_general}
	Assume we are given a set of data $\cD^{\KO}$ listed above. 
	Consider the following map in $\Spectra^{G}$: 
	\begin{align}\label{eq_JacKO_def_1}
		MT(H, \tau_H) = BH^{-\tau_H} \xhookrightarrow{\chi(V_\phi) \cdot } BH^{V_\phi-\tau_H}. 
	\end{align}
	Here, $MT(H, \tau_H)$ is regarded as a spectrum with trivial $G$-equivariance, and $V_\phi$ is regarded as a $G$-equivariant vector bundle over $BH$. The map is given by the inclusion of the zero section of $V_\phi$. 
	After tensoring $\KO \in \Spectra^{G}$, we get, again in $\Spectra^{G}$, 
	\begin{align}\label{eq_JacKO_def_2}
		\eqref{eq_JacKO_def_1} \xrightarrow{u \otimes \id}	 \KO \otimes BH^{V_\phi-\tau_H} \stackrel{\ABS(\Theta, \fraks)}{\simeq}  \KO \otimes BH_+ \otimes S^{\tau_G}, 
	\end{align} 
	by \eqref{eq_ABS_Theta}.
	Take the genuine $G$-fixed point of the composition of \eqref{eq_JacKO_def_1} and \eqref{eq_JacKO_def_2}, and define $\Jac^\KO_{\cD^{\KO}}$ to be the following composition. 
	\begin{align}\label{eq_JacKO_def_3}
		\xymatrix@C=5em{
			MT(H,\tau_H)\ar[rd]_-{\Jac^\KO_{\cD^{\KO}}} \ar[r]^-{\eqref{eq_JacKO_def_2}  \circ \eqref{eq_JacKO_def_1} }  & \left( \KO\otimes BH \otimes S^{\tau_G}\right)^{G}
			\ar[d]^-{(BH \to \pt)_*}  \\
			&\TMF[\tau_G]^{G}. 
		}
	\end{align}
\end{defn}

\begin{rem}
	Actually, Definition \ref{def_Jac_KO_general} above is the analogy of the ``alternative definition'' of topological elliptic genera, given in Proposition \ref{prop_alternative_Jac} and Remark \ref{rem_alternative_general}. 
	Note that we cannot give the analog of Definition \ref{def_general_Jac} since that definition relies on the dualizability of genuinely equivariant $\TMF$ in Fact \ref{fact_dualizability_TMF}. As noted after Fact \ref{fact_dualizability_TMF}, we do not have such a dualizability in equivariant $\KO$-theory. 
\end{rem}

\subsection{Example: The $U$-and $O$-topological $\mathbb{G}_m$-genera}

Here we introduce a {\it twin} of examples---$(U, U)$, $(O, SO)$---where the general construction of Section \ref{subsec_def_JacKO} applies. The content of this subsection is compared to Section \ref{sec_ex} in the main body of the article, where we construct {\it trio} of examples of topological elliptic genera. 

\subsubsection{Definitions}\label{subsubsec_def_twin}

\begin{defn}[{The topological $\mathbb{G}_m$ genera $\JacKO_{U(n)_k}$ and $\Jac^{\KO}_{O(n)_k}$}]\label{def_Jac_KO_twin}
	For each $k, n \in \Z_{\ge 1}$, we define the morphisms
	\begin{align}
	\Jac^\KO_{U(n)_k} &\colon MT(U(k), n\overline{V}_{U(k)}) \to \KO[kV_{U(n)}]^{U(n)}, \label{eq_def_Jac_KO_U}, \\
	\Jac^\KO_{O(n)_k} &\colon MT(SO(k), n\overline{V}_{SO(k)}) \to \KO[kV_{O(n)}]^{O(n)}, \label{eq_def_Jac_KO_O}, 
	\end{align}
by applying the general construction in Definition \ref{def_Jac_KO_general} to the following data. Here, for each group $K$ appearing below, the notation $V_{K} \in \RO(K)$ denotes the fundamental (a.k.a. defining, or vector) representation. 
\begin{itemize}
	\item For \eqref{eq_def_Jac_KO_U}, the data $\cD^\KO_{U(n)_k}$ consists of 
	\begin{align}
		G:=U(n), \  H:= U(k), \ \tau_G := kV_{U(n)}, \ \tau_H := n\overline{V}_{U(k)}, \ V_\phi := V_{U(n)} \otimes_\C V_{U(k)} 
	\end{align} 
	so that $\Theta_{U(n)_k} =\overline{V}_{U(n)} \otimes_{\C} \overline{V}_{U(k)}\in\RO(U(n) \times U(k))$, with its spin structure $\fraks$ obtained by Proposition \ref{prop_key_KO_U} below. 
	\item For \eqref{eq_def_Jac_KO_O}, the data $\cD^\KO_{O(n)_k}$ consists of 
	\begin{align}
		G:=O(n), \  H:= SO(k), \ \tau_G := kV_{O(n)}, \ \tau_H := n\overline{V}_{SO(k)}, \ V_\phi := V_{O(n)} \otimes_\R V_{SO(k)} 
	\end{align} 
	so that $\Theta_{O(n)_k} =\overline{V}_{O(n)} \otimes_{\R} \overline{V}_{SO(k)}\in\RO(U(n) \times SO(k))$, with its spin structure $\fraks$ obtained by Proposition \ref{prop_key_KO_O} below. 
\end{itemize}

\end{defn}

A particularly important case is $n=1$, where we get
\begin{align}
		\Jac^\KO_{U(1)_k} &\colon MTU(k) \to \KO[kV_{U(n)}]^{U(n)}, \label{eq_def_Jac_KO_U(1)}, \\
	\Jac^\KO_{O(1)_k} &\colon MTSO(k) \to \KO[kV_{O(n)}]^{O(n)}, \label{eq_def_Jac_KO_O(1)}. 
\end{align}

Here the necessary spin structures are provided by the following. For the case of $\JacKO_{U(n)_k}$, we have
\begin{prop}\label{prop_key_KO_U}
 The virtual representation
		\begin{align}\label{eq_key_KO_U(n)}
			\Theta_{U(n), U(k)} = \overline{V}_{U(n)} \otimes_{\C} \overline{V}_{U(k)}\in\RO(U(n) \times U(k))
		\end{align}
		has an $SU$-structure $\mathfrak{s}_{U, U}$, thus in particular a spin structure. Moreover, it is unique up to homotopy. 
\end{prop}
\begin{proof}
	The existence of an $SU$-structure is verified by the vanishing of the first Chern class. The uniqueness follows from $H^1(U(n) \times U(k); \Z) =0$.
\end{proof}
		
In order to state the proposition regarding the case of $\JacKO_{O(n)_k}$, we need a little preparation. Consider the following group homomorphisms, 
\begin{align}
	\alpha_G &\colon O(n) \hookrightarrow U(n), \\
	\beta_H &\colon U\left(\lfloor k/2\rfloor \right) \hookrightarrow SO(2\lfloor k/2\rfloor) \hookrightarrow SO(k), 
\end{align}
where $\alpha_G$ is induced by $\R \hookrightarrow \C$, and $\beta_H$ is induced by forgetting the complex structure of $\C^{\lfloor k/2\rfloor}$ to regard it as the real vector space $\R^{2\lfloor k/2\rfloor}$, and the second arrow is nontrivial only for $k$ odd. 
Then we can easily verify that
\begin{lem}\label{lem_res_O_U_app}
	We have the following canonical isomorphism in $\RO\left(O(n) \times U\left(\lfloor k/2\rfloor \right)\right)$,  
	\begin{align}\label{eq_res_O_U_app}
		\res_{\id \times \beta_H} \left( \overline{V}_{O(n)}  \otimes_\R \overline{V}_{SO(k)} \right)
		\simeq \res_{\alpha_G \times \id}\left( \overline{V}_{U(n)}  \otimes_\C \overline{V}_{U\left(\lfloor k/2\rfloor \right)} \right). 
	\end{align}
\end{lem}
Now we can state the proposition. 
\begin{prop}\label{prop_key_KO_O}
	The virtual representation
		\begin{align}
			\Theta_{O(n), SO(k)} = \overline{V}_{O(n)}  \otimes_\R \overline{V}_{SO(k)} \in \RO(O(n) \times SO(k)) 
		\end{align}
		admits a spin structure, and there is, up to homotopy, a unique choice $\fraks_{O, SO}$ which admits the following equivalence of spin structures when restricted to $O(n) \times SO(k)$,
		\begin{align}
			\res_{\id \times \beta_H}(\fraks_{O, SO}) \simeq\res_{\alpha_G \times \id}(\fraks_{U, U}). 
		\end{align}
		Here we are using Lemma \ref{lem_res_O_U_app}, and the string structure $\fraks_{U, U}$ on $\Theta_{U(n), U(\lfloor k/2 \rfloor)}$ is the one in Proposition \ref{prop_key_KO_U}. 
\end{prop}
\begin{proof}
	The existence of a spin structure is checked by the vanishing of the first and second Stiefel-Whitney classes. 
		The second claim follows by the fact that the map
	\begin{align}
		BO(n) \times BU(k') \xrightarrow{\id \times \beta_H} BO(n) \times BSO(2k')
	\end{align}
	for any $k' \ge 1$ is $3$-connected, so that giving a spin structure on $\Theta_{O(n), SO(k)}$ is equivalent to giving a spin structure on $ \res_{\id \times \beta_H}(\Theta_{O(n), SO(k)})$. 
\end{proof}

\subsubsection{Structures in the twins}
The families of examples constructed in Section \ref{subsubsec_def_twin} get unified via the {\it structure maps} relating each other. 
They consist of {\it external} and {\it internal} structure maps. 

\paragraph{The external structure: relating $(U, U)$ and $(O, SO)$}

The external structure relates $U$-and $O$-topological $\Gm$-genera. In this case, we simply have the following statement; 

\begin{prop}[{Compatibility of $\JacKO_{U(n)_k}$ and $\JacKO_{O(n)_k}$}]\label{prop_Jac_compatibility_external_KO}
	The $U$ and $O$-topological $\Gm$-genera are compatible in the sense that the following diagram commutes. 
	\begin{align}\label{diag_U_O_JacKO}
				\xymatrix{
					MT(U(k), n\overline{V}_{U(k)}) \ar[d]^-{(U(k) \hookrightarrow SO(2k))_*} \ar[rr]^-{\JacKO_{U(n)_k}} & &\TMF[kV_{U(n)}]^{U(n)} \ar[d]^-{\res^{O(n)}_{U(n)}}\\
					MT(SO(2k), n\overline{V}_{SO(2k)} )\ar[rr]^-{\JacKO_{O(n)_{2k}}} & &\TMF[2kV_{O(n)}]^{O(n)}. 
				}
	\end{align}

\end{prop}
The proof is analogous to the corresponding Proposition \ref{prop_Jac_compatibility_external}. 
Note that the choice of spin structure in Proposition \ref{prop_key_KO_O} is made precisely to make this compatibility result hold.

\paragraph{The internal structure: relating different $(n, k)$}
Now we introduce the {\it internal} structures in the twin, which relates different pairs of parameters $(n, k)$. Fixing $(\cG, \cH)$ to be any one of $(U, U)$, $(O, SO)$, and introduce the structure maps for the domains and codomains of $\JacKO$, respectively. 
Actually, for the domain, the structure maps for $MT(\cG(k), n\overline{V}_{\cG(k)}) $'s are exactly the same as the one we introduced in Section \ref{subsubsec_internal}, forming a {\it stabilization-restriction fiber sequence} of tangential bordism spectra (Proposition \ref{prop_stab_res_MT}). 
So here we focus on the structure maps for the codomain, the twisted equivariant $\KO$-theories. 

As we have remarked in Remark \ref{rem_internal_general}, we can apply the definition \eqref{eq_TMF_stab} and \eqref{eq_TMF_res} replacing $\TMF$ to $\KO$. 
Let us set $\cG$ to be one of $U$ or $O$, and $N=2, 1$, respectively. We get the maps
\begin{align}
\stab :=	\chi(V_{\cG(n)}) \cdot & \colon \KO[(k-1)V_{\cG(n)}]^{\cG(n)} \to \KO[kV_{\cG(n)}]^{\cG(n)}, \label{eq_KO_stab}\\
	\res_{\cG(n)}^{\cG(n-1)} &\colon \KO[kV_{\cG(n)}]^{\cG(n)} \to \KO[kV_{\cG(n-1)} + Nk]^{\cG(n-1)} \label{eq_KO_res}
\end{align}
We call the maps \eqref{eq_KO_stab} and \eqref{eq_KO_res} the {\it stabilization} and {\it restriction} maps, respectively. 
We get the following. 

\begin{prop}[{The stabilization-restriction fiber sequence of equivariant $\KO$ }]\label{prop_stab_res_KO}
	The maps \eqref{eq_TMF_stab} and \eqref{eq_TMF_res} form a fiber sequence of $\KO$-module spectra, 
	\begin{align}
		\KO[(k-1)V_{U(n)}]^{U(n)} \xrightarrow[\stab]{ \chi(V_{U(n)}) \cdot }\KO[kV_{U(n)}]^{U(n)} \xrightarrow[\res]{ \res_{U(n)}^{U(n-1)} }\KO[kV_{U(n-1)} + 2k]^{U(n-1)} ,\\
			\KO[(k-1)V_{O(n)}]^{O(n)} \xrightarrow[\stab]{ \chi(V_{O(n)}) \cdot }\KO[kV_{O(n)}]^{O(n)} \xrightarrow[\res]{ \res_{O(n)}^{O(n-1)} }\KO[kV_{O(n-1)} + k]^{O(n-1)}
	\end{align}
\end{prop}

Thus we get the diagram consisting of the structure maps, 
\begin{align}\label{seq_building_KO_general}
	\scalebox{0.85}{
		\xymatrix@C=0.5em{
			\ar[r]^-{\stab} &	\KO[(k-1)V_{\cG(n)}]^{\cG(n)}  \ar[r]^-{\stab} \ar[d]^-{\res}   & \KO[kV_{\cG(n)}]^{\cG(n)}  \ar[r]^-{\stab} \ar[d]^-{\res} & \KO[(k+1)V_{\cG(n)}]^{\cG(n)}  \ar[d]^-{\res} \ar[r]^-{\stab} &  \\
			&	\KO[(k-1)(V_{\cG(n-1)} + N)]^{\cG(n-1)} & \KO[k(V_{\cG(n-1)} + N)]^{\cG(n-1)}  & \KO[(k+1)(V_{\cG(n-1)} + N)]^{\cG(n-1)} &
		}
	}
\end{align}
where each pair of consecutive horizontal and vertical arrows form a fiber sequence. 
Particularly important case is the case of $n=1$. Here let us focus on the case $\cG=O$. The stabilization-restriction fiber sequence becomes
\begin{align}\label{eq_w}
	\KO[(k-1)V_{O(1)}]^{O(1)} \xrightarrow{\stab} \KO[kV_{O(1)}]^{O(1)} \xrightarrow{\res} \KO[k] \xrightarrow{w(k) \cdot} \KO[(k-1)V_{O(1)}+1]^{O(1)}, 
\end{align}
where we defined $w(k) \in \pi_{k-1} \KO[(k-1)V_{O(1)}]^{O(1)}$ by the above fiber sequence of $\KO$-modules. We call it the {\it attaching element}, by analogy of the corresponding elements \eqref{eq_xyz} in the main body. 

We can identify the case $\cG(n)= O(1)$ with a familiar sequence connecting $\KO$ and $\KU$, as follows.  

\begin{prop}[{stabilzation-restriction fiber sequence for $O(1)$-equivariant $\KO$}]\label{prop_stabres_KO_O(1)}
	\begin{enumerate}
		\item We have $[B\Z/2, P^2BO] \simeq \Z/4$, and the class $[\overline{V}_{O(1)}] \in [B\Z/2, P^2 BO]$ of the fundamental representation represents a generator. In particular, we have
		\begin{align}
			\KO[(k+4)V_{O(1)}]^{O(1)} \simeq \KO[kV_{O(1)} + 4]^{O(1)}. 
		\end{align}
		\item Let us use the identification
		$$\KO^{O(1)} \simeq \KO \oplus \KO$$ corresponding to the decomposition $\pi_0 \KO^{O(1)} \simeq \RO(O(1)) = \Z[\underline{\R}] \oplus \Z[{V}_{O(1)}]$. 
	The diagram \eqref{seq_building_KO_general} of $\KO$-modules for $n=1$ and $\cG = O$ is identified as follows. 
	\begin{align}\label{seq_building_KO_O(1)}
			\scalebox{0.9}{
			\xymatrix{
		\KO \oplus \KO \ar[d]^-{\simeq} \ar[r]^-{\id \oplus ( -\id )}& \KO \ar[d]^-{\simeq} \ar[r]^-{\beta \circ c} & \KU[2] \ar[d]^-{\simeq} \ar[r]^-{R\circ \beta} & \KO[4] \ar[d]^-{\simeq} \ar[r]^{ \id \oplus ( -\id)} & \KO[4] \oplus \KO[4]\ar[d]^-{\simeq}\\
	\KO^{O(1)}	 \ar[r]^-{\stab} \ar[d]^-{\res}   & \KO[V_{O(1)}]^{O(1)}  \ar[r]^-{\stab} \ar[d]^-{\res} & \KO[2V_{O(1)}]^{O(1)}   \ar[d]^-{\res} \ar[r]^-{\stab} &  \KO[3V_{O(1)}]^{O(1)}   \ar[d]^-{\res} \ar[r]^-{\stab} & \KO[4V_{O(1)}]^{O(1)} \simeq \KO[4]^{O(1)}  \ar[d]^-{\res} \ar[d]^-{\res}\\
					\KO & \KO[1] & \KO[2]  & \KO[3]  & \KO[4]
			}
		}
	\end{align}
	Here, $c \colon \KO \to \KU$ is the complexification, $\beta \colon \KU \simeq \KU[2]$ is the Bott periodicity, $R \colon \KU \to \KO$ is the realification. 
	\item For each integer $m \in \Z_{\ge 0}$, attaching elements \eqref{eq_w} are identified as follows. 
	\begin{align}
		y(4m) &= 0 \in \pi_{-1}\KO[-V_{O(1)}]^{O(1)} \simeq \pi_{-1}\KO, \\
		y(4m+1) &= (0, 1) \in \pi_0 \KO^{O(1)} \simeq \pi_0 \KO \oplus \pi_0\KO, \\
		y(4m+2) &= \eta \in \pi_1 \KO[V_{O(1)}]^{O(1)} \simeq \pi_1 \KO, \\
		y(4m+3) &= 1 \in \pi_2 \KO[2V_{O(1)}] \simeq \pi_0 \KU. 
	\end{align}
	\end{enumerate}
\end{prop}
\begin{proof}
	(1) is a classical result which is not difficult to check directly. So we omit the detail here. Let us prove (2). 
	First, let us consider the rightmost square of \eqref{seq_building_KO_O(1)}. 
	The restriction map gives
	\begin{align}
	\res=	\id \oplus \id \colon \KO^{O(1)} \simeq \KO \oplus \KO \to \KO,
	\end{align}
	so by Proposition \ref{prop_stab_res_KO} we get the isomorphism $\KO[4] \simeq \KO[3V_{O(1)}]^{O(1)}$ and the commutativity of the rightmost square of \eqref{seq_building_KO_O(1)}. 
	
	For the leftmost square, we use the self-duality of stabilization-restriction fiber sequence. It is the $\KO$-version of Proposition \ref{prop_selfdual_stabres}. There we have used self-duality result of equivariant $\TMF$ in Fact \ref{fact_dualizability_TMF}, but since we are dealing with the finite group $O(1)$, the $O(1)$-equivariant $\KO$-theory is also self-dual and the analogous statement holds. In particular, we get the fiber sequence
	\begin{align}
		\KO \xrightarrow{\tr_{e}^{O(1)}} \KO^{O(1)} \xrightarrow{\stab} \KO[V_{O(1)}]^{O(1)}. 
	\end{align}
	We know that, under the identification $\KO^{O(1)} \simeq \KO \oplus \KO$ as above, the transfer map is identified as $\id \oplus \id \colon \KO \to \KO \oplus \KO$. 
	This gives an isomorphism $\KO \simeq \KO[V_{O(1)}]^{O(1)}$ with the commutativity of the leftmost square of \eqref{seq_building_KO_O(1)}. 
	
	Now let us prove the commutativity of the middle-left square of \eqref{seq_building_KO_O(1)}. In order for this, we use the model of twisted equivariant $\KO$-theory in terms of twisted group algebras. See \cite{Gomi} for details. 
	In general, for a discrete group $G$, an element $\omega \in [BG, P^2BO]$ defines a {\it $\Z/2$-graded twisted group algebra} $\R_\omega[G]$. The corresponding $\omega$-twisted $G$-equivariant $\KO$-spectrum is realized by the space of Fredholm operators on $\Z/2$-graded Hilbert spaces with an action of $\R_\omega[G]$. 
	In the case where the element $\omega$ lifts to an element $\omega \in H^2(BG; \Z/2)$, the algebra $\R_\omega[G]$ has the trivial $\Z/2$-grading, and explicitly given by twisting the multiplication by $\omega$ as $g \cdot h = \omega(g, h) gh$, where we are regarding $\omega$ as a $\pm 1$-valued group $2$-cocycle on $G$. 
	
	In our case, the middle twist $\omega := [2\overline{V}_{O(1)}] \in [BO(1), P^2BO]$ is the image of the nontrivial generator of $H^2(BO(1), \Z/2)$. We immediately see that we actually have an isomorphism $\R_{\omega}[O(1)] \simeq \C$ of algebras over $\R$. 
	 Thus we get the identification $\KO[2\overline{V}_{O(1)}]^{O(1)} \simeq \KU$, together with the commutative diagram
	 \begin{align}\label{eq_prf_building_KO_O(1)}
	 	\xymatrix{
	 	\KO[2\overline{V}_{O(1)}]^{O(1)}  \ar[r] \ar[d]^-{\res_{O(1)}^e}& \KU \ar[d]^-{R} \\
	 	\KO \ar@{=}[r] & \KO
	 	}
	 \end{align}
	 Now we invoke of the classical fact that the following is a fiber sequence (e.g., \cite{bruner2012postnikov}), 
	  \begin{align}\label{eq_eta_c_R}
	 \KO[1] \xrightarrow{\cdot \eta} \KO \xrightarrow{\beta \circ c} \KU[2] \xrightarrow{R} \KO[2]. 
	 \end{align}
Combining this fiber sequence and commutativity of \eqref{eq_prf_building_KO_O(1)} and Proposition \ref{prop_stab_res_KO} gives the commutativity of the middle-left square of \eqref{seq_building_KO_O(1)}. 

For the remaining middle-right square in \eqref{seq_building_KO_O(1)}, we again use the self-duality argument. We know that the two stabilization maps neighbouring $\KO[2V_{O(1)}]^{O(1)}$ in \eqref{seq_building_KO_O(1)} are $\KO$-linear dual to each other, up to degree shift by $4$. On the other hand, we observe that the fiber sequence \eqref{eq_eta_c_R} is also self-dual in $\Mod_\KO$, where $\beta \circ c$ is identified as the dual to $R$. This means that the commutativity of the middle-right square in \eqref{seq_building_KO_O(1)} follows from that of the middle-left square which we have already proved. This completes the proof of Proposition \ref{prop_stabres_KO_O(1)} (2). 

(3) follows directly from the analysis so far in the proof of (2). $y(0)$ and $y(1)$ are obvious. The identifications of $y(2)$ follow from the fiber sequence \eqref{eq_eta_c_R}. This finishes the proof of Proposition \ref{prop_stabres_KO_O(1)}. 
\end{proof}

Going back to the general situation, the compatibility of the topological $\Gm$-genera and the internal structure maps is stated as follows. 
\begin{prop}[{Compatibility of $\JacKO$ and internal structure maps}]\label{prop_compatibility_internal_JacKO}
		Let $(\cG, \cH)$ be either one of $(U, U)$ and $ (O, SO)$. 
	The following diagram commutes. 
	\begin{align}
		\xymatrix{
			MT(\cH(k-1),n\overline{V}_{\cH(k-1)} )  \ar[d]^-{\stab} \ar[rrr]^-{\JacKO_{\cG(n)_{k-1}}} & &&\KO[(k-1)V_{\cG(n)}]^{\cG(n)} \ar[d]^-{\chi(V_{\cG(n)}) \cdot}_-{\stab} \\
			MT(\cH(k),n\overline{V}_{\cH(k)} )   \ar[d]^-{\chi(V_{\cH(k)}) \cdot}_-{\res} \ar[rrr]^-{\JacKO_{\cG(n)_k}} &&& \KO[kV_{\cG(n)}]^{\cG(n)} \ar[d]^-{\res_{\cG(n)}^{\cG(n-1)}} \\
			MT(\cH(k),(n-1)\overline{V}_{\cH(k)} ) [Nk] \ar[rrr]^-{\JacKO_{\cG(n-1)_k}} & &&\KO[kV_{\cG(n-1)}+Nk]^{\cG(n-1)}
		}
	\end{align}
\end{prop}
The proof is analogous to the corresponding Proposition \ref{prop_Jac_different_nk}. 

As a corollary, we get the statement corresponding to Corollary \ref{cor_Jac_res_Euler}. In particular, we get the following relation with the Euler numbers: 

\begin{cor}[{The restriction of $\JacKO_{G(1)}$ is the Euler number}]\label{cor_Jac_res_Euler_KO}
	Let $(\cG, \cH)$ be either one of $(U, SU)$ and $(O, SO)$. Correspondingly we set $N = 2, 1$, respectively. For any closed manifold $M$ with a {\it strict} tangential $\cH(k)$-structure $\psi$ (Definition \ref{def_strict_str}---so that in particular $\dim_\R M = Nk$), the composition
	\begin{align}
		\Omega_{Nk}^{\cH(k)} \stackrel{\PT}{\simeq} \pi_{Nk}MT\cH(k) \xrightarrow{\JacKO_{G(1)_k}}\KO[kV_{G(1)}]^{G(1)}
		\xrightarrow{\res_{G(1)}^e}  \pi_0 \KO. 
	\end{align}
	sends the class $	[M, \psi] \in  \Omega_{Nk}^{\cH(k)}$ to the Euler number $\Eul(M) \in \Z = \pi_0S \xhookrightarrow{u} \pi_0\KO$. 
\end{cor}

\subsection{Application: Divisibility constraints of Euler numbers for oriented manifolds}
In the main body of this paper, we derive interesting divisibility results of Euler numbers out of topological elliptic genera. 
Now that we got the relation between $\JacKO$ and Euler numbers in Corollary \ref{cor_Jac_res_Euler_KO}, we can think about a similar application to derive the divisibility of Euler numbers. 
We see below that this indeed gives a neat divisibility result, still provable by another elementary method. Here let us focus on the case of $O(1)$-topological $\Gm$-genera. 

Let us introduce the following notation. 
\begin{defn}\label{def_divisibility_KO}
	For each positive integer $k$, define $d^\KO_{SO}(k)$ to be the order of the element $w(k) \in \pi_{k-1}\KO[(k-1)V_{O(1)}]^{O(1)}$ in \eqref{eq_w}.
\end{defn}
Here, by Proposition \ref{prop_stabres_KO_O(1)}, we explicitly know
\begin{align}\label{eq_dKO_explicit}
	d^\KO_{SO}(k) = \begin{cases}
	1 & k \equiv 0 \pmod 4, \\
		\infty & k \equiv 1, 3 \pmod 4, \\
		2 & k \equiv 2 \pmod 4. 
	\end{cases}
\end{align}
The divisibility argument is based on the observation that, by the long exact sequence associated to the stabilization-restriction fiber sequence \eqref{eq_w}, we have
\begin{align}\label{eq_dKO=im}
	d^\KO_{SO}(k) \cdot \Z =\mathrm{im}	\left( \res_{O(1)}^e \colon \pi_{k}\KO[kV_{O(1)}] \to  \pi_0 \KO \simeq \Z \right) 
\end{align}
On the other hand, by Corollary \ref{cor_Jac_res_Euler_KO}, we know that, for any oriented closed manifold $M$ with $\dim M = k$, the Euler number $\Eul(M)$ is contained in the right hand side of \eqref{eq_dKO=im}. This implies that we have
\begin{align}
\left.	d^\KO_{SO}(k)  \,\right|\, \Eul(M). 
\end{align}
Combining \eqref{eq_dKO_explicit}, we deduce

\begin{prop}[{A divisibility constraint of Euler numbers from $\JacKO$}]\label{prop_divisibility_JacKO}
	For any oriented closed manifold $M$ with dimension $\dim M \equiv 2 $ (mod $4$), we have
	\begin{align}
		2 \,|\, \Eul(M). 
	\end{align}
\end{prop}

This result itself can be proved directly, as follows. We have
\begin{align}
	\Eul(M) \equiv \sum_{i=0}^{\dim M} \dim_\R H^i(M; \R) \pmod 2. 
\end{align}
If $M$ is oriented Riemannian and of dimension $2$ mod $4$, the intersection pairing gives a skew-symmetric nondegenerate pairing on $\oplus_i H^i(M; \R)$. Equivalently, we have a complex structure on this vector space. This means that the total dimension should be even, proving that $\Eul(M) = 0$. 

Actually, this direct proof is essentially related to our proof using $O(1)$-topological $\Gm$-genera. Namely, the right hand side of \eqref{eq_dKO=im} for $k \equiv 2$ (mod $4$) is identified with
\begin{align}
	\mathrm{im} \left( R \colon \pi_0 \KU \to \pi_0 \KO \right)
\end{align}
by our analysis in the proof of Proposition \ref{prop_stabres_KO_O(1)}. This is $2\Z$ because complex vector spaces have even real dimension. 

\def\arxivconsists of font{\rm}
\bibliographystyle{ytamsalpha}
\bibliography{ref}

\providecommand{\bysame}{\leavevmode\hbox to3em{\hrulefill}\thinspace}
\providecommand{\MR}{\relax\ifhmode\unskip\space\fi MR }
\providecommand{\MRhref}[2]{%
  \href{http://www.ams.org/mathscinet-getitem?mr=#1}{#2}
}
\providecommand{\href}[2]{#2}
\providecommand{\doihref}[2]{\href{#1}{#2}}
\providecommand{\arxivfont}{\tt}
\begin{thebibliography}{GKMP25}

\bibitem[ABG10]{ando2010twists}
M.~Ando, A.~J. Blumberg, and D.~Gepner, \emph{Twists of {K}-theory and {TMF}}, Superstrings, geometry, topology, and C∗-algebras \textbf{81} (2010) 27--63.

\bibitem[ABG18]{ABGThomSpectra}
M.~Ando, A.~J. Blumberg, and D.~Gepner, \emph{Parametrized spectra, multiplicative {T}hom spectra and the twisted {U}mkehr map}, \href{https://doi.org/10.2140/gt.2018.22.3761}{Geom. Topol. \textbf{22} (2018) 3761--3825}.

\bibitem[AFG08]{Ando_2008}
M.~Ando, C.~P. French, and N.~Ganter, \emph{The {J}acobi orientation and the two-variable elliptic genus}, \href{http://dx.doi.org/10.2140/agt.2008.8.493}{Algebraic \& Geometric Topology \textbf{8} (2008) 493--539}.

\bibitem[AG07]{ando2010circleequivariant}
M.~Ando and J.~P.~C. Greenlees, \emph{Circle-equivariant classifying spaces and the rational equivariant sigma genus}, 2010. \href{http://arxiv.org/abs/0705.2687}{{\arxivfont arXiv:0705.2687 [math.AT]}}.

\bibitem[AHR10]{ando2010multiplicative}
M.~Ando, M.~J. Hopkins, and C.~Rezk, \emph{Multiplicative orientations of {KO}-theory and of the spectrum of topological modular forms}, \href{https://rezk.web.illinois.edu/koandtmf.pdf}{preprint (2010) }.

\bibitem[Ale72]{AlexanderIndecomposable}
J.~C. Alexander, \emph{A family of indecomposable symplectic manifolds}, \href{https://doi.org/10.2307/2373752}{Amer. J. Math. \textbf{94} (1972) 699--710}.

\bibitem[Ale77]{alexander1977correction}
J.~C. Alexander, \emph{Correction and addendum to a family of indecomposable symplectic manifolds}, American Journal of Mathematics \textbf{99} (1977) 1361--1364.

\bibitem[Bau]{BauerComputationTJF}
T.~Bauer, \emph{Elliptic cohomology of projective spaces and the divisibility of taylor coefficients of jacobi forms}. To appear.

\bibitem[BM25]{BauerMeierTJF}
T.~Bauer and L.~Meier, \emph{Topological jacobi forms}, arXiv preprint arXiv:2508.08010 (2025) .

\bibitem[Bru12]{bruner2012postnikov}
R.~R. Bruner, \emph{On the {P}ostnikov towers for real and complex connective {K}-theory}, 2012. \href{http://arxiv.org/abs/1208.2232}{{\arxivfont arXiv:1208.2232 [math.AT]}}.

\bibitem[BYa]{BauerYamashitaTEJF}
T.~Bauer and M.~Yamashita, \emph{The ring of {Topological Even Jacobi} forms and the surjectivity of {$Sp(1)$-Topological Elliptic Genera}}. To appear.

\bibitem[BYb]{BauerYamashitaTJF}
\bysame, \emph{{The $U(1)$-topological elliptic genus is surjective}}. To appear.

\bibitem[Dev22]{devalapurkar:hodge}
S.~K. Devalapurkar, \emph{Hodge theory for elliptic curves and the hopf element {\^i}½$\nu$}, \href{http://dx.doi.org/10.1112/blms.12759}{Bulletin of the London Mathematical Society \textbf{55} (2022) 826--842}.

\bibitem[DFHH14]{TMFbook}
C.~L. Douglas, J.~Francis, A.~G. Henriques, and M.~A. Hill (eds.), \doihref{http://dx.doi.org/10.1090/surv/201}{\emph{Topological modular forms}}, Mathematical Surveys and Monographs, vol. 201, American Mathematical Society, Providence, RI, 2014. \url{https://doi.org/10.1090/surv/201}.

\bibitem[DMZ12]{dabholkar2014quantum}
A.~Dabholkar, S.~Murthy, and D.~Zagier, \emph{Quantum black holes, wall crossing, and mock modular forms}, 2014. \href{http://arxiv.org/abs/1208.4074}{{\arxivfont arXiv:1208.4074 [hep-th]}}.

\bibitem[Eno90]{enoki1990compact}
I.~Enoki, \emph{Compact ricci-flat k{\"a}hler manifolds}, K{\"a}hler Metric and Moduli Spaces, vol.~18, Mathematical Society of Japan, 1990, pp.~229--257.

\bibitem[Fre06]{frenkel2006representations}
I.~B. Frenkel, \emph{Representations of affine {L}ie algebras, {H}ecke modular forms and {K}orteweg—de {V}ries type equations}, Lie Algebras and Related Topics: Proceedings of a Conference Held at New Brunswick, New Jersey, May 29--31, 1981, Springer, 2006, pp.~71--110.

\bibitem[Fre19]{FreedLecture}
D.~S. Freed, \doihref{http://dx.doi.org/10.1090/cbms/133}{\emph{Lectures on field theory and topology}}, CBMS Regional Conference Series in Mathematics, vol. 133, American Mathematical Society, Providence, RI, 2019. \url{https://doi.org/10.1090/cbms/133}. Published for the Conference Board of the Mathematical Sciences.

\bibitem[GKMP25]{gukov2025newapproach31dimensionaltqfts}
S.~Gukov, V.~Krushkal, L.~Meier, and D.~Pei, \emph{A new approach to (3+1)-dimensional tqfts via topological modular forms}, 2025. \href{http://arxiv.org/abs/2509.12402}{{\arxivfont arXiv:2509.12402 [math.AT]}}. \url{https://arxiv.org/abs/2509.12402}.

\bibitem[GM]{GepnerMeierNEW}
D.~Gepner and L.~Meier, \emph{Equivariant elliptic cohomology with integral coefficients}. To appear.

\bibitem[GM23]{GepnerMeier}
D.~Gepner and L.~Meier, \emph{On equivariant topological modular forms}, \href{https://doi.org/10.1112/s0010437x23007509}{Compos. Math. \textbf{159} (2023) 2638--2693}.

\bibitem[Gom23]{Gomi}
K.~Gomi, \emph{Freed-{M}oore {$K$}-theory}, Comm. Anal. Geom. \textbf{31} (2023) 979--1067.

\bibitem[Gri99]{Gritsenko:1999fk}
V.~Gritsenko, \emph{{Elliptic genus of Calabi-Yau manifolds and Jacobi and Siegel modular forms}}, \href{http://arxiv.org/abs/math/9906190}{{\arxivfont arXiv:math/9906190}}.

\bibitem[Gri20]{gritsenko2020modified}
V.~Gritsenko, \emph{Modified elliptic genus}, \doihref{http://dx.doi.org/https://doi.org/10.1007/978-3-030-42400-8_2}{Partition Functions and Automorphic Forms (2020) 87--119}.

\bibitem[GW20]{Gristenko2020382}
V.~Gritsenko and H.~Wang, \emph{Graded rings of integral jacobi forms}, \href{https://www.sciencedirect.com/science/article/pii/S0022314X20300925}{Journal of Number Theory \textbf{214} (2020) 382--398}.

\bibitem[Hil20]{hill2020equivariant}
M.~A. Hill, \emph{Equivariant stable homotopy theory}, Handbook of Homotopy Theory, Chapman and Hall/CRC, 2020, pp.~699--756.

\bibitem[HL13]{HillLawson}
M.~Hill and T.~Lawson, \emph{Topological modular forms with level structure}, \href{https://doi.org/10.1007/s00222-015-0589-5}{Invent. Math. \textbf{203} (2016) 359--416}, \href{http://arxiv.org/abs/1312.7394}{{\arxivfont arXiv:1312.7394 [math.AT]}}.

\bibitem[HS16]{Hsin:2016blu}
P.-S. Hsin and N.~Seiberg, \emph{Level/rank duality and {C}hern-{S}imons-matter theories}, \doihref{http://dx.doi.org/10.1007/JHEP09(2016)095}{JHEP \textbf{09} (2016) 095}, \href{http://arxiv.org/abs/1607.07457}{{\arxivfont arXiv:1607.07457 [hep-th]}}.

\bibitem[Lan67]{LandweberOperation}
P.~S. Landweber, \emph{Cobordism operations and {H}opf algebras}, \href{https://doi.org/10.2307/1994365}{Trans. Amer. Math. Soc. \textbf{129} (1967) 94--110}.

\bibitem[LTY]{LinTominagaYamashita}
Y.-H. Lin, A.~Tominaga, and M.~Yamashita, \emph{On genuinely {$C_n$}-equivariant {$\TMF$}}. To appear.

\bibitem[Lura]{LurieElliptic1}
J.~Lurie, \emph{Elliptic cohomology {I}: Spectral abelian varieties}. \url{https://www.math.ias.edu/~lurie/papers/Elliptic-I.pdf}.

\bibitem[Lurb]{LurieElliptic2}
\bysame, \emph{Elliptic cohomology {II}: Orientations}. \url{https://www.math.ias.edu/~lurie/papers/Elliptic-II.pdf}.

\bibitem[Lurc]{LurieElliptic3}
\bysame, \emph{Elliptic cohomology {III}: Tempered cohomology}. \url{https://www.math.ias.edu/~lurie/papers/Elliptic-III-Tempered.pdf}.

\bibitem[Lur09]{lurie2009survey}
\bysame, \emph{A survey of elliptic cohomology}, Algebraic Topology: The Abel Symposium 2007, Springer, 2009, pp.~219--277.

\bibitem[LY]{LinYamashita2}
Y.-H. Lin and M.~Yamashita, \emph{Topological elliptic genera {II}: Physical interpretations}. To appear.

\bibitem[MNRS91]{mlawer1991group}
E.~J. Mlawer, S.~G. Naculich, H.~A. Riggs, and H.~J. Schnitzer, \emph{Group-level duality of {WZW} fusion coefficients and {C}hern-{S}imons link observables}, Nuclear Physics B \textbf{352} (1991) 863--896.

\bibitem[Mos68]{Mosher}
R.~E. Mosher, \emph{Some stable homotopy of complex projective space}, \href{https://doi.org/10.1016/0040-9383(68)90026-8}{Topology \textbf{7} (1968) 179--193}.

\bibitem[Muk84]{Mukai}
J.~Mukai, \emph{The order of the attaching class in the suspended quaternionic quasiprojective space}, \href{https://doi.org/10.2977/prims/1195181109}{Publ. Res. Inst. Math. Sci. \textbf{20} (1984) 717--725}.

\bibitem[MY]{MeierYamashita}
L.~Meier and M.~Yamashita, \emph{Looijenga line bundles in equivariant twisted {$\TMF$}}. To appear.

\bibitem[Nov67]{NovikovCobordism}
S.~P. Novikov, \emph{Methods of algebraic topology from the point of view of cobordism theory}, Izv. Akad. Nauk SSSR Ser. Mat. \textbf{31} (1967) 855--951.

\bibitem[NRS90]{naculich1990group}
S.~G. Naculich, H.~Riggs, and H.~Schnitzer, \emph{Group-level duality in {WZW} models and {C}hern-{S}imons theory}, Physics Letters B \textbf{246} (1990) 417--422.

\bibitem[NS07]{naculich2007level}
S.~G. Naculich and H.~J. Schnitzer, \emph{Level-rank duality of the {U$(N)$} {WZW} model, {C}hern-{S}imons theory, and 2d q{YM} theory}, Journal of High Energy Physics \textbf{2007} (2007) 023.

\bibitem[NT92]{nakanishi1992level}
T.~Nakanishi and A.~Tsuchiya, \emph{Level-rank duality of {WZW} models in conformal field theory}, Communications in mathematical physics \textbf{144} (1992) 351--372.

\bibitem[Och91]{Ochanine}
S.~Ochanine, \emph{Elliptic genera, modular forms over {$K{\rm O}_*$} and the {B}rown-{K}ervaire invariant}, \href{https://doi.org/10.1007/BF02571343}{Math. Z. \textbf{206} (1991) 277--291}.

\bibitem[ORS20]{ostrik2020symplectic}
V.~Ostrik, E.~C. Rowell, and M.~Sun, \emph{Symplectic level-rank duality via tensor categories}, Journal of Lie Theory \textbf{30} (2020) 909--924.

\bibitem[OS14]{ostrik2014}
V.~Ostrik and M.~Sun, \emph{Level-rank duality via tensor categories}, \href{https://doi.org/10.1007/s00220-013-1869-9}{Communications in Mathematical Physics \textbf{326} (2014) 49--61}.

\bibitem[Ray72]{ray1972symplectic}
N.~Ray, \emph{The symplectic bordism ring}, Mathematical Proceedings of the Cambridge Philosophical Society, vol.~71, Cambridge University Press, 1972, pp.~271--282.

\bibitem[Tom]{Tominaga}
A.~Tominaga, \emph{Computation of topological {J}acobi forms}. To appear.

\bibitem[Tot00]{totaro2000}
B.~Totaro, \emph{Chern numbers for singular varieties and elliptic homology}, \href{http://www.jstor.org/stable/121047}{Annals of Mathematics \textbf{151} (2000) 757--791}.

\bibitem[TY23]{tachikawa2023anderson}
Y.~Tachikawa and M.~Yamashita, \emph{Anderson self-duality of topological modular forms, its differential-geometric manifestations, and vertex operator algebras}, \href{https://arxiv.org/abs/2305.06196}{arXiv preprint arXiv:2305.06196 (2023) }, \href{http://arxiv.org/abs/2305.06196}{{\arxivfont arXiv:2305.06196 [math.AT]}}.

\bibitem[Wak58]{wakakuwa1958riemannian}
H.~Wakakuwa, \emph{{On Riemannian manifolds with homogeneous holonomy group $Sp (n)$}}, Tohoku Mathematical Journal, Second Series \textbf{10} (1958) 274--303.

\bibitem[Wit88]{Witten:1987cg}
E.~Witten, \emph{The index of the {D}irac operator in loop space}, \doihref{http://dx.doi.org/10.1007/BFb0078045}{Lect. Notes Math. \textbf{1326} (1988) 161--181}.

\end{thebibliography}

\end{document}